\documentclass{amsart}
\usepackage[english]{babel}
\usepackage{amsmath}
\usepackage{amssymb, amscd, amsthm}
\usepackage{mathrsfs}
\usepackage{cases}
\usepackage{multicol}
\usepackage{enumerate}
\usepackage{stackrel}
\usepackage[all]{xy}
\usepackage[left=3.5cm,right=3.5cm, top=2cm, bottom=3cm,includehead]{geometry}
\usepackage{stmaryrd}
\usepackage{graphicx}
\usepackage{url}
\usepackage{verbatim}
\usepackage{hyperref}
\hypersetup{
  linktoc=page,
}
\newtheorem{satz}{Satz}[section]
\newtheorem{Theorem}[satz]{Theorem}
\newtheorem{Cor}[satz]{Corollary}
\newtheorem{Lemma}[satz]{Lemma}
\newtheorem{Prop}[satz]{Proposition}

\newtheorem*{thmx}{Theorem}
\theoremstyle{definition}
\newtheorem{Def}[satz]{Definition}

\newtheorem{Remark}[satz]{Remark}
\newtheorem*{Remarkx}{Remark}
\newtheorem{Example}[satz]{Example}
\newtheorem*{Def*}{Definition}
\newtheorem*{Example*}{Example}

\newcommand{\R}{\mathbb{R}}

\newcommand{\upi}{\underline{\pi}}
\newcommand{\N}{\mathbb{N}}
\newcommand{\F}{\mathcal{F}}
\newcommand{\U}{\mathcal{U}}
\newcommand{\xr}{\xrightarrow}
\newcommand{\T}{\mathbf{T}}
\newcommand{\map}{\operatorname{map}}
\newcommand{\Ho}{\operatorname{Ho}}
\newcommand{\conn}{\operatorname{conn}}
\newcommand{\M}{\mathbf{M}^{\U}}
\newcommand{\Hom}{\operatorname{Hom}}

\DeclareMathOperator*{\colim}{colim}

\DeclareMathOperator*{\res}{res}
\DeclareMathOperator*{\tr}{tr}
\DeclareMathOperator*{\Inj}{Inj}
\DeclareMathOperator*{\Bij}{Bij}

\DeclareMathOperator*{\Lin}{L}

\DeclareMathOperator*{\Iso}{Iso}
\DeclareMathOperator*{\Fun}{Fun}
\DeclareMathOperator*{\Cyl}{Cyl}
\DeclareMathOperator*{\im}{im}

\numberwithin{equation}{section}

\makeatletter
\newcommand*{\defeq}{\mathrel{\rlap{%
                     \raisebox{0.3ex}{$\m@th\cdot$}}%
                     \raisebox{-0.3ex}{$\m@th\cdot$}}%
                     =}
\makeatother

\title{$G$-symmetric spectra, semistability and the multiplicative norm}
\author{Markus Hausmann}
\date{}

\begin{document}

\begin{abstract}
In this paper we develop the basic homotopy theory of $G$-symmetric spectra (that is, symmetric spectra with a $G$-action) for a finite group $G$, as a model for equivariant stable homotopy with respect to a $G$-set universe. This model lies in between Mandell's equivariant symmetric spectra and the $G$-orthogonal spectra of Mandell and May and is Quillen equivalent to the two. We further discuss equivariant semistability, construct model structures on module, algebra and commutative algebra categories and describe the homotopical properties of the multiplicative norm in this context.
\end{abstract}

\maketitle

\setcounter{section}{-1}
\section{Introduction}

Stable equivariant homotopy theory has seen various applications both in equivariant and non-equivariant topology. Many of these applications make use of constructions which require or are simplified by a highly structured model of equivariant spectra with a point-set level smash product. For example, model structures on module or (commutative) algebra categories allow one to employ algebraic techniques in order to obtain new equivariant spectra with desired properties, fixed points of commutative equivariant ring spectra form non-equivariant $E_{\infty}$-ring spectra and their homotopy groups inherit the structure of a Tambara functor. Another example is the construction of the spectrum level multiplicative norm introduced in \cite{HHR16}, which plays a central role in the solution of the Kervaire invariant one problem.

The first highly structured models for stable equivariant homotopy theory were the $S_G$-modules and $G$-orthogonal spectra of \cite{MM02} and the equivariant symmetric spectra of \cite{Man04}. While the first two can be formed over arbitrary compact Lie groups, the latter only works for finite groups. In turn it has the advantage that it can also be based on simplicial sets, whereas the other two need topological spaces.

In this paper we develop a further model we call $G$-symmetric spectra. It also requires the group to be finite and can be based both on simplicial sets and topological spaces. Conceptually, it lies in between Mandell's equivariant symmetric spectra and $G$-orthogonal spectra. In order to describe the difference, we have to recall the definitions. $G$-orthogonal spectra are formed with respect to a $G$-representation universe which determines what kinds of representation spheres $S^V$ become invertible in the homotopy category. Then, roughly speaking, a $G$-orthogonal spectrum $X$ consists of a family of based $G$-spaces $X(V)$ for every finite dimensional subrepresentation $V$ of the chosen universe, together with structure maps of the form $X(V)\wedge S^W\to X(V\oplus W)$. In addition, every $X(V)$ possesses an action of the orthogonal group $O(V)$ which is suitably compatible with the 
$G$-action and the structure maps. $G$-orthogonal spectra formed with respect to different universes are connected by so-called change of universe functors. It was noticed in \cite[V, Thm. 1.5]{MM02} that on the point-set level these are always equivalences of categories (though the derived functors are not). In other words, the underlying category does not really depend on the chosen $G$-representation universe, while the weak equivalences do. Furthermore, every such category of $G$-orthogonal spectra is equivalent to the category of $G$-objects in ordinary non-equivariant orthogonal spectra. In short, the reason for this phenomenon is that the values $X(V)$ can be reconstructed from those on trivial representations together with the orthogonal group actions. While this observation can be misleading in terms of the homotopy theory, it is very useful from a categorical viewpoint. For example, it is now clear how to define restriction to subgroups, induction, fixed point and orbit functors on the point set 
level and it only remains to examine with respect to which model structures (or weak equivalences) they can be derived, in order to obtain corresponding functors between the homotopy categories. A further example for such a functor is the multiplicative norm of \cite{HHR16} mentioned earlier.

%\todo{shorten this paragraph}
%Mandell`s equivariant symmetric spectra are ... in the sense of Hovey (REF) ... with respect to the $G$-sphere $S^{G/N}$, where $N$ is a fixed normal subgroup of $G$. Concretly, such a spectrum consists of ...

For Mandell's equivariant symmetric spectra (an equivariant version of the symmetric spectra of \cite{HSS00}) the situation is somewhat different. There, a normal subgroup $N$ of $G$ is fixed and a $G\Sigma_{G/N}$-spectrum is defined as a family of based $(G\times \Sigma_n)$-simplicial sets $X(n\times G/N)$ for all natural numbers $n$, connected by structure maps $X(n\times G/N)\wedge S^{G/N}\to X((n+1)\times G/N)$.
The category of such spectra models the stable theory with respect to the universe of $G$-representations on which $N$ acts trivially (in particular the case $N=1$ models the genuine theory). 
At first sight this looks similar to $G$-orthogonal spectra, replacing $G$-representations by $G$-sets and orthogonal groups by symmetric groups. But if $N$ is a proper subgroup of $G$, then $X(n\times G/N)$ does not have an action of the full 
symmetric group of $n\times G/N$, but only of the subgroup of block permutations of the copies of $G/N$. Furthermore, it does not contain evaluations at trivial $G$-sets. Taken together, this results in the categories for varying $N$ being pairwise non-equivalent, and only if $N$ is equal to $G$ the category of $G\Sigma_{G/N}$-spectra is equivalent to the category of $G$-objects in non-equivariant symmetric spectra. As a consequence, to quote \cite{Man04}, ``the construction of these [change of universe, change of groups, fixed point and orbit] functors and their relationship is significantly more complicated for equivariant symmetric spectra than it is for equivariant orthogonal spectra.'' The norm is not discussed in \cite{Man04}.

Another simplicial set model for the $G$-equivariant stable homotopy category was introduced by Dundas, {\O}stv{\ae}r and R{\"o}ndigs in \cite{DOR03}. It is given by the $G$-enriched functor category from finite $G$-simplicial sets to all $G$-simplicial sets, an equivariant generalization of Lydakis' model for the stable homotopy category \cite{Lyd98}. By evaluating at spheres, every such spectrum gives rise to a $G$-symmetric spectrum in our sense below. The paper \cite{DOR03} does not contain a treatment of the norm and genuine commutative multiplications.

\subsection*{$G$-symmetric spectra}

We now come to the content of this paper. We consider the category of $G$-symmetric spectra, which are the direct symmetric analog of $G$-orthogonal spectra. The category is always that of $G$-objects in the non-equivariant symmetric spectra introduced in \cite{HSS00}. Parallel to the orthogonal case, these already secretly inherit evaluations at arbitrary finite $G$-sets. Which of these evaluations are declared homotopically meaningful in the model structure is based on a $G$-set universe $\U$, i.e., a countably infinite $G$-set which is isomorphic to the disjoint union of two copies of itself. The $\R$-linearization $\R[\U]$ of such a $G$-set universe is then a $G$-representation universe and we show:
\begin{thmx}[Model structures, Theorems \ref{theo:flatmod}, \ref{theo:projmod}, \ref{theo:gquillen} and \ref{theo:gquillenmandell}] For every finite group $G$ and every $G$-set universe $\U$ there exists a model structure on $G$-symmetric spectra of spaces and simplicial sets, which is Quillen equivalent to the $G$-orthogonal spectra of \cite{MM02} formed with respect to the $G$-representation universe $\R[\U]$. If $\U$ is an infinite disjoint union of orbits $G/N$ for a fixed normal subgroup $N$ of $G$, the model structure is also Quillen equivalent to the $G\Sigma_{G/N}$-spectra of \cite{Man04}.
\end{thmx}
It can happen that two non-isomorphic $G$-set universes become isomorphic when linearized and hence model the same homotopy theory. This leads to various different model structures with 
equivalent homotopy categories. For example, the $N$-fixed $G$-equivariant stable homotopy category can be modeled by a $G$-set universe consisting of infinitely many copies of the $G$-orbit $G/N$, which is the closest approximation to \cite{Man04}, as noted above. But the disjoint union $\U_G(N)$ of infinitely many copies of all orbits $G/H$ with $N\leq H$ leads to an equivalent theory and often has technical advantages, in particular for understanding the homotopical properties of the norm discussed below.

The central notion for these model structures is that of a $G^{\U}$-stable equivalence (Definition \ref{def:gomega}), whose description is more complicated than in the orthogonal case. This is due to the phenomenon of semistability, which we say more about below. Around these $G^{\U}$-stable equivalences we construct two kinds of model structures, a projective one generalizing the stable model structure in \cite{HSS00} and analogous to the one on $G$-orthogonal spectra constructed in \cite{MM02}; and a flat one, which is an equivariant generalization of the one constructed in \cite{Shi04} and the analog of the model structures of \cite{Sto11} on $G$-orthogonal spectra. The flat cofibrations do not depend on $\U$, for symmetric spectra of simplicial sets they are just the morphisms which are non-equivariant flat cofibrations if one forgets the $G$-action. In Section \ref{sec:derived} we explain how these model structures can be used to derive change of universe, change of groups, fixed point and orbit functors.

\subsection*{Multiplicative properties}

In Section \ref{sec:monoidal} we deal with the multiplicative properties of our model structures. We show that they are monoidal (Proposition \ref{prop:stablemonoidal}), that they lift to categories of modules and algebras (Corollary \ref{cor:modmodules}) and that positive versions lift to the categories of commutative algebras (Theorem \ref{theo:modcomm}). For this we use the machinery of \cite{SS00} and \cite{Whi17}. Non-equivariantly, a major step in the proof of the existence of model structures on commutative algebras is that given a positive flat symmetric spectrum $X$, the maps $(E\Sigma_n)_+\wedge_{\Sigma_n} X^{\wedge n}\to X^{\wedge n}/\Sigma_n$, collapsing $E\Sigma_n$ to a point, are stable equivalences for all $n$. This is also the case equivariantly, but the $G$-equivariant generalization of $E\Sigma_n$ which is needed depends on the universe with regard to which the stable equivalences are formed (Proposition \ref{prop:commfree}).
%When working over $\U_G(N)$, it is the universal principal $\Sigma_n$-bundle in the category of $G$-spaces with trivial $N$-action, or equivalently a universal space for the family of subgroups of $G\times \Sigma_n$ which are of the form $\{(h,\varphi(h))\ |\ h\in H\}$ for a subgroup $H$ of $G$ and a group homomorphism $\varphi:H\to \Sigma_n$ with $H\cap N$ in the kernel.
Only if the $G$-set universe is trivial one can use an $E\Sigma_n$ with trivial $G$-action. As remarked in \cite{HHR16}, this was overlooked in \cite{MM02}.

The appearance of the $E\Sigma_n$'s with non-trivial $G$-action requires the study of the homotopical properties of the multiplicative norm $N_H^G$ of \cite{HHR16}, which is also interesting in its own right. While the rest of the paper works with respect to arbitrary $G$-set universes, we here need to restrict to the full $N$-fixed ones with respect to a normal subgroup $N$ of $G$ contained in $H$, denoted $\U_G(N)$, and show:
\begin{thmx}[Homotopical properties of the norm, Theorem \ref{theo:norm2}] \label{thmx:homnorm} The norm functor $N_H^G:HSp^{\Sigma}\to GSp^{\Sigma}$ takes $H^{\U_H(N)}$-stable equivalences between $H$-flat $H$-symmetric spectra to $G^{\U_G(N)}$-stable equivalences of $G$-flat $G$-symmetric spectra.
\end{thmx}
In fact, the projective model structures satisfy an even stronger compatibility property with respect to the norm (Theorem \ref{theo:norm1}), but the compatibility with the flat model structures is particularly relevant for understanding the derived norm on commutative $H$-symmetric ring spectra. This is because an $H$-symmetric spectrum underlying an $H$-flat commutative $H$-symmetric ring spectrum is $H$-flat (this is the equivariant analog of the ``convenience'' property of the flat model structure discussed in \cite{Shi04}), with the projective analog of that statement being false. So the theorem above shows that the derived norm of a commutative $H$-symmetric ring spectrum is equivalent as a $G$-symmetric spectrum to the derived norm of the underlying $H$-symmetric spectrum (Corollary \ref{cor:normcomp}). In \cite{HHR16}, where they work with a projective model structure, this issue needs to be addressed differently (cf. \cite[Sec. B.8]{HHR16}).
\subsection*{Equivariant semistability}

Finally, this paper also contains a discussion of equivariant semistability: In Section \ref{sec:homgroups} we define the equivariant homotopy groups of a $G$-symmetric spectrum, depending on the $G$-set universe $\U$. Just like for $G$-orthogonal spectra, they carry a natural structure of a Mackey functor (with the amount of transfers depending on $\U$, cf. \cite{Lew95}). However, equivariant homotopy groups of symmetric spectra are not as well-behaved, as a stable equivalence does not induce isomorphisms on them in general. In the non-equivariant case, this phenomenon is captured by a natural action of the monoid of injective self-maps of the natural numbers on the homotopy groups of a symmetric spectrum, which detects whether the ``naively defined'' homotopy groups coincide with the derived ones (\cite{Sch08}). For $G$-symmetric spectra a similar action exists by the monoid of equivariant self-injections of the chosen universe $\U$, and we show that it detects equivariant semistabiliy and enjoys similar properties to the non-equivariant one.

\begin{Remarkx} The stable model structures on $G$-symmetric spectra for varying $G$ can be combined to a new model structure on the category of (non-equivariant) symmetric spectra, which models global equivariant homotopy theory for finite groups in the sense of Schwede \cite{Sch18}. This was another motivation for this project and is carried out in \cite{Hau15}.
\end{Remarkx}

\tableofcontents

\section{$G$-spaces and $G$-simplicial sets}
\subsection{Elementary definitions and notation}
\label{sec:unstable}
Let $G$ be a finite group. Throughout this paper a topological space means a compactly generated weak Hausdorff space, so that the category of topological spaces becomes closed symmetric monoidal with respect to the cartesian product. We begin by fixing some notation.

\begin{Def} We write $G\T_*$ for the category of based $G$-spaces, and $G\mathcal{S}_*$ for the category of based $G$-simplicial sets. The geometric realization functor $G\mathcal{S}_*\to G\T_*$ is denoted by $|.|$ and its right adjoint singular complex $G\T_*\to G\mathcal{S}_*$ by $\mathcal{S}$.
%\begin{Def} A based \emph{$G$-space} is a topological space with an action of $G$ through homeomorphisms, together with a chosen basepoint which is fixed under the action of $G$. A \emph{$G$-map} between based $G$-spaces is a continuous based map which commutes with the given $G$-actions. We denote the category of based $G$-spaces together with based $G$-maps by $G\T_*$.
%\end{Def}
%\begin{Def} Likewise, a based \emph{$G$-simplicial set} is a based simplicial set with a $G$-action or equivalently a functor $\Delta^{op}\to G-\text{sets}_*$. The category of based $G$-simplicial sets is denoted $G\mathcal{S}_*$.
%\end{Def}
%\begin{Def}For a subgroup $H$ of $G$ and a based $G$-space/$G$-simplicial set $X$ we denote the based space/simplicial set of $H$-fixed points of $X$ (which are taken degreewise in the simplicial case) by $X^H$. The $H$-orbits $X/H$ are defined to be the quotient by the $H$-action.
%\end{Def}
%In the case where $H$ is a normal subgroup, both $X^H$ and $X/H$ carry an induced action of the quotient group $G/H$.
%\begin{Def}
%The geometric realization functor $G\mathcal{S}_*\to G\T_*$ is denoted by $|.|$ and its right adjoint singular complex $G\T_*\to G\mathcal{S}_*$ by $\mathcal{S}$.
%\end{Def}
It follows from adjointness that $|.|$ commutes with $G$-orbits and $\mathcal{S}$ commutes with $G$-fixed points. Moreover, since $|.|$ commutes with all finite limits, it also preserves $G$-fixed points.
\end{Def}
%\begin{Def} \todo{leave out} The smash product $X\wedge Y$ of two based $G$-spaces/$G$-simplicial sets $X$ and $Y$ is the cofiber of the canonical map $X\vee Y\to X\times Y$. If not specified otherwise, it is equipped with the diagonal $G$-action.
%\end{Def}
%The canonical $G$-homeomorphism $|X\times Y|\cong |X|\times |Y|$ descends to a $G$-homeomorphism $|X\wedge Y|\cong |X|\wedge |Y|$ for any pair of based $G$-simplicial sets $X$ and $Y$,

\begin{Example}[Representation spheres] Given a finite dimensional real $G$-representation $V$, its one-point compactification is denoted $S^V$ and called the \emph{representation sphere} of $V$. For another $G$-representation $W$ there is a canonical $G$-homeomorphism $ S^{V\oplus W}\cong S^V\wedge S^W$ sending a pair $(v,w)$ to $v\wedge w$ and the basepoint to the basepoint.

For a finite $G$-set $M$, we further let $S^M$ stand for the simplicial set given by the $M$-fold smash product of $S^1=\underline{\Delta}^1/\partial \underline{\Delta}^1$, with $G$-action permuting the smash factors. Fixing a based homeomorphism $|S^1|\cong S^1=S^\mathbb{R}$ induces a $G$-homeomorphism $|S^M|= S^{\R[M]}$ for all finite $G$-sets~$M$.
%Let $G$ be a finite group and $V$ a finite dimensional $G$-representation, i.e., a finite dimensional real vector space equipped with a scalar product and a $G$-action by linear isometries. Then the one-point compactification of $V$ with induced $G$-action is called the \emph{representation sphere} of $V$ and denoted $S^V$. It is based at infinity. For another $G$-representation $W$ there is a canonical $G$-homeomorphism $ S^{V\oplus W}\cong S^V\wedge S^W $
%sending a pair $(v,w)$ to $v\wedge w$ and the basepoint to the basepoint. 
%A special case of this construction, which is of particular importance to us, is the representation sphere associated to a finite $G$-set $M$, namely $S^M \defeq S^{\mathbb{R}[M]}$ where $\mathbb{R}[M]$ is the $\mathbb{R}$-linearization of $M$ for which the canonical basis is orthonormal. These can also be defined over simplicial sets: Taking $S^1$ to be the quotient $\underline{\Delta}^1/\partial \underline{\Delta}^1$ (where $\underline{\Delta}^1$ denotes the standard $1$-simplex), based at $\partial \underline{\Delta}^1$, we set $S^M \defeq(S^1)^{\wedge M}$, with $G$ permuting the smash factors according to its action on $M$. Fixing a based homeomorphism $|S^1|\cong S^1=S^\mathbb{R}$ induces a $G$-homeomorphism $|S^M|= |(S^1)^{\wedge M}|\cong |S^1|^{\wedge M} \cong (S^1)^{\wedge M} \cong S^M$ for all finite $G$-sets $M$.
\end{Example}

\begin{Def}[Mapping spaces] 
%\todo{only notation quickly}
Given two based $G$-spaces/simplicial sets $X$ and $Y$ we write $\map(X,Y)$ for the space/simplicial set of not necessarily equivariant based continuous maps with conjugation $G$-action, and $\map_G(X,Y)$ for the subspace of equivariant maps.
%The space $\map(X,Y)$ carries a $G$-action via conjugation, the $G$-fixed points are given by the subspace of $G$-equivariant maps $\map_G(X,Y)$.
%For two based $G$-simplicial sets $X$ and $Y$ the simplicial mapping space $\map(X,Y)$ has as $n$-simplices the set of all based simplicial maps $(\underline{\Delta}^n)_+\wedge X\to Y$ with faces and degeneracies through the cosimplicial object $\underline{\Delta}$ in the first variable. Again, this carries a $G$-action via conjugation, where $G$ acts trivially on $(\underline{\Delta}^n)_+$.
\end{Def}

\begin{Remark} These definitions give two enrichments of $G\T_*$, one over $G\T_*$ itself with mapping spaces $\map(-,-)$ with conjugation $G$-action and one over $\T_*$ using $\map_G(-,-)$. Likewise, $G\mathcal{S}_*$ is enriched over $G\mathcal{S}_*$ and $\mathcal{S}_*$.
\end{Remark}

%\begin{Remark} \todo{Do I need this?} The smash product together with the mapping spaces $\map_G(-,-)$ make $G\T_*$ and $G\mathcal{S}_*$ closed symmetric monoidal categories with monoidal unit $S^0$.
%\end{Remark}

\begin{Def}[Change of groups]
Given a subgroup $H\leq G$, we write $\res_H^G:G\T_*\to H\T_*$ and $\res_H^G:G\mathcal{S}_*\to H\mathcal{S}_*$ for the associated restriction functors. The left adjoint induction functors are denoted $G\ltimes_H -$, the right adjoint coinduction functors by $\map_H(G,-)$.
%A subgroup inclusion $H\to G$ induces \emph{restriction} functors $\res_H^G:G\T_*\to H\T_*$ and $\res_H^G:G\mathcal{S}_*\to H\mathcal{S}_*$ by sending a based $G$-space/$G$-simplicial set $X$ to the same space/simplicial set with restricted $H$-action and not changing the morphisms.
%The restriction functor has both a left and a right adjoint. The prior is given by \emph{induction} $G\ltimes_H -:H\T_*\to G\T_*$ (respectively $G\ltimes_H -:H\mathcal{S}_*\to G\mathcal{S}_*$), sending a based $H$-space/$H$-simplicial set $X$ to the based space/simplicial set $G\ltimes_HX\defeq G_+\wedge_H X$ with $G$-action via $g[g',x]=[gg',x]$. The right adjoint is called \emph{coinduction} and given by the space $\map_H(G_+,X)$ of $H$-equivariant maps from $G$ to $X$, with $G$-action through $(g\cdot \varphi)(g')=\varphi(g'g)$.
\end{Def}
The \emph{double coset formula}, which will be used throughout this paper, gives a decomposition of an induction followed by a restriction. Given a finite group $G$ and two subgroups $H$ and $K$, we let $K\backslash G/H$ denote the set of double cosets, i.e., equivalence classes of $G$ under the $K\times H$ action given by $(k,h)\cdot g=kgh^{-1}$. Then for every based $H$-space/$H$-simplicial set $X$ there is a natural $K$-equivariant decomposition
 \[ {\res}_K^G (G\ltimes_H X)\cong \bigvee_{[g]\in K\backslash G/H} K\ltimes_{K\cap gHg^{-1}}(c_g^*({\res}^G_{g^{-1}Kg \cap H}X)). \]
Here, $c_g^{*}$ stands for restriction along the conjugation isomorphism $g^{-1}(-)g:K\cap gHg^{-1}\to g^{-1}Kg \cap H$. On the summand belonging to a double coset representative $g$, the isomorphism above is the adjoint of the $K\cap gHg^{-1}$-equivariant map $c_g^*({\res}^H_{g^{-1}Kg \cap H}X)\to {\res}_{K\cap gHg^{-1}}^G X$ given by multiplication with $g$.

Likewise, there is the following decomposition of coinduction followed by restriction:
\begin{equation} \label{eq:dec2} {\res}_K^G (\map_H(G,X))\cong \prod _{[g]\in K\backslash G/H} \map_{K\cap gHg^{-1}}(K,c_{g}^*({\res}^G_{g^{-1}Kg \cap H} X)). \end{equation}
Furthermore, induction and coinduction are related by a natural transformation \[ \gamma_X:G\ltimes_H X\to \map_H(G_+,X),\] which is adjoint to the $H$-map $X\to \map_H(G_+,X)$ sending $x$ to the function which maps $g$ to $gx$ if $g$ is contained in $H$ and the basepoint otherwise.

\subsection{Genuine model structure}
%\todo{perhaps shorten a bit, say in the text what genuine $G$-equivalences and fibrations are, also for the generating cofibrations and acyclic cofibrations}
We review certain model structures on based $G$-spaces and based $G$-simplicial sets, starting with the genuine model structure. For the notions of non-equivariant Serre and Kan fibrations and model category terminology we refer to the textbooks \cite{Qui67}, \cite{Hir03} or \cite{Hov99}.

A map of based $G$-spaces or based $G$-simplicial sets is called a \emph{genuine $G$-equivalence} (\emph{genuine $G$-fibration}) if for all subgroups $H$ of $G$ the fixed point map $f^H:X^H\to Y^H$ is a weak homotopy-equivalence (respectively Serre/Kan fibration). It is called a genuine $G$-cofibration if it has the left lifting property with respect to all maps that are simultaneously genuine $G$-equivalences and genuine $G$-fibrations, which in the simplicial case is equivalent to being levelwise injective. These classes define a proper, cofibrantly generated and monoidal model structure, with sets of generating cofibrations $I_G$ and acyclic cofibrations $J_G$ given by smashing their respective standard non-equivariant versions $I$ and $J$ with $G/H_+$ for all subgroups $H\leq G$.

Throughout the paper we will use the expression \emph{cofibrant $G$-space} to mean a based topological $G$-space or $G$-simplicial set which is cofibrant in the genuine model structure above. These are precisely the retracts of $I_G$-cell complexes. Every $G$-simplicial set is cofibrant. Furthermore, by a \emph{finite} cofibrant $G$-space we will mean a retract of a finite $I_G$-cell complex.
%For a $G$-simplicial set this is equivalent to having only finitely many non-degenerate simplices.

\begin{Example} All representation spheres $S^V$ for a finite dimensional $G$-representation $V$ are cofibrant. For linearizations of finite $G$-sets this follows from the simplicial description above, for general $V$ it is a consequence of a result of Illman \cite{Ill78}.
\end{Example}

\subsection{Model structures with respect to a family of subgroups}
\label{sec:fammod}
The later treatment of spectra will show the need for other model structures on $G$-spaces and $G$-simplicial sets where weak equivalences are those $G$-maps that induce weak homotopy equivalences only on fixed points for a chosen set of subgroups of $G$, more precisely for a \emph{family} of subgroups.

\begin{Def} Let $G$ be a finite group. A non-empty set of subgroups of $G$ is called a \emph{family} if it is closed under conjugation and passage to subgroups.\end{Def}
Families of the following type are of particular importance to us:
\begin{Example} \label{exa:freefam} Let $G$ and $K$ be two finite groups. Then the family of subgroups of $G\times K$ which intersect $1\times K$ trivially is denoted by $\F^{G,K}$. Every such subgroup is of the form $\{(h,\varphi(h))\ |\ h\in H\}$ for a unique subgroup $H$ of $G$ and group homomorphism $\varphi:H\to K$.
\end{Example}
Given a family $\F$, we let $E\F$ denote a universal $G$-space for $\F$, i.e., an (unbased) cofibrant $G$-space with the property that the $H$-fixed points for a subgroup $H$ are contractible if $H$ is in~$\F$ and empty otherwise.
%Such a universal space always exists and is unique up to $G$-homotopy equivalence.
We call a $G$-simplicial set universal for $\F$ (and also denote it $E\F$) if its geometric realization is a universal $G$-space for $\F$.

\begin{Def}[$\F$-equivalence] A $G$-map $f:X\to Y$ of based $G$-spaces or based $G$-simplicial sets is called an \emph{$\F$-equivalence} if for all subgroups $H$ of $G$ that lie in $\F$ the fixed point map $f^H:X^H\to Y^H$ is a weak homotopy equivalence, or equivalently if $E\F_+\wedge f$ is a genuine $G$-equivalence.
\end{Def}

\begin{comment}
We will need the following two versions of the equivariant Whitehead theorem: (cf. \cite[Prop. 2.7]{Ada84}) \todo{where do I need this?}
\begin{Prop}[Whitehead theorem]
\label{prop:whitehead}\begin{itemize}
\item Let $X$ be a $G$-CW complex with isotropy in $\F$ and $f:Y\to Z$ an $\F$-equivalence. Then $f$ induces a bijection $[X,Y]^G\cong [X,Z]^G$.

\item Let $(X,A)$ be a relative $G$-CW complex with all relative isotropy in $\F$ and $Z$ a based $G$-space which is $\F$-equivalent to a point. Then the restriction map $[X,Z]^G\to [A,Z]^G$ is a bijection.
\end{itemize}
\end{Prop}
\end{comment}
Equivalently, using an equivariant Whitehead theorem (cf. \cite[Prop. 2.7]{Ada84}), a map of based $G$-spaces is an $\F$-equivalence if and only if $\map(E\F_+,f)$ is a genuine $G$-equivalence. We say that a based $G$-space $X$ is \emph{$\F$-local} if the canonical map $X\to \map(E\F_+,X)$ is a genuine $G$-equivalence, and a based $G$-simplicial set is $\F$-local if its geometric realization is.
%By 2-out-of-3 for genuine $G$-equivalences and the previous lemma, it follows that a map between $\F$-local $G$-spaces is an $\F$-equivalence if and only if it is a genuine $G$-equivalence. A based $G$-simplicial set $X$ is said to be $\F$-local if its geometric realization is.

The class of $\F$-equivalences takes part in two model structures. For this we define:
\begin{Def} A map $f:X\to Y$ of based $G$-spaces/$G$-simplicial sets is called a
\begin{itemize}
 \item \emph{projective $\F$-cofibration} if it is a genuine $G$-cofibration where all the points away from the image of $f$ have isotropy in $\F$.
 \item \emph{projective $\F$-fibration} if it induces a Serre/Kan fibration $f^H:X^H\to Y^H$ for all $H\in \F$.
 \item \emph{mixed $\F$-fibration} if it has the right lifting property with respect to all maps that are genuine cofibrations and $\F$-equivalences.
\end{itemize}
\end{Def}

\begin{Prop}[Projective model structure] \label{prop:projmod} Let $G$ be a finite group and $\F$ a family of subgroups. Then the classes of projective $\F$-cofibrations, $\F$-equivalences and projective $\F$-fibrations assemble to a proper, cofibrantly generated, monoidal and $G$-topological ($G$-simplicial) model structure on based $G$-spaces ($G$-simplicial sets), called the \emph{projective} $\F$-model structure.
\end{Prop}
Here and below, $G$-topological or $G$-simplicial is meant with respect to the genuine model structure of the previous section. The projective model structure is cofibrantly generated by the subsets $I^{\F}_{G,proj}$ of $I_G$ and $J_{G,proj}^{\F}$ of $J_G$ of the maps of the form $G/H_+\wedge f$ with $H$ an element of $\F$.

\begin{Prop}[Mixed model structure] \label{prop:mixedmod} Let $G$ be a finite group and $\F$ a family of subgroups. Then the classes of genuine $G$-cofibrations, $\F$-equivalences and mixed $\F$-fibrations assemble to a proper, cofibrantly generated, monoidal and $G$-topological ($G$-simplicial) model structure on based $G$-spaces ($G$-simplicial sets), called the \emph{mixed} $\F$-model structure.
\end{Prop}
The mixed model structure can be obtained from the genuine model structure by applying Bousfield's localization theorem \cite[Thm. 9.3]{Bou01} with respect to the replacement $\alpha_X:X\to \map(E\F_+,X^f)$, where $(-)^f$ is a genuine fibrant replacement functor (e.g. the identity for based $G$-spaces and $\mathcal{S}(|.|)$ for based $G$-simplicial sets). By the characterization of fibrations in \cite{Bou01}, it follows that the fibrant objects are precisely the genuinely $G$-fibrant and $\F$-local $G$-spaces or $G$-simplicial sets, and that the mixed model structure is cofibrantly generated with the previous $I_G$ as generating cofibrations and
\[ J_{G,mix}^{\F}\defeq \{j \square \alpha\ |\ j\in J_G\}\cup J_G\]
as generating acyclic cofibrations. Here, $\alpha$ is the inclusion $E\F_+\to C(E\F)_+$ into the cone and $\square$ denotes the pushout-product.

\begin{comment}
\todo{do i need the fibrations or only the fibrant objects?} It then follows that mixed $\F$-fibrations can be characterized as exactly those based $G$-maps $f:X\to Y$ which are genuine $G$-fibrations and for which in addition the square
\[
	\xymatrix{X \ar[rr]^-{\alpha_X} \ar[d]_f && \map(E\F_+,X^f) \ar[d]^{\map(E\F_+,f^f)} \\
	Y \ar[rr]_-{\alpha_Y} && \map(E\F_+,Y^f)}\]
is homotopy cartesian in the genuine model structure. In particular, a based $G$-space/$G$-simplicial set is mixed $\F$-fibrant if and only if it is genuinely $G$-fibrant and $\F$-local. Using this characterization of the fibrations and the adjunction between smash product and mapping space, we see that 

\end{comment}
%the mixed model structure is cofibrantly generated with the previous $I_G$ as generating cofibrations and
%\[ J_{G,mix}^{\F}\defeq \{j \square \alpha\ |\ j\in J_G\}\cup J_G\]
%as generating acyclic cofibrations. Here, $\alpha$ is the inclusion $E\F_+\to C(E\F)_+$ into the cone and $\square$ denotes the pushout-product of two maps.

\section{$G$-symmetric spectra}
In this section we introduce $G$-symmetric spectra, explain various point-set level constructions and define level model structures.
\subsection{Definitions}
\label{sec:defsymm}
\begin{Def}[Symmetric spectrum]
A \emph{symmetric spectrum} $X$ of spaces or simplicial sets consists of \begin{itemize}
\item a based $\Sigma_n$-space/$\Sigma_n$-simplicial set $X_n$ and
\item a based \emph{structure-map} $\sigma_n:X_n\wedge S^1 \to X_{n+1}$
\end{itemize}
for all $n\in \N$. This data has to satisfy the condition that for all $n,m\in\mathbb{N}$ the \emph{iterated structure map} \[
	\sigma_n^m: X_n\wedge S^m\cong (X_n\wedge S^1)\wedge S^{m-1}\xr{\sigma_n\wedge S^{m-1}} X_{n+1}\wedge S^{m-1}\to \hdots \xr{\sigma_{n+m-1}}  X_{n+m}
\]
is $\Sigma_n\times \Sigma_m$-equivariant. Here, the $\Sigma_n\times \Sigma_m$-action on $X_n\wedge S^m$ is the smash product of the $\Sigma_n$-action on $X_n$ and the $\Sigma_m$-action on $S^m$ as the permutation sphere for the natural $\Sigma_m$-set $\underline{m}\defeq \{1,\hdots, m\}$. The action on $X_{n+m}$ is induced from the inclusion $\Sigma_n\times \Sigma_m\to \Sigma_{n+m}$ which arises from disjoint union and the bijection $\underline{n}\sqcup \underline{m}\cong \underline{n+m}$ that sends each $k$ in $\underline{n}$ to itself and each $l$ in $\underline{m}$ to $n+l$.

A \emph{map of symmetric spectra} $f:X\to Y$ is a sequence of based $\Sigma_n$-equivariant maps $f_n:X_n\to Y_n$ such that $f_{n+1}\circ \sigma_n^{(X)}=\sigma_{n+1}^{(Y)}\circ (f_n\wedge S^1)$ for all $n\in \N$.
\end{Def}

We denote the category of symmetric spectra over spaces or simplicial sets by $Sp^\Sigma_{\T}$ or $Sp^\Sigma_{\mathcal{S}}$, respectively. In statements that make sense over both categories we sometimes simply write $Sp^{\Sigma}$ and mean they hold in either of them.

\begin{Def}[$G$-symmetric spectrum] Let $G$ be a finite group. A \emph{$G$-symmetric spectrum} of spaces or simplicial sets is a symmetric spectrum together with a $G$-action via automorphisms of symmetric spectra, or equivalently a functor $G\to Sp^\Sigma_\T$ or $G\to Sp^\Sigma_\mathcal{S}$. A \emph{morphism of $G$-symmetric spectra} is a morphism of symmetric spectra that commutes with the given $G$-actions. We denote the categories of $G$-symmetric spectra by $GSp^{\Sigma}_{\T}$ and $GSp^{\Sigma}_{\mathcal{S}}$.
\end{Def}

Equivalently, a $G$-symmetric spectrum is a symmetric spectrum $X$ together with a $G$-action on each level $X_n$ which commutes with the $\Sigma_n$-action and for which all structure maps $\sigma_n:X_n\wedge S^1\to X_{n+1}$ are $G$-equivariant for the trivial action on $S^1$.

\begin{Example}[Suspension spectra] Let $A$ be a based $G$-space ($G$-simplicial set). Then the \emph{suspension spectrum} of $A$, denoted $\Sigma^{\infty}A$, has  $(\Sigma^{\infty} A)_n\defeq  A\wedge S^n$ as its $n$-th level where the $(G\times \Sigma_n)$-action is the smash product of the $G$-action on $A$ and the permutation $\Sigma_n$-action on~$S^n$. The structure map is the associativity homeomorphism \[(A\wedge S^n)\wedge S^1\cong A\wedge (S^n\wedge S^1)\cong A\wedge S^{n+1}.\]
Taking $A$ to be $S^0$ with trivial $G$-action yields the sphere spectrum $\Sigma^{\infty}S^0$, which we also denote by $\mathbb{S}$.
\end{Example}
The categories $GSp^{\Sigma}_{\T}$ and $GSp^{\Sigma}_{\mathcal{S}}$ have all limits and colimits and these are computed levelwise, cf. \cite[Prop. 1.2.10]{HSS00}.
%More precisely, for $I$ a small category and $F:I\to Sp^{\Sigma}$ a functor the $G$-symmetric spectrum $\colim F$ with $(\colim F)_n\defeq \colim  (F_n)$ and structure map \begin{eqnarray*} & \colim  (F_n)\wedge S^1\cong \colim  (F_n\wedge S^1)\xr{\colim  (\sigma_n)} \colim (F_{n+1})
%\end{eqnarray*}
%is a colimit for $F$. A limit is given by $\lim F$ with $(\lim F)_n\defeq \lim(F_n)$ and structure map adjoint to \[ \lim (F_n)\xr{\lim (\widetilde{\sigma}_n)} \lim(\Omega F_{n+1})\cong \Omega \lim (F_{n+1}). \]
\subsection{Evaluation on finite $G$-sets and generalized structure maps}
\label{sec:evaluation}
For the homotopy theory of $G$-symmetric spectra it is essential that they can be evaluated on finite $G$-sets, which we now explain.

Let $M$ be a finite $G$-set of order $m$. We denote by $\Bij(\underline{m},M)$ the discrete space or simplicial set of bijections between the sets $\underline{m}=\{1,\hdots,m\}$ and $M$. It possesses a right $\Sigma_m$-action by precomposition and a left $\Sigma_M$-action by postcomposition.
\begin{Def}[Evaluation] The \emph{evaluation} of a $G$-symmetric spectrum $X$ on a finite $G$-set $M$ is defined by
\begin{align*}	X(M) & \defeq X_m \wedge_{\Sigma_m} \Bij(\underline{m},M)_+ \\ & \defeq X_m \wedge \Bij(\underline{m},M)_+/\{(\sigma x,f)\sim(x,f\sigma),\sigma \in \Sigma_m\}
 \end{align*}
with diagonal $G$-action $g[x\wedge f]\defeq [gx\wedge gf]$.
\end{Def}
%\begin{Remark} The notation $X_m \wedge_{\Sigma_m} \Bij(\underline{m},M)_+$ might be slightly confusing since $\Sigma_m$ acts from the left on $X_m$ and from the right on $\Bij(\underline{m},M)_+$. In analogy to the tensor product of modules one would expect the notation $\Bij(\underline{m},M)_+\wedge_{\Sigma_m} X_m$. We write it the other way around because both this construction and the structure maps $X_m\wedge S^1\to X_{m+1}$ are special cases of the action of the category $G\mathbf{\Sigma}$ on the levels of a $G$-symmetric spectrum, as we describe in Section \ref{sec:free}, and we have chosen to define sphere actions from the right.
%\end{Remark}
%The underlying space/simplicial set of $X(M)$ should be thought of as a coordinate change.
Since $\Bij(\underline{m},M)$ is free and transitive as a right $\Sigma_m$-set, any choice of bijection from $\underline{m}$ to $M$ gives rise to a non-equivariant isomorphism from $X_m$ to $X(M)$. In particular, $X(\underline{m})$ is canonically isomorphic to $X_m$. However, the $G$-action on $X(M)$ is usually different from that on $X_m$. It is twisted by the $G$-set structure on $M$, which corresponds to a group homomorphism from $G$ into $\Sigma_M$. In particular, the $G$-action on $X(M)$ depends on the $\Sigma_m$-action on $X_m$.

Evaluation is functorial in isomorphisms of finite sets. Given any bijective map $\varphi:M\to N$ we obtain an induced isomorphism $X(\varphi):X(M) \to X(N)$ by sending $[x\wedge f]$ to $[x\wedge (\varphi\circ f)]$. This isomorphism is in general only $G$-equivariant if $\varphi$ is.
%In particular, the induced $\Sigma_M$-action on $X(M)$ does in general not commute with the $G$-action. However, the two assemble to an action by the semi-direct product $G\ltimes \Sigma_M$ associated to the composition $G\to \Sigma_M\xr{conj} \Aut(\Sigma_M)$, where the first map is the group homomorphism defining the $G$-set structure on $M$.

%The following are two examples of evaluations:
\begin{Example} Let $A$ be a based $G$-space or $G$-simplicial set, $M$ a finite $G$-set. Then the map $(\Sigma^\infty A)(M)\to A\wedge S^M$ that sends a class $[(a\wedge x)\wedge f]$ to $a\wedge f_*(x)$ is a $G$-isomorphism for the diagonal action on the target.
\end{Example}

%We note that for non-trivial $M$ the $G$-action on $\mathbb{S}(M)\cong S^M$ is non-trivial even though the one on the sphere spectrum $\mathbb{S}$ is. The reason for this phenomenon is the influence of the symmetric group actions of the spectrum. This is even more apparent in the next example:

\begin{Example} \label{exa:nontriveva} Let $G$ be the symmetric group $\Sigma_n$ and $M$ be the natural $\Sigma_n$-set $\underline{n}$, $X$ any symmetric spectrum with trivial $\Sigma_n$-action. Then $X(\underline{n})$ is canonically isomorphic to $X_n$ but now carries the $\Sigma_n$-action coming from the data of the underlying symmetric spectrum. In contrast, evaluating at $\{1,\hdots,n\}$ with \emph{trivial} $\Sigma_n$-action yields $X_n$ with trivial action.
\end{Example}

These evaluations are connected by so-called generalized structure maps. Let $G$ be a finite group, $M$ and $N$ two finite $G$-sets of orders $m$ and $n$, respectively, and $X$ a $G$-symmetric spectrum. We further choose a bijection $\psi: \underline{n} \xr{\cong}N$.

\begin{Def}[Generalized structure map] The map \begin{eqnarray*}
	 \sigma_M^N: X(M)\wedge S^N & \to & X(M\sqcup N) \\
	 ([x\wedge f]\wedge s) & \mapsto & [\sigma_m^n(x\wedge \psi ^{-1}_*(s))\wedge (f\sqcup \psi)]
\end{eqnarray*}
is called the \emph{generalized structure map} of $M$ and $N$.
\end{Def}
It is straightforward to check:
\begin{Lemma} \label{lem:genstructmap} The generalized structure map does not depend on the choice of bijection $\psi: \underline{n} \xr{\cong}N$. Furthermore, it is $G$-equivariant for the diagonal $G$-action on $X(M)\wedge S^N$.
\end{Lemma}
Here, the second statement is a consequence of $\sigma_m^n$ being $(\Sigma_m\times \Sigma_n)$-equivariant.
\begin{Remark} \label{rem:genstructmap} In fact we will also need the generalized structure map in the situation where $G$ acts on $M\sqcup N$ but not necessarily in a way such that $M$ and $N$ are preserved by $G$. In this case $G$ does not act on $X(M)\wedge S^N$, but it acts on the wedge over all $X(\alpha(M))\wedge S^{(M\sqcup N)-\alpha(M)}$ for injections $M\hookrightarrow M\sqcup N$ (cf. Section \ref{sec:free}). The map from this wedge to $X(M\sqcup N)$ obtained via the various generalized structure maps is then $G$-equivariant.
\end{Remark}

%\subsection{$G$-orthogonal spectra} \todo{put this later! before teh equivalence}

\subsection{Functors on $G$-symmetric spectra}
\label{sec:funcgsym} The following functors on equivariant symmetric spectra are obtained by applying the respective space-level functors degreewise with suitably defined structure maps (see \cite[Def. 1.2.9]{HSS00} on how to prolong enriched space-level functors to functors on symmetric spectra):
%In this section we explain various functors on the category of $G$-symmetric spectra. The first class of examples describes lifts from functors on based $G$-spaces/$G$-simplicial sets:

%Every $\T_*$-enriched functor $F:G\T_*\to K\T_*$ for possibly different finite groups $K$ and $G$ can be lifted to a functor (of the same name) $F:GSp^{\Sigma}_\T\to KSp^{\Sigma}_\T$ via $F(X)_n\defeq F(X_n)$. The structure map $F(X_n)\wedge S^1\to F(X_{n+1})$ is adjoint to
%\[ S^1\xr{\widetilde{\sigma}_n} \map_G(X_n,X_{n+1}) \xrightarrow{F} \map_K(F(X_n),F(X_{n+1})), \]
%where $\widetilde{\sigma}_n$ is adjoint to the structure map of $X$. Moreover, if two such functors on the space-level form an adjunction, then so do their spectrum-level versions. The same applies to $G$-symmetric spectra of simplicial sets, with $\T_*$ replaced by $\mathcal{S}_*$.
%The following examples arise via this construction. We omit the index $\T$ or $\mathcal{S}$, as all of them make sense in both categories.
\begin{itemize}
	\item The trivial action functor $Sp^{\Sigma}\to GSp^{\Sigma}$ with its left adjoint $G$-orbits $(-)/G$ and its right adjoint $G$-fixed points $(-)^G$.
	\item The smash product $A\wedge -$ for a based $G$-space/$G$-simplicial set $A$ and its right adjoint $\map(A,-)$. 
	\item The restriction functor $\res_H^G:GSp^{\Sigma}\to HSp^{\Sigma}$ for a subgroup $H\leq G$ with its left adjoint $G\ltimes_H -$ and its right adjoint $\map_H(G,-)$.
	\item The geometric realization functor $|.|:GSp^{\Sigma}_{\mathcal{S}}\to GSp^{\Sigma}_{\T}$ with right adjoint singular complex $\mathcal{S}:Sp^{\Sigma}_\T\to Sp^{\Sigma}_{\mathcal{S}}$.
\end{itemize}
%\begin{Def}[Geometric realization and singular complex] The geometric realization $|X|$ of a $G$-symmetric spectrum of simplicial sets $X$ is a $G$-symmetric spectrum of spaces with $n$-th level $|X|_n\defeq |X_n|$ and structure map $|X|_n\wedge S^1\cong |X_n\wedge S^1|\xrightarrow{|\sigma_n|} |X_{n+1}|=|X|_{n+1}$.
%As in the unstable case, geometric realization is left adjoint to the singular complex functor $\mathcal{S}:Sp^{\Sigma}_\T\to Sp^{\Sigma}_{\mathcal{S}}$. The singular complex $\mathcal{S}X$ of a $G$-symmetric spectrum of spaces $X$ is given in level $n$ by $(\mathcal{S}X)_n\defeq \mathcal{S}X_n$ with structure map adjoint to $(\mathcal{S}X)_n=\mathcal{S}X_n\xrightarrow{\mathcal{S}\widetilde{\sigma}_n} \mathcal{S}(\Omega X_{n+1})\cong \Omega\mathcal{S}X_{n+1}=\Omega (\mathcal{S}X)_{n+1}$, where $\widetilde{\sigma}_n$ denotes the adjoint structure map of $X$.
%\end{Def}
An important functor which does not arise through a levelwise construction is the following:
\begin{Def}[Shift] \label{def:shift} The \emph{shift} along a finite $G$-set $M$ is defined by $(sh^M X)_n \defeq X(M\sqcup \underline{n})$ with $\Sigma_n$ acting through $\underline{n}$. The structure map \[ \sigma_n:(sh^M X)_n\wedge S^1=X(M\sqcup \underline{n})\wedge S^1\to X(M\sqcup \underline{n+1})=(sh^MX)_{n+1}\] is the generalized structure map $\sigma_{M\sqcup \underline{n}}^1$ of $X$.

For all $G$-sets $M$ there is a natural map $\alpha_X^M:S^M\wedge X\to sh^M X$ given in level $n$ by the composite\[
	S^M\wedge X_n\cong X_n\wedge S^M \xr{\sigma_n^M} X(\underline{n}\sqcup M)\xr{X(\tau_{\underline{n},M})} X(M\sqcup \underline{n})=(sh^MX)_n.
\]
Here, the notation $X(\tau_{\underline{n},M})$ stands for the $G$-map induced from the symmetry isomorphism $\tau_{\underline{n},M}$ of the $G$-sets $\underline{n}\sqcup M$ and $M\sqcup \underline{n}$.
\end{Def}

\begin{Def}[Enrichments] Given $G$-symmetric spectra $X$ and $Y$, we write $\map_{Sp^{\Sigma}}(X,Y)$ for the space/simplicial set and $\Hom(X,Y)$ for the symmetric spectrum of not necessarily equivariant morphisms, cf. \cite[Sec. 1.3 and Def. 2.2.9]{HSS00}. Both carry a $G$-action by conjugation, with fixed points $\map_{GSp^\Sigma}(X,Y)$ respectively $\Hom_G(X,Y)$.
\end{Def}

\begin{Def}[Smash product] The category of $G$-symmetric spectra inherits a symmetric monoidal product by forming the smash product of the underlying non-equivariant spectra (cf. \cite[Sec. 2]{HSS00}) and giving it the diagonal $G$-action. For a fixed $G$-symmetric spectrum $X$, the functor $-\wedge X$ is left adjoint to $\Hom_G(X,-)$.
\end{Def}
%The following is a formal consequence of REF \todo{Referenz einfuegen}:
%\begin{Prop}[Adjunction smash product/internal hom] For all $G$-symmetric spectra $X, Y$ and $Z$ there is a (natural) $G$-isomorphism
%  \[ \map_{Sp^{\Sigma}}(X\wedge Y,Z)\cong \map_{Sp^{\Sigma}}(X,\Hom(Y,Z)). \]
%\end{Prop}
\subsection{Free $G$-symmetric spectra and $G$-symmetric spectra as enriched functors}
\label{sec:free}
For every finite $G$-set $M$ the evaluation functors $-(M):Sp^\Sigma_\T \to G\T_*$ and $-(M):Sp^\Sigma_\mathcal{S} \to G\mathcal{S}_*$ have left adjoints called \emph{free $G$-symmetric spectra}, which we now describe.

Given two finite sets $M$ and $N$ we set \[\mathbf{\Sigma}(M,N)=\bigvee \limits_{\alpha:M\hookrightarrow N}^{}S^{N-\alpha(M)}, \]
where $N-\alpha(M)$ is the complement of the image of $\alpha$ in $N$. This comes with two group actions: a right one by $\Sigma_M$ via precomposition on the indexing set of injective maps and identities on the spheres and a left one by $\Sigma_N$ where an element $\sigma\in \Sigma_N$ maps a sphere $S^{N-\alpha(M)}$ associated to $\alpha:M\hookrightarrow N$ to the one associated to $\sigma\alpha:M\hookrightarrow N$ via the induced action of $\sigma_{|N-\alpha(M)}:N-\alpha(M)\to N-(\sigma\alpha)(M)$. In the case where $M$ and $N$ are $G$-sets, we can pull back this $(\Sigma_N\times \Sigma_M^{op})$-action to get a conjugation action by $G$ on $\mathbf{\Sigma}(M,N)$.
%If $M$ and $N$ have the same number of elements, $\mathbf{\Sigma}(M,N)$ is canonically isomorphic to $\Bij(M,N)_+$, the discrete space/simplicial set of bijections from $M$ to $N$ with an added basepoint.
%We note that the geometric realization of the $G$-simplicial set $\mathbf{\Sigma}(M,N)$ is $G$-homeomorphic to the topological $\mathbf{\Sigma}(M,N)$ and thus the latter 
%possesses the structure of a $G$-CW complex.

For another finite set $K$ there is a composition map \begin{eqnarray*}
	\circ:\mathbf{\Sigma}(M,N)\wedge \mathbf{\Sigma}(N,K) \to \mathbf{\Sigma}(M,K) 
\end{eqnarray*}
defined via the formula $(\alpha;x)\circ (\beta;y)\defeq (\beta\alpha;\beta(x)\wedge y)$. These composition maps are associative, unital (with respect to the identity in $\mathbf{\Sigma}(M,M)\cong \Bij(M,M)_+$) and $G$-equivariant. We obtain categories called $G\mathbf{\Sigma}_\T$ and $G\mathbf{\Sigma}_\mathcal{S}$, enriched over based $G$-spaces and based $G$-simplicial sets, respectively. The objects are finite $G$-sets and the morphisms between $M$ and $N$ are $\mathbf{\Sigma}(M,N)$.

We write $\underline{G\T}_*$ to denote the category of $G$-spaces enriched over themselves via the mapping spaces $\map(-,-)$, and similarly $\underline{G\mathcal{S}}_*$ in the simplicial case. Every $G$-symmetric spectrum $X$ of spaces or simplicial sets gives rise to a $G\T_*$-enriched functor $G\mathbf{\Sigma}_\T\to \underline{G\T}_*$ (respectively $G\mathcal{S}_*$-enriched functor $G\mathbf{\Sigma}_\mathcal{S}\to \underline{G\mathcal{S}}_*$) in the following way: A finite $G$-set $M$ is sent to $X(M)$, the evaluation of $X$ on $M$. The map $\mathbf{\Sigma}(M,N)\to \map(X(M),X(N))$ is the adjoint of \begin{eqnarray*}  X(M)\wedge \mathbf{\Sigma}(M,N) & \longrightarrow & X(N) \\
									x\wedge (\alpha;y) & \mapsto & X(\alpha\sqcup i)(\sigma_M^{N-\alpha(M)}(x\wedge y)). 
\end{eqnarray*}
Here, $X(\alpha\sqcup i)$ is the isomorphism $X(M\sqcup (N-\alpha(M)))\cong X(N)$ induced from $\alpha\sqcup i:M\sqcup (N-\alpha(M))\to N$ in the way described in Section \ref{sec:evaluation} and we have made use of the generalized structure maps $\sigma_M^{N-\alpha(M)}$ in the sense of Remark \ref{rem:genstructmap}.

In fact, it is possible to show that this construction defines an equivalence from the category of $G$-symmetric spectra to the category of $G\T_*$-enriched functors $\Fun_{G\T_*}(G\Sigma_{\T},\underline{G\T}_*)$ (respectively $G\mathcal{S}_*$-enriched functors $\Fun_{G\mathcal{S}_*}(G\Sigma_{\mathcal{S}},\underline{G\mathcal{S}}_*)$). The inverse functor is the restriction to trivial $G$-sets $\{1,\hdots,n\}$ and the sphere associated to the inclusion $\{1,\hdots,n\}\hookrightarrow \{1,\hdots,n+1\}$.
\begin{Def}[Free $G$-symmetric spectra] \label{def:free} Let $A$ be a based $G$-space/$G$-simplicial set and $M$ a finite $G$-set. The \emph{free $G$-symmetric spectrum on $A$ in level $M$} is denoted by $\mathscr{F}_M A$ and defined via \[	(\mathscr{F}_M A)_n\defeq A\wedge \mathbf{\Sigma}(M,\underline{n}) \]
with diagonal $G$-action and $\Sigma_n$-action through $\mathbf{\Sigma}(M,\underline{n})$. The structure map is the composition \[
 A\wedge \mathbf{\Sigma}(M,\underline{n})\wedge S^1 \hookrightarrow A\wedge \mathbf{\Sigma}(M,\underline{n})\wedge \mathbf{\Sigma}(\underline{n},\underline{n+1})\xr{A\wedge \circ} A\wedge \mathbf{\Sigma}(M,\underline{n+1}).
\]
%Here, the first map is induced by the embedding $S^1\hookrightarrow \mathbf{\Sigma}(\underline{n},\underline{n+1})$ as the sphere associated to the inclusion $\underline{n}\hookrightarrow \underline{n+1}$.
\end{Def}
%We note that more generally $(\mathscr{F}_M A)(N)$ is canonically $G$-isomorphic to $A\wedge \mathbf{\Sigma}(M,N)$ with diagonal $G$-action.
When it is unclear with respect to which finite group the free spectrum is formed, we write~$\mathscr{F}^{(G)}_M A$.

\begin{Prop}[Adjunction between free spectra and evaluation] Let $M$ be a finite $G$-set, $A$ a based $G$-space (or based $G$-simplicial set) and $Y$ a $G$-symmetric spectrum. Then the natural map $\map_{Sp^\Sigma}(\mathscr{F}_MA,Y)\to \map(A,Y(M))$ that sends a (not necessarily equivariant) morphism of symmetric spectra $f:\mathscr{F}_MA\to Y$ to the composite 
\[ A\cong A\wedge \{id_{M}\}_+\hookrightarrow A\wedge \Sigma(M,M)\cong (\mathscr{F}_M A)(M) \xr{f(M)} Y(M) \]
is a $G$-isomorphism.
\end{Prop}
Viewing $G$-symmetric spectra as an enriched functor category, this is a consequence of the enriched Yoneda Lemma.

As described in Section \ref{sec:evaluation}, the evaluation of a $G$-symmetric spectrum on a $G$-set $M$ carries a $\Sigma_M$-action in addition to the $G$-action and the two assemble to  a $(G\ltimes \Sigma_M)$-action (where the semi-direct product is formed with respect to the conjugation $G$-action on $\Sigma_M$). Thus, $-(M)$ can be thought of as a functor to $(G\ltimes \Sigma_M)\T_*$ or $(G\ltimes \Sigma_M)\mathcal{S}_*$. This functor also has a left adjoint, which in the case of $G$-orthogonal spectra first appeared in Stolz's thesis \cite[Def. 2.2.14]{Sto11}.

\begin{Def}[Semi-free $G$-symmetric spectra] Let $M$ be a finite $G$-set and $A$ a based $(G\ltimes \Sigma_M)$-space/$(G\ltimes \Sigma_M)$-simplicial set. We recall that the natural $\Sigma_M$-action turns $M$ into a $(G\ltimes \Sigma_M)$-set. Then the \emph{semi-free $G$-symmetric spectrum on $A$}, denoted $\mathscr{G}_M A$, is defined as the quotient of the $(G\ltimes \Sigma_M)$-spectrum $\mathscr{F}_M^{(G\ltimes \Sigma_M)} A$ by the action of the normal subgroup $\Sigma_M$.
\end{Def}
As advertised we have:
\begin{Prop}[Adjunction between semi-free spectra and evaluation] For $A$ a based $(G\ltimes \Sigma_M)$-space/$(G\ltimes \Sigma_M)$-simplicial set and $Y$ a $G$-symmetric spectrum there is a natural isomorphism $\map_{GSp^\Sigma}(\mathscr{G}_M A,Y) \cong \map_{G\ltimes \Sigma_M}(A,Y(M))$.
\end{Prop}
\begin{proof}
The adjunction isomorphism is given by the composition
\[\map_{GSp^\Sigma}(\mathscr{G}_M A,Y) \cong \map_{(G\ltimes \Sigma_M)Sp^\Sigma}(\mathscr{F}^{(G\ltimes \Sigma_M)}_M A,Y) \cong \map_{G\ltimes \Sigma_M}(A,Y(M)). \qedhere \]
\end{proof}

There is a more concrete description of $\mathscr{G}_M(A)$ in levels of the form $M\sqcup N$, where it is given by $(G\ltimes \Sigma_{M\sqcup N})\ltimes_{G\ltimes (\Sigma_M\times \Sigma_N)} (A\wedge S^N)$. In addition, all evaluations at $G$-sets of order less than $|M|$ are given by a point. In Section \ref{sec:monoidal} we will use repeatedly that the smash product of a $G$-symmetric spectrum $X$ with a semi-free $G$-symmetric spectrum $\mathscr{G}_m(A)$ can be described explicitly in the following way:
\[ (\mathscr{G}_m(A)\wedge X)_k \cong \begin{cases} * & \text{ for } k<m \\
                                    \Sigma_{m+n}\ltimes_{\Sigma_m\times \Sigma_n} (A\wedge X_n) & \text{ for } k=m+n
                                   \end{cases}
\]
This can be checked via the universal properties of the smash product and semi-free spectra.

\subsection{Latching spaces and the skeleton filtration of a map}
\label{sec:latching}
%\todo{maybe refer to somewhere else for a definition, but check whether we need specifiy properties}
Every morphism of $G$-symmetric spectra $f:X\to Y$ can be factored as a countable composite
\begin{equation} \label{eq:latching} X=F^{-1}[f]\xr{j_0[f]} F^0[f]\xr{j_1[f]} F^1[f] \xr{j_2[f]} \hdots \longrightarrow Y, \end{equation}
which builds $Y$ out of $X$ ``one dimension at a time". This factorization is important for stating the model structures in Section \ref{sec:glev} and for proving the lifting property axioms in them.

The spectra $F^n[f]$ together with maps $i_n[f]:F^n[f]\to Y$ are constructed inductively. We set $F^{-1}[f]=X$ and $i_{-1}[f]=f$. Now we fix an $n\geq 0$ and assume that $F^{n-1}[f]$ and $i_{n-1}[f]$ have already been defined. We denote the $n$-th level of $F^{n-1}[f]$ by $L_n[f]$ and call it the \emph{$n$-th latching space} of $f$. It comes with a $\Sigma_n$-map $\nu_n[f]$ to $Y_n$, the effect of $i_{n-1}[f]$ in level $n$. Then $F^n[f]$ is defined as the pushout
\[ \xymatrix{ \mathscr{G}_n(L_n[f]) \ar[rr]^{\varepsilon_{F^{n-1}[f]}} \ar[d]_{\mathscr{G}_n(\nu_n[f])} && F^{n-1}[f] \ar[d]^{j_n[f]}\\
							\mathscr{G}_n (Y_n) \ar[rr]_{\varepsilon_{F^n [f]}} && F^n [f].}
\]
The map $i_n[f]$ is induced by $i_{n-1}[f]$ and the counit $\mathscr{G}_n(Y_n)\to Y$. The map $i_n[f]:F^n[f]\to Y$ induces an isomorphism in degrees $\leq n$ and so the sequence \eqref{eq:latching} stabilizes in every fixed degree and converges to $Y$. We write $L_nY$ and $\nu_nY$ for the latching spaces and maps of the unique map $*\to Y$.

The connection to model structures comes from the following lemma:
\begin{Lemma} \label{lem:lift} Let $f:X\to Y$ and $g:W\to Z$ be morphisms of $G$-symmetric spectra such that for all $n\geq 0$ the $(G\times \Sigma_n)$-map $\nu_n[f]:L_n[f]\to Y_n$ has the left lifting property with respect to $g_n:W_n\to Z_n$. Then $f$ has the left lifting property with respect to $g$.
\end{Lemma}
\begin{proof} Using the adjunction between semi-free spectra and evaluation, the assumption is equivalent to $\mathscr{G}_n(\nu_n[f])$ having the left lifting property with respect to $g$ for all $n$. But by the properties of the skeleton filtration above, $f$ can be obtained from these by pushouts and countable composition, which preserve the left lifting property.
\end{proof}

\subsection{Level model structures}
\label{sec:glev}
In this section we construct various level model structures on the categories of $G$-symmetric spectra of spaces and simplicial sets.  Precisely, for every $G$-set universe $\U$ we describe a projective and a flat (or $\mathbb{S}$-) model structure, in positive and non-positive versions. As mentioned in the introduction, the former is a variant of the level model structure of $G$-orthogonal spectra of \cite{MM02} and the level model structure of \cite{Man04}, and the latter is a generalization of the non-equivariant flat model structure by Shipley \cite{Shi04} and a translation from the one on $G$-orthogonal spectra by Stolz \cite{Sto11}.

From now on we fix a $G$-set universe $\U$, i.e., a countably infinite $G$-set which is isomorphic to the disjoint union of two copies of itself. Up to non-canonical isomorphism, such a $G$-set universe is always of the form $\bigsqcup_{H\in \mathcal{C}}(\N\times G/H)$ for a non-empty set of subgroups $\mathcal{C}$ which becomes unique if one requires it to be closed under conjugation.

We recall from Example \ref{exa:freefam} that $\F^{G,\Sigma_n}$ is the family of subgroups of $G\times \Sigma_n$ that intersect $1\times \Sigma_n$ only in the neutral element. Elements of this family correspond to pairs consisting of a subgroup $H$ of $G$ and a group homomorphism $\alpha:H\to \Sigma_n$. We denote by $\F^{G,\Sigma_n}_{\U}$ the subfamily of $\F^{G,\Sigma_n}$ corresponding to those pairs $(H,\alpha)$ such that $\underline{n}$ equipped with the $H$-set structure through $\alpha$ allows an $H$-embedding into $\U$. The level model structures on $G$-symmetric spectra are constructed out of the projective and mixed model structures associated to these families for varying $n$ (cf. Section \ref{sec:fammod}). We start with the non-positive versions, the positive modifications are explained in the paragraph preceding Proposition \ref{prop:poslevmod}.
\begin{Def} A morphism $f:X\to Y$ of $G$-symmetric spectra is called a \begin{itemize}
	\item \emph{$G^{\U}$-level equivalence} if for all $n\in \mathbb{N}$ the $(G\times \Sigma_n)$-map $f_n:X_n\to Y_n$ is an $\mathcal{F}_{\U}^{G,\Sigma_n}$-equivalence.
	\item $G^{\U}$-\emph{projective (resp. $G^{\U}$-flat)  level fibration} if for all $n\in \mathbb{N}$ the $(G\times \Sigma_n)$ -map $f_n:X_n\to Y_n$ is a projective (resp. mixed) $\F_{\U}^{G,\Sigma_n}$-fibration.
	\item $G^{\U}$-\emph{projective (resp. $G$-flat) cofibration} if for all $n\in \mathbb{N}$ the latching map \begin{equation*}
		\nu_n[f]:L_n[f] \to Y_n	\end{equation*}
	is a projective $\F_{\U}^{G,\Sigma_n}$-cofibration (resp. genuine $G$-cofibration).
\end{itemize}
\end{Def}
The class of $G^{\U}$-level equivalences (and that of projective $G^{\U}$-level fibrations) has a different description that motivates its definition:
\begin{Lemma}\label{lem:gleveleq}
Let $f:X\to Y$ be a morphism of $G$-symmetric spectra. Then the following are equivalent: \begin{enumerate}
\item	The map $f$ is a $G^{\U}$-level equivalence (resp. $G^{\U}$-projective level fibration).
\item For all subgroups $H$ of $G$ and all finite $H$-subsets $M\subset \U$ the $H$-map $f(M):X(M)\to Y(M)$
induces a weak homotopy-equivalence (resp. Serre/Kan fibration) on $H$-fixed points.
\end{enumerate}
\end{Lemma}
By applying it to all subgroups of $H$ with the restricted action on $M$, one can also replace the second condition by requiring $f(M)$ to be a genuine $H$-equivalence (resp. genuine $H$-fibration) for all finite $H$-subsets $M\subset \U$.
\begin{proof} Given an $H$-subset $M\subset \U$ of order $m$, any choice of bijection $M\cong \underline{m}$ defines a homomorphism $\varphi:H\to \Sigma_m$ and a natural isomorphism $X(M)^H\cong X_m^{\Gamma(\varphi)}$, where $\Gamma(\varphi)\leq G\times \Sigma_m$ is the graph of $\varphi$ which by assumption lies in $\F_{\U}^{G,\Sigma_n}$. This shows that $(1)$ implies $(2)$. The other direction is similar.
\end{proof}
In another common definition of level equivalence of $G$-spectra, for example in the translation of the model structure in \cite{MM02} to the symmetric context, one only requires $f(M)$ to be an equivalence on $H$-fixed points for $H$-sets $M$ that are restrictions of $G$-sets. This is more natural if one views $G$-symmetric spectra as enriched functor categories as explained in Section \ref{sec:free}. The stronger notion we use has the advantage that for a subgroup $H$ of $G$ the restriction functor from $G$-symmetric spectra to $H$-symmetric spectra preserves fibrations and weak equivalences and thus becomes a right Quillen functor. For $G$-orthogonal spectra, this stronger version is also used in \cite[Def. 2.3.5]{Sto11}; and its positive one appears in \cite[Sec. B.4.]{HHR16}.

Furthermore, we want to note the following convenient properties of cofibrations of $G$-symmetric spectra of simplicial sets:
\begin{Remark} \label{rem:flatdetect} Let $f:X\to Y$ be a  morphism of $G$-symmetric spectra of simplicial sets. Then the following hold: 
\begin{enumerate}[(i)]
 \item $f$ is a $G$-flat cofibration if and only if it is a non-equivariant flat cofibration.
 \item If $\U$ is a complete $G$-set universe (i.e., if every finite $G$-set can be embedded into it), then $f$ is a $G^{\U}$-projective cofibration if and only if it is a non-equivariant projective cofibration.
\end{enumerate}
\end{Remark}
We now proceed to proving the model category axioms. By definition, the classes of $G^{\U}$-projective/$G^{\U}$-flat cofibrations, fibrations and $G^{\U}$-level equivalences are formed out of the respective classes for the projective/mixed $\F_{\U}^{G,\Sigma_n}$-model structures in each dimension. We use a more general proposition that is - together with its proof - close to \cite[Appendix C]{Sch18}.

We first define:

\begin{Def}[Consistency condition] \label{def:consistency} For every $n\in \mathbb{N}$ let $\mathscr{M}_n$ be a model structure on based $(G\times \Sigma_n)$-spaces/$(G\times \Sigma_n)$-simplicial sets. The collection of model structures $\{\mathscr{M}_n\}_{n\in \mathbb{N}}$ is said to satisfy the \emph{consistency condition} if for all $m,n\in \mathbb{N}$ and every acyclic cofibration $f:A\to B$ in $\mathscr{M}_m$ the pushout of the $(G\times \Sigma_n)$-map $f\wedge_{\Sigma_m} \mathbf{\Sigma}(\underline{m},\underline{n})$ along an arbitrary $(G\times \Sigma_n)$-map $g:A\wedge_{\Sigma_m} \mathbf{\Sigma}(\underline{m},\underline{n})\to X$ is a weak equivalence in $\mathscr{M}_n$.
\end{Def}
In particular, the consistency condition holds if the functors 
\[ -\wedge_{\Sigma_m} \mathbf{\Sigma}(\underline{m},\underline{n}): (G\times \Sigma_m)\T_*\to (G\times \Sigma_{n})\T_* \] (or their simplicial analogs) are left Quillen functors, in which case we say that the $\mathscr{M}_n$ satisfy the \emph{strong consistency condition}.

The following proposition says that given the consistency condition is satisfied, the model structures $\mathscr{M}_n$ assemble to a level model structure on $G$-symmetric spectra.
\begin{Prop}[Level model structures] \label{prop:levmod} Let $\{\mathscr{M}_n\}_{n\in \mathbb{N}}$ be a collection of model structures on based $(G\times \Sigma_n)$-spaces or based $(G\times \Sigma_n)$-simplicial sets, satisfying the consistency condition. Call a map $f:X\to Y$ of $G$-symmetric spectra a level weak equivalence (level fibration) if for all $n\in \mathbb{N}$ the map $f_n$ is an $\mathscr{M}_n$-weak-equivalence ($\mathscr{M}_n$-fibration) and a \emph{level cofibration} if for all $n\in \mathbb{N}$ the latching morphism $\nu_n[f]:L_n[f]\to Y_n$ is an $\mathscr{M}_n$-cofibration.

Then these classes define a model structure on $G$-symmetric spectra. If all model structures $\mathscr{M}_n$ are right proper, then so is this model structure. If all model structures $\mathscr{M}_n$ are left proper and satisfy the strong consistency condition, then this model structure is also left proper.
\end{Prop}
The purpose of the consistency condition lies in the following:
\begin{Lemma}  \label{lem:acy} Let $f:X\to Y$ be a morphism of $G$-symmetric spectra such that all latching maps $\nu_n[f]$ are acyclic cofibrations in the respective model structures $\mathscr{M}_n$. Then $f$ is an acyclic cofibration of $G$-symmetric spectra.
\end{Lemma}
\begin{proof} We only have to show that $f$ is a weak equivalence since it is a cofibration by definition. As we saw in Section \ref{sec:latching}, each component $f_n$ is a finite composition of pushouts of \begin{eqnarray*} & (\mathscr{G}_m(\nu_m[f]))_n=\nu_m[f]\wedge_{\Sigma_m} \mathbf{\Sigma}(\underline{m},\underline{n})\end{eqnarray*}
for $m\leq n$. These maps are by the assumption on $f$ and the consistency condition weak equivalences, hence so is $f_n$.
\end{proof}
\begin{proof}[Proof of Proposition \ref{prop:levmod}] %\todo{shorten, maybe find a reference for how this is done? Hovey/Hirschhorn? maybe only for cofibrantly generated, and say it can be shown in general by hand}
$G$-symmetric spectra have all limits and colimits and these are constructed levelwise. It follows from the respective properties in the model structures $\mathscr{M}_n$ that weak equivalences possess the 2-out-of-3 property and that weak equivalences and fibrations are closed under retracts. The $n$-th latching map is a functor on the arrow category and thus turns retract diagrams into retract diagrams and so the retract of a cofibration is also one.

Using that the definition of the $n$-th latching space of a morphism only depends on the first $n$ levels of the source and the first $n-1$ levels of the target, we can inductively apply factorizations in the model categories $\mathscr{M}_n$ to obtain spectrum level factorizations of the form
\begin{enumerate}[(i)]
\item $f=p\circ i$, where $i$ is a cofibration and $p$ is an acylic fibration.
\item $f=p'\circ i'$, where $i'$ has the property that all latching maps $\nu_n[i']$ are acyclic cofibrations in the respective $\mathscr{M}_n$ and $p'$ is a fibration. In particular, by Lemma \ref{lem:acy}, such an $i'$ is an acyclic cofibration.
\end{enumerate}

The lifting properties then follow from Lemma \ref{lem:lift}, once we see that every acyclic cofibration $f$ has the property that each latching map $\nu_n[f]$ is an acyclic cofibration in $\mathscr{M}_n$. For this we factor $f=p'\circ i'$ as in item (ii) above. Since $i'$ and $f$ are weak equivalences, so is $p'$. Since $f$ is in particular a cofibration, it follows that $f$ is a retract of $i'$ (cf. \cite[Prop. 3.2.4]{HSS00}), which proves the claim.
\begin{comment}
It is a consequence of Lemma \ref{lem:lift} and the lifting properties in the model structures $\mathscr{M}_n$ that every cofibration has the left lifting property with respect to every acyclic fibration. The same lemma also implies that every map for which all latching maps are acyclic cofibrations has the left lifting property with respect to fibrations. We now show that every acyclic cofibration has this property (and thus together with Lemma \ref{lem:acy} these two classes agree). Let $f:X\to Y$ be an acyclic cofibration of $G$-symmetric spectra. As shown in the last paragraph, one can factor $f=p'i'$ with $p'$ a fibration and $i'$ with the property that all latching maps are acyclic cofibrations. We also know (cf. Lemma \ref{lem:acy}) that $i'$ is a weak equivalence, thus $p'$ is in fact an acyclic fibration. Since $f$ is in particular a cofibration, the commutative square \begin{eqnarray*}
	& \xymatrix{X\ar[r]^{i'}\ar[d]_f & A \ar[d]^{p'} \\
							Y \ar[r]_{id_Y}\ar@{-->}[ur] & Y}
\end{eqnarray*}
possesses a lift that establishes $f$ as a retract of $i'$. Thus, every latching map of $f$ is also a retract of that of $i'$ and therefore an acyclic cofibration. Therefore, $f$ has the left lifting property with respect to all fibrations and we have proved that all model category axioms are satisfied.
\end{comment}

Finally, the statement about right properness follows directly from the definition, while for left properness one first has to note that the strong consistency condition implies that all components~$f_n$ of a level cofibration $f$ are cofibrations in the respective model structure $\mathscr{M}_n$.
%, using the skeleton filtration of Section REF.
%This is a consequence of the strong consistency condition, since it guarantees that $(-)\wedge _{\Sigma_m}\mathbf{\Sigma}(\underline{m},\underline{n})$ preserves cofibrations. The $n$-th level of a cofibration $f$ is the finite composition of pushouts of $\nu_m[f]\wedge _{\Sigma_m}\mathbf{\Sigma}(\underline{m},\underline{n})$ for $m\leq n$ (cf. Section \ref{sec:latching}) and hence a cofibration in the model structure $\mathscr{M}_n$. Therefore, the resulting model structure on $G$-symmetric spectra is left proper, too, and the proposition is proven.
\end{proof}

%In order to show that the classes of $G^{\U}$-level equivalences, $G^{\U}$-projective/$G^{\U}$-flat level fibrations and $G^{\U}$-projective/$G$-flat cofibrations as defined at the beginning of the section assemble to a model structure on $G$-symmetric spectra, it now remains

Hence we need to check that the projective and mixed model structures on $(G\times \Sigma_n)\T_*$ and $(G\times \Sigma_n)\mathcal{S}_*$ satisfy the consistency condition. Since the two share the same weak equivalences, it suffices to prove:
\begin{Prop} The mixed model structures with respect to the families $\mathcal{F}_{\U}^{G,\Sigma_n}$ satisfy the strong consistency condition. \end{Prop}
\begin{proof} For $m>n$ the required condition holds trivially since in that case $\mathbf{\Sigma}(\underline{m},\underline{n})=*$. Otherwise, we write $n=m+k$ with $k\geq 0$. The map \begin{eqnarray*}
	\Sigma_{m+k} \ltimes_{\Sigma_k} S^k & \longrightarrow & \mathbf{\Sigma}(m,m+k) \\
	(\sigma,x) & \mapsto & (\sigma_{|m};(\sigma_{|k})_*(x))
\end{eqnarray*}
is a $(\Sigma_{m+k}\times \Sigma_m^{op})$-equivariant isomorphism. Thus, the functor $-\wedge_{\Sigma_m}\mathbf{\Sigma}(m,m+k)$ is naturally isomorphic to $\Sigma_{m+k} \ltimes_{\Sigma_m\times \Sigma_k} (-\wedge S^k)$. Hence, it sends genuine $(G\times \Sigma_m)$-cofibrations to genuine $(G\times \Sigma_{m+k})$-cofibrations.

For acyclic cofibrations, it is enough to show the case of topological spaces. It suffices to check the property on the set of generators $J_{G,mix}^{\mathcal{F}_{\U}^{G,\Sigma_m}}$ introduced in Section \ref{sec:fammod}. Every map in this set is a $\mathcal{F}_{\U}^{G,\Sigma_m}$-equivalence between $(G\times\Sigma_m)$-CW complexes. Hence the statement follows by putting $A=S^k$ in Lemma \ref{lem:tech1}.

Finally, since a map of $(G\times \Sigma_m)$-simplicial sets is an $\F _{\U}^{G,\Sigma_m}$-equivalence and genuine cofibration if and only if its geometric realization is, the simplicial case follows from the topological one.
\end{proof}
Hence, we obtain:
\begin{Cor} [Flat level model structure on $G$-symmetric spectra] The classes of $G$-flat cofibrations, $G^{\U}$-flat level fibrations and $G^{\U}$-level equivalences define a proper model structure on the category of $G$-symmetric spectra of spaces or simplicial sets.
\end{Cor}
As well as:
\begin{Cor} [Projective level model structure on $G$-symmetric spectra] The classes of $G^{\U}$-projective cofibrations, $G^{\U}$-projective level fibrations and $G^{\U}$-level equivalences define a proper model structure on the category of $G$-symmetric spectra of spaces or simplicial sets.
\end{Cor}
In the projective case the left properness does not directly follow from Proposition \ref{prop:levmod}, but it is a consequence of the left properness of the flat level model structure, since every $G^{\U}$-projective cofibration is also a $G$-flat cofibration and the weak equivalences are the same.
%\begin{Remark} We note that both the flat and the projective $\U$-level model structure only depend on the isomorphism type of $\U$ and not $\U$ itself.
%\end{Remark}

It follows from the adjunction between (semi-)free spectra and evaluation that the flat and the projective $\U$-level model structures on $G$-symmetric spectra are cofibrantly generated with
\begin{equation} \label{eq:gencofflat} I_{\U,fl}^{lev}\defeq \{ \mathscr{G}_n(i)\ |\ n\in \mathbb{N},\ i\in I_{G\times \Sigma_n}\}\end{equation} and \begin{equation} \label{eq:genacyflat} J_{\U,fl}^{lev}\defeq \{ \mathscr{G}_n(j)\ |\ n\in \mathbb{N},\ j\in J_{G\times \Sigma_n,mix}^{\mathcal{F}_{\U}^{G,\Sigma_n}}\}, \end{equation}
respectively
\begin{equation} \label{eq:gencofproj} I_{\U,proj}^{lev}\defeq \{ G\ltimes_H(\mathscr{F}^{(H)}_M(i))\ |\ H\leq G,\ M\subset_H \U,\ i\in I_H\} \end{equation} and \begin{equation} \label{eq:genacyproj} J_{\U,proj}^{lev}\defeq \{ G\ltimes_H (\mathscr{F}^{(H)}_M(j))\ |\ H\leq G,\ M\subset_H \U,\ j\in J_H\} \end{equation} as generating cofibrations and acyclic cofibrations.

\begin{Remark} \label{rem:freeflat} By adjunction, all free $G$-symmetric spectra $\mathscr{F}_M A$ are $G$-flat, provided that $A$ is cofibrant in the topological case. If $M$ is contained in $\U$ they are also $G^{\U}$-projective. Similarly, the semi-free $G$-spectra $\mathscr{G}_M A$ for cofibrant $(G\ltimes \Sigma_M)$-spaces $A$ are $G$-flat. Whether they are $G^{\U}$-projective depends on the isotropy of $A$ and $\U$. 
\end{Remark}

As in the non-equivariant case, in order to obtain model structures on the category of commutative $G$-symmetric ring spectra, we need a positive variant of these level model structures. It is obtained via Proposition \ref{prop:levmod} by replacing the genuine model structure $\mathscr{M}_0$ on based $G$-spaces/$G$-simplicial sets in level $0$ by the one where the cofibrations are the isomorphisms and the weak equivalences and fibrations are arbitrary maps.
%Concretely, this means that the positive cofibrations form a subclass of the cofibrations in the respective non-positive version, consisting of those where the $0$-th latching map is an isomorphism. The weak equivalences and the fibrations need to satisfy the same conditions as their non-positive analogs, except for the one in degree $0$.
It is immediate that these also satisfy the consistency condition and so we get:
\begin{Prop}[Positive level model structures] \label{prop:poslevmod} The classes of positive $G$-flat (positive $G^{\U}$-projective) cofibrations, positive $G^{\U}$-flat (resp. positive $G^{\U}$-projective) fibrations and positive $G^{\U}$-level equivalences form a proper cofibrantly generated model structure on the category of $G$-symmetric spectra.
\end{Prop}
One obtains generating (acyclic) cofibrations by leaving out the maps of the form $\mathscr{G}_{\emptyset}(-)$ and $\mathscr{F}_{\emptyset}(-)$ in the generating (acyclic) cofibrations of their non-positive analogs described above.

We further have the following, where $\square$ denotes the pushout product with respect to the smash product of $G$-symmetric spectra:
\begin{Prop}\label{prop:levmonoidal} Let $i:A\to B$ and $j:C\to D$ be two maps of $G$-symmetric spectra. Then the following hold:
\begin{itemize}
 \item If $i$ and $j$ are $G$-flat ($G^{\U}$-projective) cofibrations, then so is the pushout product $i\square j$. If either $i$ or $j$ is positive, then the pushout product is also positive.
 \item If $i$ is a $G$-flat cofibration and $G^{\U}$-level equivalence and $j$ is a $G$-flat cofibration, then $i\square j$ is also a $G^{\U}$-level equivalence.
 \item If $i$ is a positive $G$-flat cofibration and positive $G^{\U}$-level equivalence and $j$ a $G$-flat cofibration, then $i\square j$ is a $G^{\U}$-level equivalence.
\end{itemize} 
\end{Prop}
\begin{proof} It suffices to show each of these on generating (acyclic) cofibrations. The first one follows from the natural isomorphisms $\mathscr{F}_M(f)\square \mathscr{F}_N(g)\cong \mathscr{F}_{M\sqcup N} (f\square g)$ and $\mathscr{G}_n(f)\square \mathscr{G}_m(g)\cong \mathscr{G}_{m+n}(\Sigma_{n+m}\ltimes_{\Sigma_n\times \Sigma_m} (f\square g))$ and the fact that the respective model structures on equivariant spaces are closed symmetric monoidal. So do the other two in the projective case. In the flat case we need Lemma \ref{lem:tech1} in addition, which guarantees that $\Sigma_{n+m}\ltimes_{\Sigma_n\times \Sigma_m} (f\wedge X)$ is an $\F_{\U}^{G,\Sigma_{m+n}}$-equivalence if $X$ is a cofibrant $(G\times \Sigma_m)$-space. This implies that the pushout product $\Sigma_{n+m}\ltimes_{\Sigma_n\times \Sigma_m} (f\square g)$ is also one by 2-out-of-3.
\end{proof}
\begin{Cor} The non-positive $G^{\U}$-flat and $G^{\U}$-projective level model structures are monoidal with respect to the smash product of $G$-symmetric spectra. 
\end{Cor}
The positive ones are not quite monoidal because the unit $\mathbb{S}$ is not positively cofibrant and in general the map $\mathbb{S}^{+}\wedge X\to X$ is not a positive $G^{\U}$-level equivalence, even if $X$ is positive cofibrant. This problem will disappear in the stable model structures.

Since the suspension spectrum functors $G\T_*\to GSp^{\Sigma}_T$ and $G\mathcal{S}_*\to GSp^{\Sigma}_{\mathcal{S}}$ are left Quillen with respect to both the non-positive projective and flat model structures, Proposition \ref{prop:levmonoidal} implies that the positive and non-positive, flat and projective $\U$-model structures are in particular based $G$-topological (respectively based $G$-simplicial).

\subsection{$G^{\U}\Omega$-spectra and $G^{\U}$-stable equivalences}
\label{sec:gom}
We now move towards the stable model structure, beginning with the notion of an equivariant $\Omega$-spectrum:

\begin{Def}[$G^{\U}\Omega$-spectra] \label{def:gomega} A $G^{\U}$-projective level fibrant $G$-symmetric spectrum $X$ is called a \emph{$G^{\U}\Omega$-spectrum} if for all subgroups $H$ of $G$ and all finite $H$-subsets $M$ and $N$ of $\U$ the adjoint generalized structure map induces a weak equivalence \[
	(\widetilde{\sigma}_M^N)^H:X(M)^H\to \map_H(S^N,X(M\sqcup N))
\]
on $H$-fixed points. We say that a positively $G^{\U}$-projective level fibrant spectrum is a positive $G^{\U}\Omega$-spectrum if it satisfies the above condition except for the cases where $M=\emptyset$. 
\end{Def}

%Again we remark that there is another common notion of $\Omega$-spectrum for $G$-spectra in which one only requires the map $\widetilde{\sigma}_{M,N}^H:X(M)^H\to \map_H(S^N,X(M\sqcup N))$ to be a weak equivalence for $H$-sets $M$ and $N$ that are restrictions of $G$-sets. With this other definition the restriction of a $G^{\U}\Omega$-spectrum to a subgroup $H$ is not necessarily an $H^{\U}\Omega$-spectrum, which is the case for our notion.
\begin{Example} \label{exa:gomsh} Let $N$ be a finite $G$-set contained in $\U$ and $A$ a cofibrant $G$-space. Then the endofunctors $sh^N$ and $\map(A,-)$ preserve (positive) $G^{\U}\Omega$-spectra. \end{Example}
\begin{Cor} \label{cor:omegash}If $X$ is a $G^{\U}\Omega$-spectrum, then so is $\Omega^Nsh^NX$ for any finite $G$-set $N$ contained in $\U$ and the natural map $\widetilde{\alpha}_X^N:X\to \Omega^Nsh^NX$
is a $G^{\U}$-level equivalence.
\end{Cor}
%\begin{proof} \todo{leave out} Evaluated on a finite $H$-set $M\subseteq \U$ for some subgroup $H$ of $G$, the map $\widetilde{\alpha}_X^N$ is given by
%\[ X(M)\xrightarrow{\widetilde{\sigma}_M^N} \Omega^N X(M\sqcup N)\xr{\Omega^N X(\tau_{M,N})} \Omega^N X(N\sqcup M)=(\Omega^Nsh^NX)(M) \]
%and thus the composition of a genuine $H$-equivalence and an $H$-isomorphism.
%\end{proof}
\begin{Cor} \label{cor:homomega} For every $G^{\U}$-projective $G$-symmetric spectrum $X$ the functor $\Hom(X,-)$ preserves $G^{\U}\Omega$-spectra. If $X$ is $G$-flat, it preserves $G^{\U}$-flat level fibrant $G^{\U}\Omega$-spectra.
\end{Cor}
We can now give the definition of the class of stable equivalences. We denote by $\gamma:GSp^{\Sigma}\to \Ho_{lev}^{\U}$ the localization of $GSp^{\Sigma}$ at the class of $G^{\U}$-level equivalences.
\begin{Def}[$G^{\U}$-stable equivalence]	\label{def:gstableeq} A map $f:X\to Y$ of $G$-symmetric spectra is a \emph{$G^{\U}$-stable equivalence} if for all $G^{\U}\Omega$-spectra $Z$ the map 
\[	\Ho_{lev}^{\U}(\gamma(f),Z):\Ho_{lev}^{\U}(Y,Z)\to \Ho_{lev}^{\U}(X,Z) \]
is a bijection.
\end{Def}
Making use of the level model structures we constructed, one can also characterize $G^{\U}$-stable equivalences in the following way: Two maps of $G$-symmetric spectra $f,g:X\to Y$ are called \emph{$G$-homotopic} if they lie in the same path-component of the mapping space $\map_{GSp^\Sigma}(X,Y)$. The set of $G$-homotopy classes of $G$-maps is denoted $[X,Y]^G$. Then by general model category theory, a map $f:X\to Y$ is a $G^{\U}$-stable equivalence if and only if for any $G^{\U}$-flat level fibrant $G^{\U}\Omega$-spectrum $Z$ and $G$-flat replacement $f^{\flat}:X^{\flat}\to Y^{\flat}$ of $f$ the induced map \[ [f^{\flat},Z]^G:[Y^{\flat},Z]^G\to [X^{\flat},Z]^G \]
is a bijection.
%A similar characterization can be obtained via the $G^{\U}$-projective level model structure.
\begin{Remark} \label{rem:levelomega} Since $G^{\U}$-level equivalences are mapped to isomorphisms in the homotopy category, they are in particular $G^{\U}$-stable equivalences. On the other hand, the Yoneda lemma implies that every $G^{\U}$-stable equivalence between $G^{\U}\Omega$-spectra is already a $G^{\U}$-level equivalence.
\end{Remark}
\begin{Example} \label{exa:lambda} For a subgroup $H$ of $G$ and two finite $H$-subsets $M$ and $N$ of $\U$ we define \[ \lambda^{(H)}_{M,N}: \mathscr{F}^{(H)}_{M\sqcup N}S^N \to \mathscr{F}^{(H)}_M S^0 \]
to be the morphism of $H$-spectra adjoint to the embedding $S^N\hookrightarrow \mathbf{\Sigma}(M,M\sqcup N)=(\mathscr{F}^{(H)}_M S^0)(M\sqcup N)$ 
associated to the inclusion $M\hookrightarrow M\sqcup N$.

Then the induction $G\ltimes_H\lambda^{(H)}_{M,N}$ to $G$-symmetric spectra is a $G^{\U}$-stable equivalence.
\end{Example}
\begin{proof} This follows from the fact that domain and target of $G\ltimes_H\lambda^{(H)}_{M,N}$ are $G$-flat and that $G\ltimes_H\lambda^{(H)}_{M,N}$ corepresents the adjoint structure map $X(M)^H\to (\Omega^N X(M\sqcup N))^H$, cf. \cite[Ex. 3.1.10]{HSS00} for the non-equivariant analog.
%Let $Z$ be a $G^{\U}$-flat level fibrant $G^{\U}\Omega$-spectrum. The $H$-spectra $\mathscr{F}^{(H)}_{M\sqcup N}S^N$ and $\mathscr{F}^{(H)}_M S^0$ are $H$-flat (cf. Remark \ref{rem:freeflat}), thus their inductions to $G$-spectra are $G$-flat. So we do not need to replace them. We have the following chain of adjunction isomorphisms:
%\[ \xymatrix{ [G\ltimes_H\mathscr{F}^{(H)}_M S^0,Z]^G \ar[r]^-{\cong}\ar[d]_{[G\ltimes_H\lambda^{(H)}_{M,N},Z]^G} & [\mathscr{F}^{(H)}_M S^0,Z]^H \ar[r]^-{\cong}\ar[d]^{[\lambda^{(H)}_{M,N},Z]^H} &  \pi_0(Z(M)^H) \ar[d]^{\pi_0 ((\widetilde{\sigma}_M^N)^H)} \\								
%[G\ltimes_H\mathscr{F}^{(H)}_{M\sqcup N}S^N,Z]^G \ar[r]_-{\cong} & [\mathscr{F}^{(H)}_{M\sqcup N}S^N,Z]^H \ar[r]_-{\cong} & \pi_0(\map_H(S^N,Z(M\sqcup N)))} \]
%The right vertical map is a bijection since $Z$ is a $G^{\U}\Omega$-spectrum.
\end{proof}

The notion of $G^{\U}$-stable equivalence behaves well under geometric realization and singular complex:
\begin{Prop}[Relation between spaces and simplicial sets] \label{prop:ggeo}
Both geometric realization $|.|$ and singular complex $\mathcal{S}$ preserve and reflect $G^{\U}$-stable equivalences.
\end{Prop}
\begin{proof} This follows from the fact that the geometric realization and singular complex adjunction induces an equivalence on level homotopy categories and that a $G$-symmetric spectrum of spaces is a $G^{\U}\Omega$-spectrum if and only if its singular complex is.
\end{proof}

\section{Naive homotopy groups, $G^{\U}$-semistability and the action of $\M_G$}
\label{sec:homgroups}
In this section we deal with ``naive'' equivariant homotopy groups of $G$-symmetric spectra. This part is the main complication in the theory of $G$-symmetric spectra, because unlike for $G$-orthogonal spectra a $G^{\U}$-stable equivalence must not necessarily induce isomorphisms on these, hence the name ``naive''. Nevertheless they are useful for two reasons: Firstly, the converse is true, i.e., every map inducing an isomorphism on naive homotopy groups is a $G^{\U}$-stable equivalence (Theorem \ref{theo:piiso}) and for various morphisms it is easier to show that they induce such an isomorphism than to check the rather abstract condition in the definition of a $G^{\U}$-stable equivalence.

Secondly, there is a large class  of $G$-symmetric spectra, called $G^{\U}$-semistable, on which the two notions of equivalence agree. Several equivalent characterizations of semistability for non-equivariant symmetric spectra were given in \cite[Sec. 5.6]{HSS00}. This theory was developed further in \cite{Sch08} (and later in \cite{HH14}, which deals with semistability in the context of $\mathbb{A}^1$-homotopy theory), where it is shown that the naive homotopy groups carry a natural ``tame'' action of the monoid of injective self-maps of the natural numbers and that a symmetric spectrum is semistable if and only if this action is trivial. These criteria have equivariant analogs, in particular there exists a tame action of the monoid of $G$-equivariant injective self-maps of the chosen $G$-set universe $\U$ on the naive homotopy groups of $G$-symmetric spectra which detects whether the naive homotopy groups are isomorphic to the derived ones (Corollary \ref{cor:semicrit}). The notion of tameness and the 
relevant 
algebraic properties of such actions are given in Section \ref{sec:monoid}.

Moreover, the naive equivariant homotopy groups of $G$-symmetric spectra naturally form a $G^{\U}$-Mackey functor, like those of $G$-orthogonal spectra. This $G^{\U}$-Mackey functor structure is compatible with the above monoid action.

\subsection{Definition}
We begin by defining the naive equivariant homotopy groups of a $G$-symmetric spectrum and fix a $G$-set universe $\U$. We denote by $s_G(\U)$ the poset of finite $G$-subsets of $\U$, partially ordered by inclusion. 

\begin{Def} Let $n$ be an integer and $H\leq G$ a subgroup. The $n$-th $H$-equivariant homotopy group $\pi_n^{H,\U}X$ of a $G$-symmetric spectrum of spaces $X$ (with respect to $\U$) is defined as
\[ \pi_n^{H,\U}X \defeq \colim_{M\in s_G(\U)}[S^{n\sqcup M},X(M)]^H. \]
The connecting maps in the colimit system are given by the composites
\[ [S^{n\sqcup M},X(M)]^H\xr{(-)\wedge S^{N-M}} [S^{n\sqcup M\sqcup (N-M)},X(M)\wedge S^{N-M}]^H\xr{(\sigma_M^{N-M})_*}  [S^{n\sqcup N},X(N)]^H
\]
for every inclusion $M\subseteq N$. The last step implicitly uses the homeomorphism $X(M\sqcup (N-M))\cong X(N)$ induced from the canonical isomorphism $M\sqcup (N-M)\cong N$.
\end{Def}
To clarify what this means for negative $n$ we choose an isometric $G$-embedding $i:\R^{\infty}\hookrightarrow (\R[\U])^G$ and only index the colimit system over those elements $M$ of $s_G(\U)$ such that $\R^M$ contains $i(R^{-n})$, in which case the corresponding term is given by $[S^{M-i(\R^{-n})},X(M)]^H$, the expression $M-i(R^{-n})$ denoting the orthogonal complement of $i(\R^{-n})$ in $\R^M$. Since any two such embeddings are connected by a homotopy (in fact the space of embeddings is contractible) the definition only depends on this choice up to \emph{canonical} isomorphism and so we leave it out of the notation.

If the number of $G$-orbits of $M$ is larger than $-n$ by at least two, the permutation sphere $S^{n\sqcup M}$ has at least two trivial coordinates and hence the set $[S^{n\sqcup M},X(M)]^H$ has a natural abelian group structure. Since such $G$-sets are cofinal in $s_G(\U)$, this turns $\pi_0^{H,\U}X$ into an abelian group. For a $G$-symmetric spectrum of simplicial sets $A$ we set $\pi_n^{H,\U} A\defeq \pi_n^{H,\U} |A|.$

\begin{Example}[Homotopy groups of $G^{\U}\Omega$-spectra] \label{exa:piomega} For every (positive) $G^{\U}\Omega$-spectrum of spaces~$X$, every integer $n$, every subgroup $H$ of $G$ and every finite (non-empty) $G$-set $M\subset \U$, for which $i(\R^{-n})\subseteq \R[M]$ if $n$ is negative, the induced map
$[S^{n\sqcup M},X(M)]^H\to \pi_n^{H,\U}X$ is a bijection.
\end{Example}
\begin{proof} The connecting maps in the colimit system can also be described as first postcomposing with the adjoint structure map $X(M)\to \Omega^{N-M} X(N)$ and then adjoining $S^{N-M}$ to the left. Representation spheres allow the structure of $G$-CW complexes, so in the case of a (positive) $G^{\U}\Omega$-spectrum these connecting maps are (almost) all isomorphisms, which proves the claim.
\begin{comment}
This follows from the fact that by adjunction the maps in the colimit system are equal to the composite
\[[S^{n\sqcup M},X(M)]^H\xr{(\widetilde{\sigma}_M^{N-M})_*} [S^{n\sqcup M},\Omega^{M-N}X(M\sqcup (N-M)))]^H\xr{\cong}  [S^{n\sqcup N},X(N)]^H. \]
For (positive) $G^{\U}\Omega$-spectra the map $\widetilde{\sigma}_M^{N-M}$ is a genuine $G$-equivalence, hence by the Whitehead theorem (Proposition \ref{prop:whitehead}) it induces bijections on homotopy classes of $G$-maps out of $G$-CW complexes and thus in particular out of representation spheres. So all the terms above the one for $M$ in the (directed) colimit system are bijections and it follows that the map to the colimit is an isomorphism.
\end{comment}
\end{proof}
\begin{Def}[$\upi_*^{\U}$-isomorphism] A map $f:X\to Y$ of $G$-symmetric spectra is called a \emph{$\upi_*^{\U}$-isomorphism} if for every subgroup $H$ of $G$ and every integer $n\in \mathbb{Z}$ the induced map
\[ \pi_n^{H,\U}(f):\pi_n^{H,\U}(X)\to \pi_n^{H,\U}(Y) \]
is an isomorphism.
\end{Def}
A (positive) $G^{\U}$-level equivalence induces a bijection on (almost) all terms in the colimit system and so:
\begin{Lemma}
Every (positive) $G^{\U}$-level equivalence is a $\upi_*^{\U}$-isomorphism.
\end{Lemma}
%\begin{Remark}\label{rem:orthhomgroups} The $n$-th homotopy group $\pi_n^{H,\R[\U]}X$ of a $G$-orthogonal spectrum $X$ with respect to the $G$-representation universe $\R[\U]$ is defined similarly by forming a colimit over $[S^{n+V},X(V)]^H$ for all finite dimensional $G$-subrepresentations of $\R[\U]$. Since subrepresentations of the form $\R[M]$ for a finite $G$-subset $M$ of $\U$ are cofinal in that system, Lemma \ref{lem:eva} implies that $\pi_n^{H,\R[\U]}X$ is naturally isomorphic to $\pi_n^{H,\U}(UX)$.
%\end{Remark}
%\todo{put later}

\subsection{Suspension isomorphism, long exact sequences and the Wirthm{\"u}ller isomorphism}
In this section we explain how naive homotopy groups behave under various constructions. We begin with the suspension isomorphism:

\begin{Prop} \label{prop:suspiso} For every $G$-symmetric spectrum of spaces $X$ and every finite dimensional $G$-subrepresentation $V$ of $\R[\U]$ the adjunction unit $ X\xr{\eta_X} \Omega^V (S^V\wedge X)$ and counit $\epsilon_X:S^V\wedge (\Omega^V X)\to X$ are $\upi_*^{\U}$-isomorphisms.
\end{Prop}
\begin{proof}
The proof is similar to its non-equivariant analog, cf. \cite[Ex. I.2.16]{Sch07}, so we will be brief. In the case of the unit, by adjunction, we have to prove that the suspension maps $S^V\wedge (-):[S^{n\sqcup M},X(M)]^G\to [S^V\wedge S^{n\sqcup M},S^V\wedge X(M)]^G$ induce an isomorphism on the colimit over $s_G(\U)$ (and similarly for all subgroups $H$ of $G$). Injectivity is obvious, since smashing with representation spheres is part of the connecting maps in the colimit system.
%To check injectivity, we take an element in the kernel, represented by a $G$-map $f:S^{n\sqcup M}\to X(M)$. Replacing $f$ by a representative higher up in the colimit system if necessary, we can assume that the suspension $S^V\wedge f$ is already $G$-nullhomotopic. By assumption, there is a finite $G$-set $N$ contained in $\U$ (and disjoint from $M$) and an embedding $j:V\to \R^N$. Then it follows that $f\wedge S^N= f\wedge S^{j(V)}\wedge S^{n-j(V)}$ is also $G$-nullhomotopic and hence $f$ already represents the zero class in the domain colimit.
For surjectivity we let $M$ be a finite $G$-subset of $\U$ and $f:S^V\wedge S^{n\sqcup M}\to S^V\wedge X(M)$ a $G$-map. We pick a finite $G$-subset $N\subseteq \U$ (disjoint to $M$) and an embedding $j:V\to \mathbb{R}[N]$ and denote by $g$ the $G$-map $S^{n\sqcup M}\wedge S^V\to X(M)\wedge S^V$ obtained by pre- and postcomposing $f$ with the symmetry isomorphisms shifting $S^V$ to the correct position. Then the map $S^V\wedge g$ differs from $f\wedge S^V$ (as maps $S^V\wedge S^{n\sqcup M}\wedge S^V\to S^V\wedge X(M)\wedge S^V$) by pre- and postcomposing with the twist automorphism of $S^V\wedge S^V$. In general, these twists do not cancel each other out and so $S^V\wedge g$ is not always $G$-homotopic to $f\wedge S^V$, but they do become so after smashing with another copy of $S^V\wedge S^V$ (since any two transpositions in $\Sigma_4$ act $G$-homotopically on $(S^V)^{\wedge 4}$, so both twists can be replaced with the interchanging of the two new $S^V$ factors, which do cancel). Thus the element represented by \[ \sigma_M^N\circ (g\wedge S^{N-j(V)}):S^{n\sqcup M\sqcup N}=S^{n\sqcup M}\wedge S^V\wedge S^{N-j(V)}\to X(M\sqcup N)\] is an inverse image of the class of $f$, and so the induced map is surjective

%As an example of an argument, we explain how to obtain an inverse image under $(\nu_X)_*$ of an element $...$. For this we choose a finite $G$-set $N$ such that $V$ is contained in $\mathbb{R}[N]$, and let the element be represented by a $G$-map $f:S^M\to \Omega^V S^V\wedge X(M)$

%We choose a finite $G$-set $N$ such that $V$ is contained in $\mathbb{R}[N]$. Let ... be represented by a $G$-map $S^M\to \Omega^V (S^V\wedge X(M))$... adjoint $S^V\wedge S^M\to S^V\wedge X(M)$ .... $S^{M\sqcup N}\cong S^M\wedge S^V\wedge S^{N-V}\to X(M)\wedge S^N\to X(M\sqcup N)$... The composite $\eta_X\circ f'$ is the same as .... after suspension by $S^V\wedge S^V$ .. while other composite is strictly the same?

The proof for the counit is similar.
\end{proof}
Since for any $G$-set universe $\U$ the linearization $\R[\U]$ always allows an embedding of the trivial $\R^{\infty}$ (even if no trivial $G$-sets embed into $\U$), the above proposition in particular applies for all trivial representations. By adjunction, the groups $\pi_n^{H,\U}(\Omega (S^1\wedge X))$ can be naturally identified with $\pi_{n+1}^{H,\U} (S^1\wedge X)$ and so we see that there is a natural isomorphism $\pi_n^{H,\U}X\cong \pi_{n+1}^{H,\U} (S^1\wedge X)$ for all subgroups $H$ of $G$ and $n\in \mathbb{Z}$.

One uses this to construct long exact sequences in naive homotopy groups. The mapping cone $C(f)$ and the homotopy fiber $H(f)$ of a map $f:X\to Y$ of $G$-symmetric spectra (as well as the associated natural maps $i(f):Y\to C(f)$, $q(f):C(f)\to S^1\wedge X$, $j(f):\Omega Y\to H(f)$ and $p(f):H(f)\to X$) are defined levelwise.
\begin{Prop} \label{prop:exseq} The following hold:
\begin{enumerate}
\item For every map $f:X\to Y$ of $G$-symmetric spectra of spaces and all subgroups $H$ of $G$ the sequences
 \[ \hdots\to \pi_n^{H,\U}X\xr{\pi_n^{H,\U}f} \pi_n^{H,\U}Y \xr{\pi_n^{H,\U}i(f)} \pi_n^{H,\U}C(f)\xr{\pi_n^{H,\U}q(f)} \pi_n^{H,\U} (S^1\wedge X)\cong \pi_{n-1}^{H,\U} X\to\hdots \]
and
\[ \hdots \to \pi_{n+1}^{H,\U} Y\cong \pi_n^{H,\U}(\Omega Y)\xr{\pi_n^{H,\U}j(f)} \pi_n^{H,\U} H(f)\xr{\pi_n^{H,\U}p(f)} \pi_n^{H,\U} X\xr{\pi_n^{H,\U} f} \pi_n^{H,\U} Y \to \hdots \]
are exact.
\item Let $f:X\to Y$ be a morphism of $G$-symmetric spectra of spaces. Then the natural map $h:S^1\wedge H(f)\to C(f)$ is a $\upi_*^{\U}$-isomorphism.
\item Let $\{X_i\}_{i\in I}$ be a family of $G$-symmetric spectra. Then for every subgroup $H$ of $G$ the natural map 
$\bigoplus_{i\in I} \pi_*^{H,\U}X_i\to \pi_*^{H,\U}\Big(\ \hspace{-4pt}\bigvee_{i\in I}X_i\Big )$ is an isomorphism. Furthermore, for finite $I$ the natural map $ \pi_*^{H,\U}\Big(\ \hspace{-4pt}\prod_{i\in I} X_i\Big )\to \prod_{i\in I} \pi_*^{H,\U} X_i$ is an isomorphism and hence the canonical morphism $\bigvee_{i\in I} X_i\to \prod_{i\in I} X_i$ is a $\upi_*^{\U}$-isomorphism.
\end{enumerate}
\end{Prop}
\begin{proof} The proof is very similar to the one for the analogous statements in non-equivariant symmetric spectra or orthogonal spectra, cf. \cite[Prop. I.4.7 and Cor. I.4.9]{Sch07} and \cite[Thm. 7.4]{MMSS01}.
\end{proof}

Now we come to the Wirthm{\"u}ller isomorphism, which states that for a subgroup inclusion $H\leq G$ the natural map $\gamma$ from induction to coinduction (cf. Section \ref{sec:unstable}) is a $\upi^{\U}_*$-isomorphism, provided that $G/H$ allows an embedding into $\R[\U]$. It was first proved for suspension spectra in \cite{Wir74}, generalized to all $H$-spectra in \cite[Section II.6]{LMS86} and reproved in a different way in \cite{May03}. The statement in \cite{LMS86} and \cite{May03} is about the derived natural transformation in the $G$-equivariant stable homotopy category. In the case of a complete $G$-universe an underived version is given in \cite[Sec. 4]{Sch11}, the main point being that one does not have to replace the $H$-orthogonal spectrum $X$ cofibrantly for the isomorphism to hold.

The version for $G$-symmetric spectra we present here is even more underived, as it involves naive homotopy groups:
\begin{Prop}[Wirthm{\"u}ller isomorphism] \label{prop:wirth} Let $H\leq G$ be a subgroup inclusion such that $G/H$ allows an embedding into $\R[\U]$ and $X$ an $H$-symmetric spectrum of spaces or simplicial sets. Then the natural map $\gamma_X:G\ltimes_H X\to \map_H(G,X)$ is a $\upi_*^{\U}$-isomorphism.
\end{Prop}
\begin{proof} I learned this proof from Stefan Schwede. We first prove the isomorphism in the case where all evaluations of $X$ at $H$-representations are cofibrant $H$-spaces, which already implies the simplicial version. Under this assumption the morphism $\gamma_Y$ is levelwise an equivariant cofibration, and it suffices to show that the strict quotient $Q=\map_H(G,X)/G\ltimes_H(X)$ has trivial homotopy groups. Evaluated at a finite $G$-set $M$, this quotient is given by $N_H^G(X(M))$, i.e., the space-level norm with respect to the smash product, in the sense of Section \ref{sec:norm}. Furthermore, given another finite $G$-set $N$, the structure map factors as
\[ N_H^G(X(M))\wedge S^N\xr{id\wedge \Delta_N} N_H^G(X(M))\wedge N_H^G(S^N)\cong N_H^G(X(M)\wedge S^N)\xr{N_H^G(\sigma_M^N)} N_H^G(X(M\sqcup N)), \]
where the map $\Delta_N:S^N\to N_H^G(S^N)$ is the diagonal. We claim that if $\R[N]$ allows an embedding of $G/H$, then the diagonal $\Delta_N$ is $G$-equivariantly nullhomotopic. This implies that for any subgroup $K$ of $G$, infinitely many connecting maps in the colimit defining $\pi_*^{K,\U}(Q)$ are zero, and hence the colimit is trivial. To see the claim, note that $N_H^G(S^N)$ is $G$-homeomorphic to the sphere associated to the induced representation $G\ltimes_H \R[N]$. The $G$-fixed subspace of $G\ltimes_H \R[N]$ is given precisely by the image of the $H$-fixed subspace of $\R[N]$ under the diagonal. By assumption, there exists a $v\in \R[N]$ which is $H$-fixed but not $G$-fixed. This implies that $\Delta_N(v)$ is a $G$-fixed point which does not come from a $G$-fixed point of $\R[N]$, and hence the inclusion of representation spheres is equivariantly nullhomotopic.

To obtain the result for arbitrary $H$-symmetric spectra it suffices to show that a $G$-flat replacement $X^{\flat}\to X$ induces $\upi_*^{\U}$-isomorphisms $\map_H(G,X^{\flat})\to \map_H(G,X)$ and $G\ltimes_HX^{\flat}\to G\ltimes_H X$. The prior follows from the natural isomorphism $\pi_*^{G,\U}(\map_H(G,X))\cong \pi_*^{H,\U}X$ and the double coset decomposition of $\res^G_K(\map_H(G,X))$, so it remains to show the latter. Hence the claim follows by the following lemma,  for which it is not necessary to require that $G/H$ embeds into $\R[\U]$.
\end{proof}

\begin{Lemma} \label{lem:wirth2} The functor $G\ltimes_H -$ maps $\upi^{\U}_*$-isomorphisms of $H$-symmetric spectra to $\upi_*^{\U}$-isomorphisms of $G$-symmetric spectra.
\end{Lemma}
\begin{proof}[Proof] Since $G\ltimes_H -$ preserves cofiber sequences, the long exact sequence of Proposition \ref{prop:exseq} implies that it suffices to show that if an $H$-symmetric spectrum $X$ has trivial homotopy groups, then so does $G\ltimes _H X$. We prove this by induction on the order of $G$, the induction start for the trivial group being clear. Hence we take a finite group $G$, assume the statement shown for all proper subgroups and fix an $H$-symmetric spectrum $X$ with trivial homotopy groups. Now from the double coset decomposition of $\res^G_K(G\ltimes_H X)$ and the formula for the homotopy groups of a wedge we see that all the groups $\pi_n^{K,\U}(G\ltimes_H X)$, where $K$ is a proper subgroup, are trivial by the induction hypothesis. So let $f:S^{n\sqcup M}\to G\ltimes_H X(M)$ be a $G$-map, we have to show that it represents the trivial element in $\pi_n^{G,\U}(G\ltimes_H X)$. If $H$ is equal to $G$ the statement of the lemma is trivial, so we can assume 
this not to be the case. But then the $G$-fixed points of 
$G\ltimes_H X(M)$ only consist of the basepoint and hence $f$ factors through a map $\widetilde{f}:S^{n\sqcup M}/(S^{n\sqcup M})^G\to G\ltimes_H X(M)$. Now the domain of $\widetilde{f}$ is a finite based $G$-CW complex with all cells induced, and an induction over the cells shows that any such map into a level of $G\ltimes_H X$ is stably trivial, since the groups $\pi_*^{K,\U} (G\ltimes_HX)$ vanish for proper subgroups $K$. This finishes the proof.
\end{proof}

We obtain a corollary:
\begin{Cor} \label{cor:smashpi} Let $\mathcal{F}$ be a family of subgroups of $G$, $f:X\to Y$ a morphism of $G$-symmetric spectra which induces an isomorphism on $\pi_n^{H,\U}$ for all $n\in \mathbb{Z}$ and all $H$ in $\mathcal{F}$ and $A$ a cofibrant $G$-space with non-basepoint isotropy contained in $\mathcal{F}$. Then
\begin{enumerate}[(i)]
\item $A\wedge f:A\wedge X\to A\wedge Y$ is a $\upi_*^{\U}$-isomorphism.
\item If $A$ is finite and $X$ and $Y$ are $G^{\U}$-projective level fibrant, $\map(A,f):\map(A,X)\to \map(A,Y)$ is a $\upi_*^{\U}$-isomorphism.
\end{enumerate}
\end{Cor}
\begin{proof} This follows by an induction over the cells of $A$, using the natural isomorphisms $G/H_+\wedge X\cong G\ltimes_H {\res}^G_H X$ and $\map(G/H_+,X)\cong \map_H(G,{\res}_H^G X)$ together with Lemma \ref{lem:wirth2} and the preceding paragraph.
%We first show the statements for finite $I_G$-cell complexes $A$. By an induction over the cells, making use of the long exact sequence for the mapping cone and homotopy fiber, both statements reduce to the case where $A$ is of the form $G/H_+\wedge S^k$ for some $k\in \mathbb{N}$ and $H$ in $\mathcal{F}$. Smashing with/mapping out of $S^k$ only shifts the homotopy groups, so it remains to show the case $A=G/H_+$. Since there are natural isomorphisms $G/H_+\wedge X\cong G\ltimes_H {\res}^G_H X$ and $\map(G/H,X)\cong \map_H(G,{\res}_H^G X)$, this follows from Lemma \ref{lem:wirth2} and the preceding paragraph.
%In order to obtain statement $(i)$ also for infinite $A$ one uses that sequential colimits along $G$-flat cofibrations preserve $\upi_*^{\U}$-isomorphisms. Finally, the statements follow for retracts of $I_G$-cell complexes by functoriality.
\end{proof}
In particular, taking $\mathcal{F}$ to be the family of all subgroups, we see that smashing with any cofibrant $G$-space preserves all $\upi^{\U}_*$-isomorphisms, and so does $\map(A,-)$ for finite $A$ (with the fibrancy assumption above in the simplicial case).

Finally we obtain the following extension of Proposition \ref{prop:suspiso}:
\begin{Prop} \label{prop:suspiso2} For every $G$-subrepresentation $V$ of $\R[\U]$ the functors $S^V\wedge (-)$ and $\Omega^V$ preserve and reflect $\upi_*^{\U}$-isomorphisms of $G$-symmetric spectra of spaces. Furthermore, a map $f:S^V\wedge X\to Y$ is a $\upi_*^{\U}$-isomorphism if and only if its adjoint $\widetilde{f}:X\to \Omega^VY$ is.
\end{Prop}
\begin{proof} Since representation spheres allow the structure of a finite $G$-CW complex, this is a formal consequence of Proposition \ref{prop:suspiso} and Corollary \ref{cor:smashpi}.
%we already know that $S^V\wedge -$ and $\Omega^V$ preserve $\upi_*^{\U}$-isomorphisms by the previous corollary. Both the unit and counit of the adjunction $(S^V\wedge (-),\Omega^V)$ are $\upi_*^{\U}$-isomorphisms, so we see that both functors also reflect $\upi_*^{\U}$-isomorphisms. Finally, the adjoint $\widetilde{f}$ of a map $f:S^V\wedge X\to Y$ is given by the composite $X\xr{\eta_X}\Omega^V(S^V\wedge X)\xr{\Omega^V(f)}Y$ and so it is a $\upi_*^{\U}$-isomorphism if and only if $f$ is.
\end{proof}

\subsection{The monoid ring $\M_G$}
\label{sec:monoid}
The universe $\U$ is kept in the notation of $\pi_*^{G,\U}$ for two reasons: For once, as one would expect and as is also the case for $G$-orthogonal spectra, these homotopy groups depend on the isomorphism type of $\U$, for example taking $\U$ to be a trivial universe will usually lead to different homotopy groups than for a complete universe. Secondly, if two universes $\U,\U'$ are isomorphic, any such isomorphism $\varphi:\U\cong \U'$ induces a natural isomorphism $\varphi_*:\pi_n^{G,\U}\cong \pi_n^{G,\U'}$, but this isomorphism does depend on the chosen $\varphi$. (We note, however, that the notion of $\upi_*^{\U}$-isomorphism only depends on the isomorphism type of $\U$.) This phenomenon, which already occurs for non-equivariant symmetric spectra, is not present for $G$-orthogonal spectra. It is encoded in an action of the monoid ring $\M_G$, for which we now introduce the relevant algebraic background. Everything here is a rather straightforward generalization of \cite{Sch08}.

\begin{Def}[Monoid of equivariant self-injections] The set $\Inj_G(\U,\U)$ of $G$-equivariant injective self-maps of $\U$ forms a monoid under composition. We denote the associated monoid ring $\mathbb{Z}[\Inj_G(\U,\U)]$ by $\M_G$. 
\end{Def}
Let $M$ be a finite $G$-subset of $\U$ and $A$ an $\M_G$-module. Then we say that an element $a$ of $A$ is of \emph{filtration $M$} if every injection $\psi$ which leaves $M$ pointwise fixed satisfies $\psi\cdot a=a$.

\begin{Def}[Tame modules] An $\M_G$-module $A$ is called \emph{tame} if every element $a\in A$ is of filtration $M$ for some finite $G$-subset $M$ of $\U$. 
\end{Def}

Tame modules have the following property:
\begin{Lemma} \label{lem:injective} Let $A$ be a tame $\M_G$-module. Then the following hold:
\begin{enumerate}[(i)]
 \item If two injections $\psi,\psi':\U\to \U$ agree on a finite $G$-subset $M$ of $\U$, then $\psi\cdot a=\psi'\cdot a$ for every element $a\in A$ of filtration $M$.
 \item Every element $\psi\in \Inj_G(\U,\U)$ acts injectively on $A$.
\end{enumerate}
\end{Lemma}
\begin{proof}$(i)$: Let $\alpha:\U\cong \U$ be a $G$-bijection which agrees with $\psi$ and $\psi'$ on $M$. Then $\alpha^{-1}\circ \psi$ and $\alpha^{-1}\circ \psi'$ restrict to the identity on $M$ and thus
\[ \psi\cdot a=\alpha\cdot ((\alpha^{-1}\circ \psi)\cdot  a)=\alpha \cdot a = \alpha ((\alpha^{-1}\circ \psi')\cdot  a)=\psi' \cdot a \]
for every element $a$ of filtration $M$.

Regarding $(ii)$, let $\psi$ be an injective self-map of $\U$ and $a$ an element of $A$ such that $\psi\cdot a=0$. Since $A$ is tame, the element $a$ is of filtration $M$ for some finite $G$-subset $M\in \U$. Let $\alpha:\U\to \U$ be a $G$-bijection which agrees with $\psi$ on $M$. Then, by $(i)$, we have $\alpha \cdot a=0$ and thus $a=\alpha^{-1}\cdot (\alpha \cdot a)=\alpha^{-1} \cdot 0=0$ and hence $\psi$ acts injectively. 
\end{proof}
%An example of an $\M_G$-module which is not tame is the free module of rank one. A class of non-trivial examples of tame $\M_G$-modules is given by the following:
%\begin{Example} \todo{leave out?} \label{exa:proj} Let $M$ be a finite $G$-set. Then the abelian group $\mathcal{P}(M,\U)=\mathbb{Z}[\Inj_G(M,\U)]$ (which is non-trivial if and only if $M$ allows an embedding into $\U$) becomes an $\M_G$-module via postcomposition. Since every injection takes image in a finite $G$-subset of $\U$, it is tame. 
%\end{Example}

All $G$-embeddings $\varphi:\U\hookrightarrow \U$ give rise to a conjugation ring homomorphism $c_{\varphi}:\M_G\to \M_G$, even if they are not surjective:
\begin{Def}[Conjugation] \label{def:conjugation} The $\varphi$-conjugate of a $G$-embedding $f:\U\hookrightarrow \U$ is defined as:
 \[ c_{\varphi}(f)(x)=\begin{cases} \varphi(f(\varphi^{-1}(x))) & \text{ if } x\in \im(\varphi) \\
 																		x & \text{ if } x\notin \im(\varphi) \end{cases} \]
Any preimage $\varphi^{-1}(x)$ is unique if it exists, so the above is well-defined. Given an $\M_G$-module $A$ we denote by $c_{\varphi}^* A$ the same abelian group with $\M_G$-action pulled back along $c_{\varphi}$. With this definition the map $A\xr{\varphi \cdot} c_{\varphi}^*A$ becomes $\M_G$-equivariant.
\end{Def}

\begin{Def}[Shift] \label{def:algshift} Let $M$ be a finite $G$-subset of $\U$. Then a \emph{shift by $M$} is a $G$-embedding $d^M:\U\hookrightarrow \U$ with image the complement $(\U-M)$ of $M$.
\end{Def}
The behavior of these shifts along finite $G$-sets is central to the theory of semistability (cf. Section \ref{sec:shifthom}). The non-equivariant prototype is the map $d:\N \to \N$ from \cite[Lem. 2.1, (iii)]{Sch08} which sends $i$ to $i+1$. Equivariantly, shifts in every isotypical direction contained in $\U$ need to be considered. Since we prefer to work coordinate-free there is no canonical ``shifting back by one copy of $M$'', leading to the ambiguity in the definition above. However, any two shifts by $M$ only differ by precomposition with a unique $G$-automorphism.
%\begin{Remark} \todo{maybe leave out} \label{rem:comp} Let $M$ and $N$ be finite $G$-subsets of $\U$, $d^M$ a shift by $M$ and $d^N$ a shift by $N$. Then the composition $d^M\circ d^N$ is a shift by $M\sqcup d^M(N)$.
%\end{Remark}

By iteration one obtains a shift along the whole universe $\U$. For this we choose a sequence $\emptyset=M_0\subseteq M_1 \subseteq \hdots \subseteq \U$ of finite $G$-subsets of $\U$ whose union equals $\U$, along with an $(M_n-M_{n-1})$-shift $d^{M_n-M_{n-1}}$ leaving $M_{n-1}$ fixed (for all $n\in \N$). Then the composite $d^{M_n}\defeq d^{M_n-M_{n-1}}\circ d^{M_{n-1}-M_{n-2}}\circ \hdots \circ d^{M_1}$ is a shift by $M_n$. Given an $\M_G$-module $A$ we define a new $\M_G$-module $c_{d^{\U}}^*A$ as the colimit of the sequence
\[ A\xr{d^{M_1}\cdot } c_{d^{M_1}}^*A\xr{d^{M_2-M_1}\cdot } c_{d^{M_2}}^*A\to \hdots \to c_{d^{M_n}}^*A\to \hdots \]
and denote the induced map $A\to c_{d^\U}^*A$ by $d^{\U}$.

We then have:
\begin{Prop}[Criteria for triviality] \label{prop:criteria} For a tame $\M_G$-module $A$ the following are equivalent:
\begin{enumerate}[(i)]
 \item The $\M_G$-action on $A$ is trivial.
 \item All shifts along transitive $G$-subsets of $\U$ act surjectively on $A$.
 \item All shifts along finite $G$-subsets of $\U$ act surjectively on $A$.
 \item The map $d^\U:A\to c_{d^\U}^*A$ above is surjective.
 \item There exists a finite $G$-subset $M$ of $\U$ such that every element of $A$ is of filtration $M$.
\end{enumerate}
\end{Prop}
\begin{proof} For a (possibly countably infinite) sequence of injective maps the total composite is a surjection if and only if each of the constituents is. Together with the uniqueness of shifts up to precomposition with an automorphism this shows that conditions $(ii)$, $(iii)$ and $(iv)$ are equivalent. Moreover, it is immediate that $(i)$ implies all the others.

The remaining implications are a consequence of:
\begin{Lemma} If an element $a$ of $A$ does not have filtration $\emptyset$ (i.e., if there exists an injection $\U\hookrightarrow \U$ which acts non-trivially on $a$) and $d^M:\U\hookrightarrow \U$ is a shift by $M$, then $d^M\cdot a$ is not of filtration $M$. 
\end{Lemma}
\begin{proof} Let $\psi:\U\hookrightarrow \U$ be an injection with $\psi\cdot a\neq a$. Then, by the injectivity of $d^M\cdot -$ we have
 \[ d^M \cdot a\neq d^M\cdot (\psi \cdot a)=c_{d^M}(\psi)\cdot (d^M\cdot a) \]
and $c_{d^M} (\psi)$ is an injection which pointwise fixes $M$. Hence, $d^M\cdot a$ is not of filtration $M$.
\end{proof}
Now this implies that if an element $a$ is of filtration $M$ and not $\emptyset$, it cannot be in the image of $d^M \cdot -$ for any shift by $M$ (because a premiage $b$ would need to be of filtration $\emptyset$ by the lemma, implying in particular that $a=d^M\cdot b=b$, which contradicts the fact that $a$ is not of filtration $\emptyset$). Hence, $(iii)$ implies $(i)$.

Finally, $(v)$ implies $(i)$, because if there existed an element $a$ not of filtration $\emptyset$, then $d^M\cdot a$ would not be of filtration $M$ by the lemma, contradicting the assumption.
\end{proof}
\begin{Cor} \label{cor:semifingen} Every tame $\M_G$-module which is finitely generated as an abelian group has trivial action. 
\end{Cor}
\begin{proof} Let $A$ be generated by $a_1,\hdots,a_n$. Then each $a_i$ is of filtration $M_i$ for some finite $G$-subset $M_i$ of $\U$, hence they are all of filtration $\cup M_i$, which is again finite. It follows that every element of $A$ is of filtration $\cup M_i$, so by item $(v)$ in the previous proposition the action is trivial. 
\end{proof}

\subsection{Action of $\M_G$ on naive homotopy groups of $G$-symmetric spectra}
\label{sec:action}
Now we introduce the action of $\M_G$ on the naive homotopy groups of a $G$-symmetric spectrum. It suffices to explain it in the case $n=0$, since $\pi_n^{H,\U}X$ is canonically isomorphic to $\pi_0^{H,\U}(\Omega^n X)$ for $n>0$ and to $\pi_0^{H,\U} (S^{-n}\wedge X)$ for $n<0$ (cf. Proposition \ref{prop:suspiso}). So let $\alpha:\U\hookrightarrow \U$ be a $G$-equivariant injection and $x$ an element of $\pi_0^{H,\U}X$. Then $x$ is represented by an $H$-equivariant map $f:S^M\to X(M)$ for some finite $G$-subset $M$ of $\U$.
\begin{Def} \label{def:action} The element $\alpha \cdot x\in \pi_n^{H,\U}X$ is defined as the class of the composite
\[ \alpha_*f:S^{\alpha(M)}\xr{S^{\alpha_{|M}^{-1}}} S^{M}\xr{f} X(M)\xr{X(\alpha_{|M})} X(\alpha(M)), \]
i.e., the injection monoid $\M_G$ acts ``through conjugation''.
\end{Def}
\begin{Remark} \label{rem:action}
The construction also makes sense if $\alpha$ is an injective $G$-equivariant map between two different (and possibly not even isomorphic) $G$-set universes $\U$ and $\U'$, inducing a map $\alpha\cdot (-):\pi_n^{H,\U}X\to \pi_n^{H,\U'}X$. 
\end{Remark}
\begin{comment}
In order to see that this is well-defined we let $N$ be another finite $G$-subset of $\U$ disjoint to $M$ and consider the commutative diagram
\[ \xymatrix{ S^{\alpha(M)}\wedge S^{\alpha(N)} \ar[rr]^-{\alpha_{|M}^{-1}\wedge \alpha_{|N}^{-1}}  \ar[d]_{\cong} && S^{M}\wedge S^{N}\ar[r]^-{f\wedge S^{N}} \ar[d]_{\cong} & X(M)\wedge S^{N} \ar[rr]^-{X(\alpha_{|M})\wedge S^{\alpha_{|N}}} \ar[d]^{\sigma_M^{N}} && X(\alpha(M))\wedge S^{\alpha(N)} \ar[d]^{\sigma_{\alpha(M)}^{\alpha(N)}} \\
S^{\alpha(M\sqcup N)} \ar[rr]_-{\alpha^{-1}_{|M\sqcup N}} && S^{M\sqcup N} \ar[r] & X(M\sqcup N) \ar[rr]_-{X(\alpha_{|M\sqcup N})} && X(\alpha(M\sqcup N))}
\]
where the commutativity of the lower square is a consequence of the equivariance of the structure maps with respect to products of symmetric groups. So we see that
\[ \sigma_{\alpha(M)}^{\alpha(N)}\circ (\alpha_* f\wedge S^{\alpha(N)}) = \alpha_* (\sigma_M^{N}\circ (f\wedge S^{N})) \]
and hence $\alpha_*f$ and $\alpha_* (\sigma_M^{N}\circ (f\wedge S^{N}))$ define the same class in $\pi_0^{H,\U}X$.
\end{comment}

It is straightforward to check that the action does not depend on the chosen representative~$f$, and that it is unital, associative, additive and natural. Furthermore, every element in $\pi_0^{H,\U}X$ represented by a map $f:S^{M}\to X(M)$ is of filtration $M$ in the sense of the previous section, since by definition the action of an injection $\alpha$ on $[f]$ depends only on the restriction of $\alpha$ to $M$. In particular, this implies that $\pi_0^{H,\U}X$ is a tame $\M_G$-module.
\begin{Remark} In order for an injection $\alpha:\U\hookrightarrow \U$ to act on $\pi_n^{H,\U}X$ it would suffice that it is $H$-equivariant and not necessarily $G$-equivariant. However, since finite $G$-subsets are cofinal in the poset of finite $H$-subsets of $\U$, the action of $\M_H$ is trivial if and only if the one of $\M_G$ is (cf. Proposition \ref{prop:criteria}), and hence the latter is enough to detect semistability.
\end{Remark}

\begin{Def}[Semistability] A $G$-symmetric spectrum $X$ is called \emph{$G^{\U}$-semistable} if the action of $\M_G$ on $\pi_n^{H,\U}X$ is trivial for every subgroup $H$ of $G$ and every $n\in \mathbb{Z}$.
\end{Def}
We have:
\begin{Lemma} \label{lem:omegasemi}Every $G^{\U}\Omega$-spectrum is $G^{\U}$-semistable. \end{Lemma}
\begin{proof} Let $X$ be a $G^{\U}\Omega$-spectrum and $n\in \mathbb{Z}$ an integer. By Example \ref{exa:piomega} there exists a finite $G$-subset $M$ of $\U$ such that the map $[S^{n\sqcup M},X(M)]^H\to \pi_n^{H,\U}X$ is a bijection. In particular, it is surjective and hence every element in $\pi_n^{H,\U}X$ is of filtration $M$. By criterion  $(v)$ of Proposition \ref{prop:criteria}, this implies that the action of $\M_G$ is trivial.
\end{proof}
This already implies that if a $G$-symmetric spectrum $X$ is not $G^{\U}$-semistable, there cannot exist a $\upi_*^{\U}$-isomorphism from $X$ to a $G^ {\U}\Omega$-spectrum. Conversely, in Corollary \ref{cor:semipiiso} and Proposition \ref{prop:semiomega} it is shown that every $G^{\U}$-semistable $G$-symmetric spectrum admits such a $\upi_*^{\U}$-isomorphism. %Example \ref{exa:semifree} explains that the free spectrum $\mathscr{F}_MS^M$ is not $G^{\U}$-semistable for any non-empty finite $G$-set $M$.

Finally, we note the following immediate consequence of Corollary \ref{cor:semifingen}:
\begin{Cor} \label{cor:semifinite} Every $G$-symmetric spectrum $X$ for which all $\pi_n^{H,\U}X$ are finitely generated as abelian groups is $G^{\U}$-semistable. 
\end{Cor}

\subsection{Mackey functor structure}
\label{sec:mackey}
It can be shown that for every $n\in \mathbb{Z}$ the collection $\upi_n^{\U} X=\{\pi_n^ {H,\U}X\}_{H\leq G}$ of $n$-th naive homotopy groups of a $G$-symmetric spectrum $X$ forms a (restricted) Mackey functor, i.e., a coefficient system together with transfer maps for a certain class of subgroup inclusions $H\leq K$ depending on the universe $\U$ or rather its linearization. Concretely this means that there are
\begin{itemize}
\item contravariantly functorial restriction homomorphisms $\res_H^K:\pi_n^{K,\U}X\to \pi_n^{H,\U}X$ for every subgroup inclusion $H\leq K$,
\item covariantly functorial transfer homomorphisms $\tr_H^K:\pi_n^{H,\U}X\to \pi_n^{K,\U}X$ for every subgroup inclusion $H\leq K$ such that $K/H$ allows a $K$-embedding into $\R[\U]$, and
\item transitive conjugation homomorphisms $c_g:\pi_n^{H,\U}X\to \pi_n^{gHg^{-1},\U}X$ for every element $g\in G$ and subgroup $H$ of $G$
\end{itemize}
such that inner conjugations act trivially, restrictions and transfers commute with conjugations and the double coset formula holds. Furthermore, it can be shown that these structure maps commute with the $\M_G$-actions. Because we do not make use of the Mackey-functor structure in this paper and it is similar to the one for $G$-orthogonal spectra (cf. \cite[Sec. 3 and 4]{Sch11}), we do not describe its construction here.
\subsection{Shift and relation to $\M_G$-action}
\label{sec:shifthom}
We now discuss the effect of the shift $sh^M$ along a finite $G$-set $M$ (contained in $\U$) on homotopy groups, or more precisely the effect of the composite $\Omega^M sh^M$ together with the natural transformation $\widetilde{\alpha}^M:id\to \Omega^M sh^M$ of Definition \ref{def:shift}. By adjunction, the $n$-th homotopy group $\pi_n^{H,\U}(\Omega^M sh^M X)$ is naturally isomorphic to the colimit over the terms $[S^{n\sqcup M\sqcup N},X(M\sqcup N)]^H$ for all finite $G$-subsets $N$ of $\U$. Finite $G$-subsets of the form $M\sqcup N$ are cofinal in the $G$-set universe $M\sqcup \U$, so we see that this colimit and hence $\pi_n^{H,\U}(\Omega^M sh^M X)$ is naturally isomorphic to $\pi_n^{H,M\sqcup \U}X$. Since $M\sqcup \U$ is isomorphic to $\U$, it follows that the homotopy groups $\pi_n^{H,\U}X$ and $\pi_n^{H,\U}(\Omega^M sh^M X)$ are abstractly isomorphic. This already shows:
\begin{Cor} \label{cor:shpiiso} The shift $sh^M$ along any finite $G$-subset $M\subset \U$ preserves and reflects $\upi_*^{\U}$-isomorphisms.
\end{Cor}
\begin{proof} By the (non-canonical) isomorphism of homotopy groups above, a map $f:X\to Y$ is a $\upi_*^{\U}$-isomorphism if and only if $\Omega^M sh^M f$ is. The latter is equivalent to $sh^M f$ being one by Proposition \ref{prop:suspiso2}.
\end{proof}

However, the fact that the isomorphism $\pi_n^{H,\U}(\Omega^M sh^M X)\cong \pi_n^{H,\U} X$ is not canonical (it depends on a choice of isomorphism $M\sqcup \U\cong \U$) should make one skeptical about whether such an isomorphism is induced by $\widetilde{\alpha}_X^M$. In fact, while there is no canonical isomorphism $M\sqcup \U\cong \U$, there is of course a canonical $G$-equivariant embedding $i:\U\hookrightarrow M\sqcup \U$ and it turns out that this describes the effect of $\widetilde{\alpha}_X^M$ on homotopy groups:
\begin{Prop} \label{prop:pishift} The composite
$\pi_n^{H,\U}X\xr{(\widetilde{\alpha}_X^M)_*} \pi_n^{H,\U}(\Omega^M sh^M X)\cong \pi_n^{H,M\sqcup \U}X$ equals the action of the inclusion $i:\U\hookrightarrow M\sqcup \U$, in the sense of Remark \ref{rem:action}.
\end{Prop}
\begin{proof} Let $x\in \pi_n^{H,\U}X$ be an arbitrary element, represented by an $H$-map $f:S^{n\sqcup N}\to X(N)$ for some finite $H$-set $N$ contained in $\U$. Then $i_*f$ is again $f$, with the only difference of thinking of $N$ as now sitting inside $M\sqcup \U$. In order to compare it to an element of $\pi_n^{H,\U}(\Omega^M sh^M X)$ under the isomorphism above, we have to suspend this element by $M$, i.e., form the composition
\[ S^{n\sqcup N}\wedge S^M\xr{f\wedge S^M} X(N)\wedge S^M\xr{\sigma_N^M}X(N\sqcup M)\xr{X(\chi_{N,M})} X(M\sqcup N). \]
But, by the definition of $\widetilde{\alpha}_X^M$, after adjoining $S^M$ to the right this is precisely the composition \[ S^{n\sqcup N}\xr{f} X(N)\xr{(\widetilde{\alpha}_X^M)(N)} \Omega^M X(M\sqcup N)=(\Omega^M sh^M X)(N) \]
and thus $i\cdot x=[i_*f]=[\widetilde{\alpha}_X^M(N)\circ f]=\widetilde{\alpha}_X^M(x)$.
\end{proof}
This proposition can be translated into a statement about the internal action of $\M_G$ on $\pi_n^{H,\U}X$, where it corresponds to the algebraic shift discussed in Definition \ref{def:algshift}. For this we choose a ``shift by $M$'', i.e., an embedding $d^M:\U\hookrightarrow \U$ with image $\U-M$. In particular we obtain an isomorphism $id_M\sqcup d^M:M\sqcup \U\xr{\cong} \U$ of $G$-set universes and hence a natural isomorphism $\pi_n^{H,\U}(\Omega^M sh^M X)\cong \pi_n^{H,M\sqcup \U}X\cong \pi_n^{H,\U}X$. This isomorphism is not $\M_G$-equivariant in general, but it becomes so when conjugating the action on $\pi_n^{H,\U}X$ along $d^M$ in the sense of Definition \ref{def:conjugation}:
\begin{Prop} \label{prop:pishift2} The above defines an $\M_G$-isomorphism
  \[ \pi_n^{H,\U}(\Omega^M sh^M X)\cong c_{d^M}^*(\pi_n^{H,\U}X). \]
Moreover, the composition $\pi_n^{H,\U}X\xr{\widetilde{\alpha}_X^M}\pi_n^{H,\U}(\Omega^M sh^M X)\xr{\cong} c_{d^M}^*(\pi_n^{H,\U}X)$ equals multiplication with $d^M$.
\end{Prop}
%\todo{maybe leave out proof}
\begin{proof} Under the isomorphism $\pi_n^{H,\U}(\Omega^M sh^M X)\cong \pi_n^{H,M\sqcup \U}X$ the $\M_G$-action on the former corresponds to the $\mathbf{M}^{M\sqcup \U}_G$-action on the latter pulled back along the homomorphism $id_M\sqcup -:\M_G\to \mathbf{M}^{M\sqcup \U}_G$, since $M$ is left untouched. The equality $(id_M\sqcup d^M) \circ i_{\U}=d^M$ implies that after conjugation with $id_M\sqcup d^M$ this homomorphism becomes $c_{d^M}$, and so the action on $\pi_n^{H,\U}X$ is pulled back along $c_{d^M}$ as claimed. The second statement is then an immediate consequence of Proposition \ref{prop:pishift}.
\end{proof}
Applying Lemma \ref{lem:injective} and Proposition \ref{prop:criteria} on algebraic properties of tame $\M_G$-modules, we obtain the following corollaries:
\begin{Cor} The map $\widetilde{\alpha}_X^M:X\to (\Omega^M sh^MX)$ is a $\upi_*^{\U}$-isomorphism if and only if $d^M$ (and hence every shift by $M$) acts surjectively on $\pi_n^{H,\U}X$ for all subgroups $H$ of $G$ and all $n\in \mathbb{Z}$. 
\end{Cor}
\begin{Cor} \label{cor:semi} A $G$-symmetric spectrum $X$ is $G^{\U}$-semistable if and only if the map $\widetilde{\alpha}_X^M:X\to \Omega^Msh^MX$ is a $\upi_*^{\U}$-isomorphism for every finite $G$-subset $M$ of $\U$ and if and only if $\alpha_X^M:S^M\wedge X\to sh^M X$ is a $\upi_*^{\U}$-isomorphism for every finite $G$-subset $M$ of $\U$.
\end{Cor}
The last ``if and only if'' follows from Proposition \ref{prop:suspiso2}.
\begin{Cor} \label{cor:mapsemi} Let $\F$ be a family of subgroups of $G$, $A$ a cofibrant $G$-space with non-basepoint isotropy in $\F$ and $X$ a $G$-symmetric spectrum such that $\M_{G}$ acts trivially on $\pi_*^{H,\U}X$ for all $H$ in $\F$. Then $A\wedge X$ is $G^{\U}$-semistable. If $A$ is finite and $X$ is $G^{\U}$-projective level fibrant, then $\map(A,X)$ is also $G^{\U}$-semistable.
\end{Cor}
\begin{proof} This follows from Corollaries \ref{cor:semi} and \ref{cor:smashpi}, since $A\wedge (-)$ commutes both with $S^M\wedge (-)$ and $sh^M(-)$ and $\map(A,-)$ commutes with $\Omega^M sh^M(-)$.
\end{proof}
\begin{Cor} \label{cor:shiftsemi} If $X$ is $G^{\U}$-semistable, then so is $sh^N X$ for any finite $G$-subset $N$ of $\U$.
\end{Cor}
\begin{proof} The shift $sh^N(-)$ commutes up to natural isomorphism with $sh^M(-)$ and $S^M\wedge (-)$, so this also follows from Corollaries \ref{cor:semi} and \ref{cor:shpiiso}.
\end{proof}

We now introduce the endofunctor $\Omega^{\U}sh^{\U}$ on $G$-symmetric spectra (together with a natural transformation $\widetilde{\alpha}^{\U}:id\to\Omega^{\U}sh^{\U}$), the equivariant analog of the construction called $R^{\infty}$ in \cite[Thm. 3.1.11]{HSS00}. It serves two purposes for us: First, we show that for all $G^{\U}$-semistable $G$-symmetric spectra $X$ the $G$-symmetric spectrum $\Omega^{\U}sh^{\U}X$ is a $G^{\U}\Omega$ spectrum and the map $\widetilde{\alpha}^{\U}_X$ is a $\upi_*^{\U}$-isomorphism, proving that $G^{\U}$-semistable $G$-symmetric spectra can be replaced by $G^{\U}\Omega$-spectra up to $\upi_*^{\U}$-isomorphism. Secondly, via an equivariant version of the proof of \cite[Thm. 3.1.11]{HSS00}, it is used to show that every $\upi_*^{\U}$-isomorphism is a $G^{\U}$-stable equivalence.

We choose an exhaustive filtration $\emptyset=M_0\subseteq M_1\subseteq \hdots \subseteq M_n \subseteq \hdots \subseteq \U$
of $\U$ by finite $G$-subsets. For a $G$-symmetric spectrum of spaces $X$ we then define $\Omega^{\U}sh^{\U}X$ as the mapping telescope of the sequence
\[ X=\Omega^{M_0}sh^{M_0}X\to \Omega^{M_1}sh^{M_1}X \to \Omega^{M_2}sh^{M_2}X\to \hdots \]
with induced natural map $\widetilde{\alpha}^{\U}_X:X\to \Omega^{\U}sh^{\U}X$.
The connecting maps in the system are given by
\[\Omega^{M_n}sh^{M_n}X\xrightarrow{\Omega^{M_n}\widetilde{\alpha}^{(M_{n+1}-M_n)}_{sh^{M_n}X}} \Omega^{M_n}\Omega^{M_{n+1}-M_n}sh^{M_{n+1}-M_n}sh^{M_n}X\cong \Omega^{M_{n+1}}sh^{M_{n+1}}X. \]
The effect on homotopy groups of this construction corresponds to the infinite shift below Definition \ref{def:algshift} (applied degreewise to $\M_G$-Mackey functors), whose terminology we now use. So for every $n\in \N$ we choose an $(M_n-M_{n-1})$-shift $d^M:\U\hookrightarrow \U$ leaving $M_{n-1}$ fixed. By composition we obtain $M_n$-shifts $d^{M_n}$, inductively defined via $d^M=d^{M_n-M_{n-1}}\circ d^{M_{n-1}}$. Making use of Proposition \ref{prop:pishift2} we see that there is an isomorphism of sequences
\[ \xymatrix{ \upi_n^{\U} X\ar[rr] \ar[d]_{id} && \upi_n^{\U}(\Omega^{M_1}sh^{M_1} X)\ar[rr] \ar[d]_{\cong} && \upi_n^{\U}(\Omega^{M_2}sh^{M_2} X) \ar[r] \ar[d]_{\cong} &  \hdots \\
  \upi_n^{\U} X \ar[rr]^-{d^{M_1}\cdot } && c_{d^{M_1}}^*(\upi_n^{\U} X)\ar[rr]^{d^{M_2-M_1}\cdot} && c_{d^{M_2}}^*(\upi_n^{\U} X) \ar[r] & \hdots }  \]
of $\M_G$-modules. In particular, it follows that the map $(\widetilde{\alpha}^{\U}_X)_*:\pi_n^{H,\U}X\to \pi_n^{H,\U}(\Omega^{\U}sh^{\U}X)$ can be identified with $d^{\U}:\pi_n^{H,\U}X\to c_{d^\U}^*(\pi_n^{H,\U} X)$ and via Proposition \ref{prop:criteria} we obtain:
\begin{Cor} \label{cor:semipiiso} The map $\widetilde{\alpha}^{\U}_X:X\to \Omega^{\U} sh^{\U}X$ is a $\upi_*^{\U}$-isomorphism if and only if $X$ is $G^{\U}$-semistable. 
\end{Cor}

Now we want to show that if $X$ is $G^{\U}$-semistable, then $\Omega^{\U}sh^{\U}X$ is a $G^{\U}\Omega$-spectrum. For this we need:
\begin{Lemma} \label{lem:omu}\begin{enumerate}[(i)]
	\item For every $G$-symmetric spectrum of spaces $X$, every subgroup $H$ of $G$ and every two finite $H$-sets $L$ and $M$ there is a natural bijection \[ [S^{L},(\Omega^{\U}sh^{\U}X)(M)]^H \cong \pi_0^{H,\U} (\Omega^{L} sh^M X)\]
	Moreover, given another finite $H$-subset $N$ of $\U$ disjoint to $M$, the square
	\[\xymatrix{ \pi_k ((\Omega^{\U}sh^{\U}X)(M)^H) \ar[rr]^{(\widetilde{\sigma}_M^N)_*} \ar[d]_{\cong} && \pi_k ((\Omega^N (\Omega^{\U}sh^{\U}X)(M))^H) \ar[d]^{\cong} \\
	\pi_k^{H,\U}(sh^M X)\ar[rr]^{(\widetilde{\alpha}_{sh^MX}^N)_*} && \pi_k^{H,\U} (\Omega^N sh^{M\sqcup N} X)}	 
	\]
	commutes, where the vertical isomorphisms are those obtained by setting $L=k$ respectively $L=N\sqcup k$.
	\item For every $G^{\U}\Omega$-spectrum of spaces $X$, the map $\widetilde{\alpha}^{\mathscr{U}}_X:X\to \Omega^{\U}sh^{\U}X$ is a $G^{\U}$-level equivalence.
\end{enumerate}
\end{Lemma}
\begin{proof} The first part follows by a chain of adjunctions, the second is a consequence of Corollary \ref{cor:omegash} and the fact that looping with respect to a $G$-set preserves $G^{\U}$-level equivalences.
\begin{comment}
Regarding $(i)$, we see:
\begin{align*} 
	[S^L,(\Omega^{\U}sh^{\U}X)(M)]^H & \cong \colim _{n\in \N}[S^L,(\Omega^{M_n}sh^{M_n}X)(M)]^H \cong \colim _{n\in \N}[S^{L\sqcup M_n},(sh^{M_n}X)(M)]^H \\
				& \cong \colim _{n\in \N}[S^{L\sqcup M_n},X(M_n\sqcup M)]^H
\cong \colim _{n\in \N}[S^{L\sqcup M_n},X(M\sqcup M_n)]^H \\																		& \cong \pi_0^{H,\U}(\Omega^L(sh^M X))
\end{align*}
%The last step uses that the $M_n$ are an exhaustive sequence in $\U$. The described levelwise isomorphisms are compatible with the maps in the colimit system since on both sides neither $L$ nor $M$ are touched by the stabilization maps.
The commutativity of the diagram is a direct check from the definitions.

By Corollary \ref{cor:omegash}, for every $G^{\U}\Omega$-spectrum $Z$ and every finite $G$-set $M\subset \U$ the map $\alpha_Z^M:Z\to \Omega^M sh^M Z$ is a $G^{\U}$-level equivalence and $\Omega^M sh^M Z$ is again a $G^{\U}\Omega$-spectrum. Since looping with respect to a $G$-set preserves $G^{\U}$-level equivalences, it follows that every map in the sequence is a $G^{\U}$-level equivalence and hence so is $\widetilde{\alpha}^{\U}_Z$.
\end{comment}
\end{proof}
\begin{Prop} \label{prop:semiomega} If $X$ is $G^{\U}$-semistable, then $\Omega^{\U}sh^{\U} X$ is a $G^{\U}\Omega$-spectrum.
\end{Prop}
\begin{proof} Let $H$ be a subgroup of $G$ and $M,N$ two finite disjoint $H$-subsets of $\U$. Since $X$ is $G^{\U}$-semistable, so is $sh^M X$ by Corollary \ref{cor:shiftsemi} and hence the map $\widetilde{\alpha}^N_{sh^M X}$ is a $\upi_*^{\U}$-isomorphism. Now Lemma \ref{lem:omu} implies that in this case the adjoint structure map $\widetilde{\sigma}_M^N:(\Omega^{\U}sh^{\U}X)(M)\to \Omega^N (\Omega^{\U}sh^{\U}X)(M\sqcup N)$ induces a bijection on all homotopy groups at the basepoint of the $H$-fixed points.

We are not yet done because we have not shown anything about the homotopy groups at elements of $(\Omega^{\U}sh^{\U}X)(M)^H$ which lie in other components. For this we
%we use that there is a group-like $H$-space structure on $(\Omega^{\U}sh^{\U}X)(M)^H$ and $(\Omega^N (\Omega^{\U}sh^{\U}X)(M\sqcup N))^H$ for which $\widetilde{\sigma}_M^N$ is a homomorphism, obtained in the following way: We 
take some $n\in \N$ such that $M_n$ is not the empty set. Then $(\Omega^{\U}sh^{\U}X)$ is $G$-homotopy equivalent to the mapping telescope only taken over the terms $\Omega^{M_i}sh^{M_i} X$ for $i\geq n$.
From $n$ on all the maps in the system are of the form $\Omega^{M_i}f_i$ for some map $f_i$, in particular all the maps are loop maps with respect to $M_n$. Since $\Omega^{M_i}$ commutes with mapping telescopes up to weak equivalence, we obtain that $\Omega^{\U}sh^{\U}X$ is $G^{\U}$-level equivalent to $\Omega$ of some spectrum. So since $\R[M_n]$ contains a trivial summand, the levels obtain a group-like $H$-space structure and the adjoint structure map is an $H$-space homomorphism, which finishes the proof.
%s of the loops of a spectrum are given by the loops of the adjoint structure maps, we see that $\widetilde{\sigma}_M^N$ is a homomorphism with respect to these $H$-space structures and so it is a genuine $H$-equivalence. 
\end{proof}
Together with Lemma \ref{lem:omegasemi} and Corollary \ref{cor:semipiiso} this gives:
\begin{Cor} \label{cor:semicrit} A $G$-symmetric spectrum allows a $\upi_*^{\U}$-isomorphism to a $G^{\U}\Omega$-spectrum if and only if it is $G^{\U}$-semistable. 
\end{Cor}
\subsection{Relation to $G^{\U}$-stable equivalences}
We are now ready for:
\begin{Theorem} \label{theo:piiso} Every $\upi_*^{\U}$-isomorphism of $G$-symmetric spectra is a $G^{\U}$-stable equivalence.
\end{Theorem}
\begin{proof}
%\todo{maybe completely refer to HSS?}
%The proof is an equivariant analog of the argument in \cite[Theorem 3.1.11]{HSS00}. We first show the following special case: If a $G$-symmetric spectrum of spaces has trivial homotopy groups, it is $G^{\U}$-stably contractible. So we let $X$ be such a $G$-symmetric spectrum. By Corollary \ref{cor:shpiiso} we know that all shifts $sh^MX$ along finite $G$-subsets $M$ of $\U$ also have trivial homotopy groups and so by Lemma \ref{lem:omu} the spectrum $\Omega^{\U}sh^{\U}X$ is $G^{\U}$-level contractible and in particular $G^{\U}$-stably contractible.
By Proposition \ref{prop:ggeo} it suffices to show the topological case. The proof is an equivariant analog of the argument in \cite[Thm. 3.1.11]{HSS00}. So let $f:X\to Y$ be a $\upi_*$-isomorphism. By Corollary \ref{cor:shpiiso} we know that all shifts $sh^Mf$ along finite $G$-subsets $M$ of $\U$ are also $\upi_*$-isomorphisms. So by Lemma \ref{lem:omu} (and the same argument as in the proof of Proposition \ref{prop:semiomega} to deal with the non-basepoint components) we see that $\Omega^{\U} sh^{\U}f$ is a $G^{\U}$-level equivalence and in particular a $G^{\U}$-stable equivalence.

We claim that this implies that $f$ is a $G^{\U}$-stable equivalence, too: The functors $\Omega^{M_i}sh^{M_i}$ preserve $G^{\U}$-level equivalences and hence so does their mapping telescope $\Omega^{\U}sh^{\U}$. Therefore, we obtain a functor\[ \Omega^{\U}sh^{\U}:\Ho^{\U} _{lev}(GSp^{\Sigma}_\T)\to \Ho^{\U}_{lev}(GSp^{\Sigma}_\T) \] and a natural transformation $\widetilde{\alpha}:id_{\Ho^{\U}_{lev}(GSp^{\Sigma}_\T)}\to \Omega^\U sh^\U$. Let $\gamma:GSp^\Sigma_\T\to \Ho^{\U}_{lev}(GSp^{\Sigma}_\T)$ be the projection. In the following we abbreviate $\Ho^{\U}_{lev}(GSp^{\Sigma}_\T)$ by $\Ho^{\U}_{lev}$. We have to show that $\Ho^{\U}_{lev}(f,Z)$ is a bijection for all $G^{\U}\Omega$-spectra $Z$. We consider the composite
\[ \Ho^{\U}_{lev}(X,Z)\xr{\Omega^\U sh^\U} \Ho^{\U}_{lev}(\Omega^\U sh^\U X,\Omega^\U sh^\U Z)\xr{\widetilde{\alpha}_X^*} \Ho^{\U}_{lev}(X,\Omega^\U sh^\U Z), \]
and similarly for $Y$. It sends a morphism $\phi:X\to Z$ to $\Omega^\U sh^\U (\phi)\circ \widetilde{\alpha}_X$, which by naturality equals $\widetilde{\alpha}_Z\circ \phi$. Since $Z$ is a $G^{\U}\Omega$-spectrum, it follows from Lemma \ref{lem:omu} that $\widetilde{\alpha}_Z$ is an isomorphism in the level homotopy category and hence postcomposition with it gives a natural isomorphism on morphism sets. It follows that $\Ho^{\U}_{lev}(f,Z)$ is a retract of $\Ho^{\U}_{lev}(\Omega^{\U}sh^{\U}f,Z)$ which we know to be a bijection since $\Omega^{\U}sh^{\U}f$ is a $G^{\U}$-stable equivalence. Hence, $\Ho^{\U}_{lev}(f,Z)$ is a bijection, too, and $f$ is a $G^{\U}$-stable equivalence.
\end{proof}
\begin{Cor} \label{cor:piiso} A morphism between $G^{\U}$-semistable $G$-symmetric spectra is a $\upi_*^{\U}$-isomorphism if and only if it is a $G^{\U}$-stable equivalence.
\end{Cor}
\begin{proof} It suffices to show the topological case. We have just seen that every $\upi_*^{\U}$-isomorphism is a $G^{\U}$-stable equivalence. For the other direction, let $f:X\to Y$ be a $G^{\U}$-stable equivalence between $G^{\U}$-semistable $X$ and $Y$. Then $\Omega^{\U}sh^{\U}(f)$ is a $G^{\U}$-stable equivalence between $G^{\U}\Omega$-spectra, hence a $G^{\U}$-level equivalence by the Yoneda Lemma and in particular a $\upi_*^{\U}$-isomorphism. By 2-out-of-3 it follows that $f$ is also a $\upi_*^{\U}$-isomorphism. 
\end{proof}

\section{Stable model structures}
\subsection{Properties of $G^{\U}$-stable equivalences} \label{sec:properties}
This section examines the behavior of $G^{\U}$-stable equivalences with respect to certain constructions. First we need a definition:

\begin{Def}[$h$-cofibrations] \label{def:hcof} A map $i:X\to Y$ of $G$-symmetric spectra of spaces is called an \emph{$h$-cofibration} if it satisfies the homotopy extension property.
\end{Def}
This property is equivalent to the inclusion $([0,1]_+\wedge X)\cup_{\{0\}_+\wedge X}Y\to [0,1]_+\wedge Y$ having a retraction, and so it follows that any functor preserving pushouts and the smash product with the interval also preserves $h$-cofibrations. In particular, this holds for smashing with any $G$-symmetric spectrum. Furthermore, the pushout of an $h$-cofibration along an arbitrary map is again an $h$-cofibration, and so is the (transfinite) composition of $h$-cofibrations. Since it is easily checked that the generating $G$-flat cofibrations (Equation \ref{eq:gencofflat}) are $h$-cofibrations, this implies that any $G$-flat cofibration (and hence also any $G^{\U}$-projective cofibration) is an $h$-cofibration.

%\emph{levelwise $G$-cofibration} if all its evaluations at finite $G$-sets are genuine $G$-cofibrations. In particular, $G$-flat cofibrations are levelwise $G$-cofibrations.

\begin{Prop}\label{prop:gwedge} \begin{enumerate}
\item A wedge of $G^{\U}$-stable equivalences is again a $G^{\U}$-stable equivalence.
\item Smashing with a cofibrant $G$-space $A$ preserves $G^{\U}$-stable equivalences.
\item Let $f:X\to Y$ be a morphism of $G$-symmetric spectra of spaces and $V$ a $G$-subrepresentation of $\R[\U]$. Then the following are equivalent: \begin{enumerate}[(i)]
	\item $f$ is a $G^{\U}$-stable equivalence.
	\item $S^V\wedge f$ is a $G^{\U}$-stable equivalence.
	\item $\Omega^Vf$ is a $G^{\U}$-stable equivalence.
	\item $C(f)$ is $G^{\U}$-stably contractible, i.e., the unique map $C(f)\to *$ is a $G^{\U}$-stable equivalence.
	\item $H(f)$ is $G^{\U}$-stably contractible.
\end{enumerate}
\item Let\[ \xymatrix{A \ar[r]^f\ar[d]_i & B \ar[d]^j \\
						X \ar[r]_g & Y } \]
be a pushout diagram of $G$-symmetric spectra of spaces with $i$ an $h$-cofibration. Then: \begin{enumerate}[(i)]
	\item If $i$ is a $G^{\U}$-stable equivalence, then so is $j$.
	\item If $f$ is a $G^{\U}$-stable equivalence, then so is $g$.
\end{enumerate}
\item Let $X_0\xr{i_0} X_1 \xr{i_1} X_2 \xr{i_2}\hdots$
be a sequence of morphisms $i_j$ of $G$-symmetric spectra of spaces which are $h$-cofibrations and $G^{\U}$-stable equivalences. Then the induced morphism $X_0\to \colim _{j}X_j$ is also a $G^{\U}$-stable equivalence.
\end{enumerate}
\end{Prop}
There is also a dual version of item $4$, saying that the pullback of a $G^{\U}$-stable equivalence along a projective $G^{\U}$-level fibration is again a $G^{\U}$-stable equivalence (or similarly, that the pullback of a $G^{\U}$-stable equivalence that is also a projective $G^{\U}$-level fibration along an arbitrary map is again a $G^{\U}$-stable equivalence).
\begin{proof} The arguments are similar to the non-equivariant case, cf. \cite[Thm. 8.12]{MMSS01}. Using that $A\wedge -$, $\Omega^V$ and wedges take all $\upi_*$-isomorphisms (in particular, $G^{\U}$-level equivalences) to $\upi_*$-isomorphisms, one can reduce to the case where all the spectra involved in items $(1)-(3)$ are $G$-flat and $f$ is a $G$-flat cofibration. Items $(1)$ and $(2)$ then follow by adjunction. To show that $(3.ii)$ implies $(3.i)$, note that every $G^{\Omega}$-spectrum $X$ is $G^{\U}$-level equivalent to $\Omega^V \Omega^{M-V} sh^M X$, where $M$ is some finite $G$-subset of $\U$ whose linearization contains $V$. The equivalence to $(3.iii)$ follows from Proposition \ref{prop:suspiso2}, and that to $(3.iv)$ from the long exact sequence in homotopy classes of maps associated to a cofiber sequence. Finally, $C(f)$ and $S^1\wedge H(f)$ are $G^{\U}$-stably equivalent by Proposition \ref{prop:exseq}, so via the equivalence of $(3,i)$ and $(3,ii)$ it follows that $(3.iv)$ and $(3.v)$ are also equivalent.

The proof of item $(4)$ uses that (by $(3)$) $G^{\U}$-stable equivalences can be tested on cones. Regarding item $(5)$, after replacing each $i_j$ by a flat cofibration (up to $G^{\U}$-level equivalence) and applying $\map(-,Z)$ for a level-fibrant $G^{\U}\Omega$-spectrum $Z$ yields a tower of acylic fibrations of spaces with limit weakly equivalent to $\map(\colim _{j}X_j,Z)$, since the level model structure is topological. It follows that the map $\map(\colim _{j}X_j,Z)\to \map(X_0,Z)$ is a weak equivalence and taking $\pi_0$ yields the result.
\end{proof}

\subsection{$Q^{\U}$-replacement and stable model structures}
\label{sec:gmodstable}
Our next goal is the construction of an endofunctor $Q^{\U}$ of $G$-symmetric spectra with image in $G^{\U}\Omega$-spectra, together with a natural $G^{\U}$-stable equivalence $q^{\U}:id\to Q^{\U}$. We note that for $G^{\U}$-semistable $G$-symmetric spectra the construction $\widetilde{\alpha}_X^{\U}:X\to \Omega^{\U}sh^{\U}X$ of Section \ref{sec:shifthom} has these properties, but for general $X$ one needs to proceed in a different way.

We recall the $G^{\U}$-stable equivalences \[	G\ltimes_H\lambda^{(H)}_{M,N}: G\ltimes_H\mathscr{F}_{M\sqcup N}^{(H)}S^N \to G\ltimes_H\mathscr{F}_M^{(H)} S^0 \]
from Example \ref{exa:lambda}.
%In general they are not $G$-flat cofibrations, in fact not even levelwise injective, as the case $G=H=\{e\}$, $M=\emptyset$ and $N=1$ shows. To correct this we recall that the mapping cylinder $\Cyl(f)$ of a morphism of $G$-symmetric spectra $f:X\to Y$ is defined as the pushout of $f$ along the inclusion $X\xr{i_1} [0,1]_+\wedge X$ (with $\Delta^1$ taking the place of the interval in the simplicial case).
We use the levelwise mapping cylinder $\Cyl(-)$ to factor $G\ltimes_H \lambda^{(H)}_{M,N}$ as \[ G\ltimes_H\mathscr{F}^{(H)}_{M\sqcup N}S^N\xr{G\ltimes_H \overline{\lambda}^{(H)}_{M,N}} G\ltimes_H \Cyl(\lambda^{(H)}_{M,N})\xr{G\ltimes_H r^{(H)}_{M,N}} G\ltimes_H \mathscr{F}^{(H)}_M S^0. \]
Here, the map $G\ltimes_H r^{(H)}_{M,N}$ is a $G$-homotopy equivalence and $\overline{\lambda}^{(H)}_{M,N}$ is a $G^{\U}$-projective cofibration. The latter is a formal consequence of the fact that $G\ltimes_H\mathscr{F}_{M\sqcup N}^{(H)}S^N$ and $G\ltimes_H \mathscr{F}_M^{(H)} S^0$ are $G^{\U}$-projective as well as the $G^{\U}$-projective level model structures being topological (or simplicial). For example, this is explained in the proof of \cite[Lem. 3.4.10]{HSS00}.

We define \[  J_{\U,proj}^{st}\defeq \{i\square (G\ltimes_H\overline{\lambda}^{(H)}_{M,N})\ |\ i\in I_{\{e\}},\ H\leq G,\ M,N\subseteq_H\U\}\cup J_{\U,proj}^{lev}. \]
Again, $\square$ denotes the pushout product with respect to the smash product and $I_{\{e\}}$ is a set of generating cofibrations of the Quillen model structures on non-equivariant topological spaces or simplicial sets.
\begin{Lemma} For a $G^{\U}$-projectively level fibrant $G$-symmetric spectrum $X$, $H$ a subgroup of $G$ and $M$ and $N$ two finite $H$-subsets of $\U$ the following are equivalent:
\begin{enumerate}
\item The map $(\widetilde{\sigma}_M^N)^H:X(M)^H\to \map_H(S^N,X(M\sqcup N))$ is a weak homotopy equivalence.
\item $X$ has the right lifting property with respect to the set $\{i\square (G\ltimes_H \overline{\lambda}^{(H)}_{M,N})\}_{i\in I_{\{e\}}}$.
\end{enumerate}
\end{Lemma}
\begin{proof} The argument is the same as in the non-equivariant case, cf. \cite[Lem. 3.4.12 and Cor. 3.4.13]{HSS00}.
%By adjunction, $X$ has the right lifting property with respect to $\{i\square (G\ltimes_H \overline{\lambda}_{M,N}^{(H)})\}_{i\in I_{\{e\}}}$ if and only if 
%\[ \map_{GSp^\Sigma}(G\ltimes_H\overline{\lambda}_{M,N}^{(H)},X): \map_{GSp^\Sigma}(G\ltimes_H \Cyl(\lambda^{(H)}_{M,N}),X)\to \map_{GSp^\Sigma}(G\ltimes_H\mathscr{F}^{(H)}_{M\sqcup N}S^N,X) \]
%has the right lifting property with respect to the set $I_{\{e\}}$. Since, by Proposition \ref{prop:levmonoidal}, the level model structure is topological (or simplicial), this map is always a Serre-(or Kan-)fibration. Hence, it has the right lifting property with respect to $I_{\{e\}}$ if and only if it is a weak homotopy equivalence. Since $G\ltimes_H r^{(H)}_{M,N}$ is a $G$-homotopy equivalence, this in turn is equivalent to
%\[ \map_{GSp^\Sigma}(G\ltimes_H \mathscr{F}_M S^0,X)\xr{\map_{GSp^\Sigma}(G\ltimes_H\lambda^{(H)}_{M,N},X)} \map_{GSp^\Sigma}(G\ltimes_H \mathscr{F}_{M\sqcup N}S^N,X) \]
%being a weak homotopy equivalence. Via two adjunctions this map is isomorphic to
%\[ X(M)^H\xr{(\widetilde{\sigma}_M^N)^H} \map_H(S^N,X(M\sqcup N)), \]
%so we are done.
\end{proof}
\begin{Cor} A $G$-symmetric spectrum of spaces or simplicial sets is a $G^{\U}\Omega$-spectrum if and only if it has the right lifting property with respect to the set $J_{\U,proj}^{st}$.
\end{Cor}
%\begin{proof} Since we already know (cf. Section \ref{sec:glev}) that a $G$-symmetric spectrum is $G^{\U}$-projectively level fibrant if and only if it has the right lifting property with respect to the set $J_{\U,proj}^{lev}$, this is a direct consequence of the previous lemma.
%\end{proof}

It follows from Proposition \ref{prop:levmonoidal}, i.e., the fact that the level model structures are $G$-topological (or $G$-simplicial), that every map in $J_{\U,proj}^{st}$ is a $G^{\U}$-projective cofibration. Also, all domains (and codomains) of maps in $J_{\U,proj}^{st}$ are small with respect to countably infinite sequences of flat cofibrations. Hence we can apply the small object argument (cf. \cite[10.5.16]{Hir03}) to obtain a functor $Q^{\U}:GSp^\Sigma\to GSp^\Sigma$ with image in $G^{\U}\Omega$-spectra and a natural transformation $q^{\U}:id\to Q^{\U}$.

We are left to show that for every $G$-symmetric spectrum $X$ the map $q^{\U}_X:X\to Q^{\U}X$ is a $G^{\U}$-stable equivalence.
We first prove the following:
\begin{Lemma} \label{lem:gpushoutproduct} Let $i:A\to B$ be a genuine $G$-cofibration of based $G$-spaces (or based $G$-simplicial sets) and $f:X\to Y$ a $G$-flat cofibration of $G$-symmetric spectra which is also a $G^{\U}$-stable equivalence. Then the pushout product map
\[ i\square f:(B\wedge X)\cup_{A\wedge X} (A\wedge Y)\to B\wedge Y \]
is again a $G$-flat cofibration and a $G^{\U}$-stable equivalence. 
\end{Lemma}
\begin{proof} We already know that $i\square f$ is again a $G$-flat cofibration, since the level model structure is based $G$-topological (or based $G$-simplicial) by Proposition \ref{prop:levmonoidal}. Hence, using Proposition \ref{prop:gwedge} and the fact that for $G$-flat cofibrations the cone is $G^{\U}$-level equivalent to the cofiber, we can prove that $i\square f$ is a $G^{\U}$-stable equivalence by showing that its cofiber 
\[ B\wedge Y/((B\wedge X)\cup_{A\wedge X} (A\wedge Y))\cong (B/A)\wedge (Y/X)\]
is $G^{\U}$-stably contractible. Since $f$ is a $G$-flat cofibration and a $G^{\U}$-stable equivalence, the $G$-symmetric spectrum $Y/X$ is $G^{\U}$-stably contractible. Hence, so is its smash product with the cofibrant based $G$-space $B/A$ (by Proposition \ref{prop:gwedge}). So $i\square f$ is a $G^{\U}$-stable equivalence.
\end{proof}

Using this we obtain:
\begin{Prop} The map $q^{\U}_X:X\to Q^{\U}X$ is a $G^{\U}$-stable equivalence for every $G$-symmetric spectrum $X$.
\end{Prop}
\begin{proof} The map $q_X$ is a relative $J_{\U,proj}^{st}$-complex. Every map in $J_{\U,proj}^{st}$ is in particular a $G$-flat cofibration and $G^{\U}$-stable equivalence (for those of the form $G\ltimes_H \overline{\lambda}_{M,N}^{(H)}$ this is a consequence of Lemma \ref{lem:gpushoutproduct}), so the result follows from Proposition \ref{prop:gwedge}.
%, it suffices to show that every map in $J_{\U,proj}^{st}$ is a $G^{\U}$-stable equivalence and a $G$-flat cofibration. Every map in $J_G^{lev}$ is even a $G^{\U}$-level equivalence and a $G^{\U}$-projective cofibration, since they are generators for the $G^{\U}$-level model structure. All the maps $G\ltimes_H \overline{\lambda}_{M,N}^{(H)}$ are $G^{\U}$-stable equivalences by Example \ref{exa:lambda} and $G^{\U}$-projective cofibrations. Hence, the statement follows from Lemma \ref{lem:gpushoutproduct}, as every $G^{\U}$-projective cofibration is in particular a $G$-flat cofibration.
\end{proof}

We are now ready to establish the $G^{\U}$-stable model structures for $G$-symmetric spectra of spaces and of simplicial sets. We call a map a (positive) \emph{flat $G^{\U}$-stable fibration} if it has the right lifting property with respect to all maps that are (positive) $G$-flat cofibrations and $G^{\U}$-stable equivalences. (Positive) \emph{projective $G^{\U}$-stable fibrations} are analogously defined.
%if it has the right lifting property with respect to all maps that are (positive) projective $G^{\U}$-cofibrations and $G^{\U}$-stable equivalences. Then we have:

\begin{Theorem}[Flat stable model structures] \label{theo:flatmod} The classes of (positive) $G$-flat cofibrations, $G^{\U}$-stable equivalences and (positive) $G^{\U}$-stable fibrations define a cofibrantly generated, proper model structure on the categories of $G$-symmetric spectra of spaces and simplicial sets, called the (positive) $G^{\U}$-flat stable model structure.
\end{Theorem}
\begin{Theorem}[Projective stable model structures] \label{theo:projmod} The classes of (positive) projective $G^{\U}$-cofibrations, $G^{\U}$-stable equivalences and (positive) projective $G^{\U}$-stable fibrations define a cofibrantly generated, proper model structure on the categories of $G$-symmetric spectra of spaces and simplicial sets, called the (positive) $G^{\U}$-projective stable model structure.
\end{Theorem}
\begin{proof} All the model structures are obtained in the same way via left Bousfield localization of the respective level model structure,. We apply Theorem \cite[Thm. 9.3]{Bou01}, with respect to the functor $Q^{\U}$ and the natural transformation $q^{\U}:id\to Q^{\U}$ just constructed. Any $G^{\U}$-stable equivalence between $G^{\U}\Omega$-spectra is a $G^{\U}$-level equivalence by the Yoneda Lemma and thus a map $f$ is a $G^{\U}$-stable equivalence if and only if $Q^{\U}f$ is a $G^{\U}$-level equivalence (and if and only if $Q^{\U}f$ is a positive $G^{\U}$-level equivalence). In other words the class of $G^{\U}$-stable equivalences agrees with that of $Q^{\U}$-equivalences in the sense of Bousfield's theorem.

We now check the axioms in \cite{Bou01}. (A1) requires that every $G^{\U}$-level equivalence is a $G^{\U}$-stable equivalence, which is clear. For every $G$-symmetric spectrum $X$ the maps $q^{\U}_{QX}$ and $Q^{\U}q^{\U}_X$ from $Q^{\U}X$ to $Q^{\U}Q^{\U}X$ are $G^{\U}$-stable equivalences between $G^{\U}\Omega$-spectra and hence $G^{\U}$-level equivalences, so (A2) is satisfied. 
In (A3) we have to show that the pullback of a $G^{\U}$-stable equivalence along a $G^{\U}$-level fibration between level-fibrant (positive) $G^{\U}\Omega$-spectra is again a $G^{\U}$-stable equivalence. This is a direct application of the dual version of item $4$ in Proposition \ref{prop:gwedge} (and does not need the $G^{\U}\Omega$ spectrum hypothesis).
\begin{comment}
In (A3) we are given a pullback square \[
	\xymatrix{V\ar[r]^k \ar[d]_g & X \ar[d]^f\\
						W\ar[r]_h & Y}
\]
where $X$ and $Y$ are level-fibrant (positive) $G^{\U}\Omega$-spectra, f is a $G^{\U}$-level fibration and $h$ a $G^{\U}$-stable equivalence. We have to show that then $k$ is a $G^{\U}$-stable equivalence, too. This is a direct application of the dual version of Proposition \ref{prop:pushout} (and does not need the hypothesis that $X$ and $Y$ are (positive) $G^{\U}\Omega$-spectra).
\end{comment}
\end{proof}
%\todo{leave out or shorten?}
Using the characterization given in \cite[Thm. 9.3]{Bou01}, we see that
%a map $f:X\to Y$ is a $G^{\U}$-flat (resp. $G^{\U}$-projective) stable fibration if and only if it is a $G^{\U}$-flat (resp. $G^{\U}$-projective) level fibration and the square \[ \xymatrix{X \ar[r]^-{q^{\U}_X} \ar[d]_f & Q^{\U}X \ar[d]^{Q^{\U}f} \\
%							Y \ar[r]_-{q^{\U}_Y} & Q^{\U}Y} \]
%is homotopy cartesian in the respective level model structure. This shows that
generating acyclic cofibrations can be obtained by adding the maps $i\square (G\ltimes_H \overline{\lambda}_{M,N})$ for all $H$-subsets $M$ and $N$ of $\U$ (with $M\neq \emptyset$ in the positive case) to the respective set of generating acyclic cofibrations for the level model structure (while nothing needs to be added to the generating cofibrations). Moreover, a $G$-symmetric spectrum is fibrant in either of the (positive) stable model structures if and only if it is a level-fibrant (positive) $G^{\U}\Omega$-spectrum.

%Finally, we have the following commutative square of identity Quillen equivalences, where the $+$ denotes the positive model structures:
%\[ \xymatrix{ GSp^{\Sigma}_{\U,proj,+} \ar@<1ex>[r] \ar@<-1ex>[d] & GSp^{\Sigma}_{\U,proj}\ar@<1ex>[l] \ar@<-1ex>[d]\\
%							GSp^{\Sigma}_{\U,flat,+} \ar@<1ex>[r] \ar@<-1ex>[u] & GSp^{\Sigma}_{\U,flat}\ar@<1ex>[l] \ar@<-1ex>[u]} 
%\]
%Here, the arrows going from the left to the right and from the top to the bottom denote left Quillen functors.

\subsection{Some properties of the homotopy category of $G$-symmetric spectra}
%\todo{shorten! first three propositions as one?}
We list certain properties of the homotopy category $\Ho^{\U}(GSp^{\Sigma})$ of the $G^{\U}$-stable model structures that we need later for the comparison with $G$-orthogonal spectra. %First we check that they are in fact stable in the sense of model categories, i.e., that the derived suspension functor is an auto-equivalence. As a consequence, the homotopy category $\Ho^{\U}(GSp^{\Sigma})$ has an induced structure of a triangulated category (cf. \cite[Chapter 7]{Hov99}). More generally, the homotopy category is stable with respect to all representation spheres $S^V$ for representations $V$ which embed into $\R[\U]$:
\begin{Prop}
\begin{enumerate}
\item(Stability) Let $V$ be a finite dimensional $G$-representation which embeds into $\R[\U]$. Then the adjunction
\[ S^V\wedge (-):GSp_{\T}^{\Sigma} \rightleftarrows GSp^{\Sigma}_{\T}:\Omega^V \]
%\[ \xymatrix{GSp^{\Sigma}_{\T}\ar@/^/[rr]^{S^V\wedge -} && GSp^{\Sigma}_{\T}\ar@/^/[ll]^{\Omega^V}} \]
is a Quillen equivalence for either of the $\U$-stable model structures of the previous section. In particular, the $\U$-model structures are stable and so their homotopy categories inherit a triangulated structure.
\item(True homotopy groups) The $G$-symmetric spectra $G/H_+\wedge \mathbb{S}^n$ represent the \emph{true} homotopy groups, i.e., for $H$ a subgroup of $G$ and a $G^{\U}$-semistable $G$-symmetric spectrum $X$ there is a natural isomorphism \[	\Ho^{\U}(GSp^{\Sigma})(G/H_+\wedge \mathbb{S}^n,X)\cong \pi^{H,\U}_n(X). \]
\item(Generators) The suspension spectra of $G$-orbits $\{G/H_+\wedge \mathbb{S}\ |\ H\leq G\}$ form a set of compact generators of the triangulated homotopy category $\Ho^{\U}(GSp^{\Sigma})$.
\end{enumerate}
\end{Prop}
%See REF for a definition of compact generators of a triangulated category.
\begin{proof} $(i)$: This is a consequence of the fact that both functors preserve and reflect all $G^{\U}$-stable equivalences (Proposition \ref{prop:suspiso2}) and that adjunction unit and counit are $\upi_*^{\U}$-isomorphisms (Proposition \ref{prop:suspiso}) and thus $G^{\U}$-stable equivalences.

$(ii)$: We can restrict to the case $n=0$ and $X$ a $G^{\U}\Omega$-spectrum, since every $G^{\U}$-semistable $G$-symmetric spectrum can be replaced by one up to $\upi_*^{\U}$-isomorphism. The statement then follows by adjunction, since
\[ \Ho^{\U}(GSp^{\Sigma})(\Sigma^{\infty}G/H_+,X)\cong [G/H_+,X_0]^G=\pi_0 (X_0^H)\cong \pi_0^{H,\U}X, \]
where we have used the non-positive $G^{\U}$-projective stable model structure to compute the morphism set in the homotopy category. The last step uses Example \ref{exa:piomega}.

$(iii)$: Using item $(ii)$, this follows from item $3$ of Proposition \ref{prop:exseq} and the fact that equivariant homotopy groups detect $G^{\U}$-stable equivalences between $G^{\U}$-semistable $G$-symmetric spectra and hence in particular $G^{\U}\Omega$-spectra.
%By Proposition \ref{prop:stablemonoidal}, the stable model structures are $G$-topological, therefore the adjunction is a Quillen pair. By Proposition \ref{prop:cone} we know that both functors preserve and reflect all $G^{\U}$-stable equivalences and that adjunction unit and counit are $\upi_*^{\U}$-isomorphisms (Proposition \ref{prop:suspiso}), thus $G^{\U}$-stable equivalences by Theorem \ref{theo:piiso}. Hence, they form a Quillen equivalence.
\end{proof}

\section{Derived functors}
\label{sec:derived}
In this section we explain how the model structures can be used to derive various functors between categories of equivariant symmetric spectra.

\subsection{Change of universe}
\label{sec:changeofuniverse}
The $\U$-stable homotopy categories $\Ho^{\U}(GSp^{\Sigma})$ for varying $G$-set universes $\U$ are related by change of universe functors.
For this we let $\U,\U'$ be $G$-set universes such that $\U$ allows an embedding into $\U'$.
%The identity adjunction can be derived with respect to the $G^{\U}$- and $G^{\U'}$-stable equivalences in both directions, one of which is classical and one seems to not have appeared in the literature.
%We start with the classical one.
It follows immediately from the definitions that every $G^{\U'}$-level equivalence is also a $G^{\U}$-level equivalence and the same is true for projective level fibrations. Moreover, every $G^{\U'}\Omega$-spectrum is in particular a $G^{\U}\Omega$-spectrum and so we see that the identity functor is a right Quillen functor from the $G^{\U'}$-projective stable model structure to the $G^{\U}$-projective stable model structure. Hence, we obtain a Quillen pair
\begin{equation} \label{eq:change1} id:GSp^\Sigma_{\U,proj} \rightleftarrows GSp^{\Sigma}_{\U',proj}:id
\end{equation}
and a derived adjunction between the homotopy categories. We note that the Quillen pair does not depend on a choice of embedding $\U\hookrightarrow \U'$.

If the $\R$-linearizations of $\U$ and $\U'$ are isomorphic, it follows by 2-out-of-3 for the Quillen equivalences from Theorem \ref{theo:gquillen} that the above adjunction is a Quillen equivalence. In fact, less is necessary: It is a consequence of \cite[Thm. 1.2]{Lew95}, that we obtain a Quillen equivalence already if every $G$-orbit $G/H$ which embeds into $\R[\U']$ also embeds into $\R[\U]$. 

\begin{Remark} The identity adjunction is usually not a Quillen pair for the flat stable model structures, since a $G^{\U}$-stable equivalence between $G$-flat $G$-symmetric spectra is not necessarily a $G^{\U'}$-stable equivalence.
%Put differently, since the flat cofibrations are independent of the universe it may seem that the $G^{\U'}$-flat stable model structure is a left Bousfield localization of the $G^{\U}$-flat stable model structure, but this is not the case. In general neither of the derived functors are fully-faithful.
\end{Remark}

\subsection{Change of groups} %\todo{restrictions for induced H-universes?}
\label{sec:changeofgroups} Let $H\leq G$ be a subgroup inclusion. Then the restriction functor $\res_H^G:GSp^{\Sigma}\to HSp^{\Sigma}$ preserves all the structure one could ask for, namely it maps $G$-flat cofibrations ($G^{\U}$-flat stable fibrations) to $H$-flat cofibrations (resp. $H^{\U}$-flat stable fibrations) and all $G^{\U}$-stable equivalences to $H^{\U}$-stable equivalences, and similarly for the projective model structure and their positive versions. Hence, it is both a right and a left Quillen functor and one obtains Quillen pairs
\[ G\ltimes_H(-):HSp^{\Sigma}\rightleftarrows GSp^{\Sigma}:{\res}_H^G \]
and
\[ {\res}_H^G:GSp^{\Sigma} \rightleftarrows HSp^{\Sigma}:\map_H(G,-) \]
with respect to either the flat or the projective $\U$-stable model structures, positive and non-positive. The Wirthm{\"u}ller isomorphism (Proposition \ref{prop:wirth}) implies that if $G/H$ embeds into $\R[\U]$, the derived functors $\mathbb{L}(G\ltimes_H -)$ and $\mathbb{R}(\map_H(G,-))$ are naturally isomorphic.

\subsection{Categorical fixed points}
Let $N$ be a normal subgroup of $G$, $p:G\to G/N$ the projection, $\U_{G/N}$ a $G/N$-set universe and $p^*(\U_{G/N})$ the pulled-back $G$-set universe. Then the fixed point adjunction
\[ p^*:(G/N)Sp^{\Sigma}\rightleftarrows GSp^{\Sigma}:(-)^{N} \]
becomes a Quillen pair for the $(G/N)^{\U_{G/N}}$-projective ($(G/N)^{\U_{G/N}}$-flat) model structure on $(G/N)Sp^{\Sigma}$ and the $G^{p^*(\U_{G/N})}$-projective (resp. $G^{p^*(\U_{G/N})}$-flat) model structure on $GSp^{\Sigma}$.

\subsection{Orbits}
Let $G$, $N$, $p$ and $\U_{G/N}$ be as above. Then the $N$-orbits adjunction
\[ (-)/N:GSp^{\Sigma}\rightleftarrows (G/N)Sp^{\Sigma}:p^* \]
is a Quillen pair for the $G^{p^{*}(\U_{G/N})}$-projective model structure on the left and the $(G/N)^{\U_{G/N}}$-projective model structure on the right, and likewise for the flat model structures.

\section{Monoidal properties and the norm}
\label{sec:monoidal}

In this section we deal with the relationship between the smash product and the stable model structures, constructing model structures on categories of modules, algebras (both via \cite{SS00}) and commutative algebras (using \cite{Whi17}). We also explain homotopical properties of the multiplicative norm.

We start with:
\begin{Prop}\label{prop:stablemonoidal} The (positive and non-positive) flat and projective $G^{\U}$-stable model structures on $G$-symmetric spectra are monoidal with respect to the smash product. 
\end{Prop}
\begin{proof} Since the stable model structure shares the cofibrations with the respective level model structure and every $G^{\U}$-stably acyclic projective cofibration is also a $G^{\U}$-stably acyclic flat cofibration, we only have to show that the pushout product $f\square g$ of a $G^{\U}$-stably acyclic flat cofibration $f:A\to B$ with a flat cofibration $C\to D$ is again a $G^{\U}$-stable equivalence. The proof is the same as for Lemma \ref{lem:gpushoutproduct}, this time making use of the fact that the $\Hom$-spectrum from a flat $G$-symmetric spectrum into a $G^{\U}$-flat level fibrant $G^{\U}\Omega$-spectrum is again a flat level fibrant $G^{\U}\Omega$-spectrum (which is Corollary \ref{cor:homomega}).

To establish monoidality for the positive model structures we furthermore need to show that for a positive $G^{\U}$-projective cofibrant replacement $\mathbb{S}^+\to \mathbb{S}$ and every positive flat cofibrant $G$-symmetric spectrum $X$ the map $(\mathbb{S}^+\wedge X)\to X$ is a $G^{\U}$-stable equivalence, but this is clear because every positive $G^{\U}$-level equivalence is a $G^{\U}$-stable equivalence and $\mathbb{S}$ and $\mathbb{S}^+$ are in particular flat.
%\todo{argument unklar, was ist X?}
\end{proof}

\subsection{The monoid axiom and model structures on module and algebra categories}
In order to obtain model structures on modules and algebras over an arbitrary $G$-symmetric ring spectrum we need to show one more property called the monoid axiom (cf. \cite[Def 3.3]{SS00}). The main ingredient is the following:
\begin{Prop}[Flatness] \label{prop:flatness}
\begin{enumerate}[(i)]
 \item Smashing with a $G$-flat $G$-symmetric spectrum preserves $\upi_*^{\U}$-isomorphisms and $G^{\U}$-stable equivalences.
 \item Smashing with any $G$-symmetric spectrum preserves $\upi_*^{\U}$-isomorphisms and $G^{\U}$-stable equivalences between $G$-flat $G$-symmetric spectra.
\end{enumerate}
\end{Prop}
Before we come to showing this we need one lemma. We recall that $\mathscr{G}_n(-)$ denotes the semi-free $G$-symmetric spectrum functor in level $n$ (cf. Section \ref{sec:free}).
\begin{Lemma} \label{lem:flat} Let $n\in \N$, $A$ a cofibrant $(G\times \Sigma_n)$-space and $Y$ a $G$-symmetric spectrum with $\upi_*^{\U}Y=0$. Then $\upi_*^{\U}(\mathscr{G}_n(A)\wedge Y)=0$.
\end{Lemma}
\begin{proof} It suffices to prove the lemma for $G$-symmetric spectra of spaces. The $(n+m)$-th level of the $G$-symmetric spectrum $\mathscr{G}_n(A)\wedge Y$ is given by $\Sigma_{n+m}\ltimes_{\Sigma_n\times \Sigma_m} (A\wedge Y_m)$. We let $\varphi:G\to \Sigma_{n+m}$ be a homomorphism such that the associated $G$-set $\underline{n+m}_{\varphi}$ (which we denote by $M$ from now on) embeds into $\U$, and $f:S^{k\sqcup M}\to (\mathscr{G}_n(A)\wedge Y)(M)$ a $G$-map. We have to show that $f$ is stably null-homotopic. Applying the double coset formula (cf. Section \ref{sec:unstable}), we see that $(\mathscr{G}_n(A)\wedge Y)(M)$ splits off as
\[ \bigvee_{[N\subseteq M=\underline{n+m},|N|=n]} G\ltimes_{Stab(N)} (A\wedge Y(M-N)), \]
where the wedge is taken over a system of $G$-orbit representatives $N$ of subsets of $M$ of cardinality $n$ and $Stab(N)\subseteq G$ is the stabilizer of such a subset. The stabilizer fixes both $N$ and the complement $M-N$, so it acts on $Y(M-N)$. The action on $A$ comes from pulling back the $(G\times \Sigma_n)$-action along the graph of the homomorphism $Stab(f)\to \Sigma_N\cong \Sigma_n$ (where the latter is induced by the canonical order-preserving bijection $\underline{n}\cong N\subseteq \underline{n+m}$). We now fix such an $n$-element subset $N$ and consider the $G$-map $f':S^{k\sqcup M}\to (\mathscr{G}_n(A)\wedge Y)(M)\to G\ltimes_{Stab(N)} (A\wedge Y(M-N))$, i.e., $f$ followed by the projection to this summand. It represents an element in $\pi_k^{G,\U}(\Omega^M (G\ltimes_{Stab(N)}(A\wedge Y)))$, which is trivial by Lemma \ref{lem:wirth2} and Corollary \ref{cor:smashpi}. Hence, there exists a finite $G$-set $M'$ such that $\sigma_{M-N}^{M'}\circ (f'\wedge S^{M'})$ is $G$-nullhomotopic. Taking $M'
$ large enough so that this holds for all subsets $N$ at once, we obtain that $\sigma_M^{M'}\circ (f\wedge S^{M'})$ has the property that all postcompositions with projections to the summands are $G$-nullhomotopic. Using item 3 of Proposition \ref{prop:exseq} we see that this implies that a suspension of $f$ itself is $G$-nullhomotopic and so we are done.
\end{proof}

\begin{proof}[Proof of Proposition \ref{prop:flatness}] To $(i)$: We first consider the statement about homotopy groups. By the long exact sequence of the mapping cone (Proposition \ref{prop:exseq}) it suffices to show that $X\wedge Y$ has trivial homotopy groups whenever $X$ is $G$-flat and $Y$ has trivial homotopy groups. This follows from Lemma \ref{lem:flat} via an induction over the skeleton filtration of $X$, using that the smash product with $Y$ takes $G$-flat cofibrations to $h$-cofibrations (cf. \cite[Prop. 12.3]{MMSS01} for the non-equivariant analog).

To obtain the statement on $G^{\U}$-stable equivalences we use that we can replace any $G$-symmetric spectrum $Y$ by a $G$-flat one $Y^{\flat}$ up to $G^{\U}$-level equivalence. As we just saw, smashing with $X$ takes $G^{\U}$-level equivalences to $\upi_*^{\U}$-isomorphisms and hence $G^{\U}$-stable equivalences. So the claim follows from monoidality of the $G^{\U}$-flat stable model structure.

Statement $(ii)$ follows by 2-out-of-3 applied to a flat replacement.
\end{proof}

Now we come to the monoid axiom. For a $G$-symmetric spectrum $Y$ we denote by $\{ J_{st}^{\U}\wedge Y\}^{cell}$ the class of maps obtained via (transfinite) compositions and pushouts from maps of the form $j\wedge Y$, where $j$ is a $G$-flat cofibration and $G^{\U}$-stable equivalence.

\begin{Prop}[Monoid axiom] \label{prop:monoidaxiom} Every map in $\{ J_{st}^{\U}\wedge Y\}^{cell}$ is a $G^{\U}$-stable equivalence.
\end{Prop}
\begin{proof} By Proposition \ref{prop:flatness} we know that each map $j\wedge Y$ is a $G^{\U}$-stable equivalence and an $h$-cofibration. Since the class of $G^{\U}$-stable equivalences which are also $h$-cofibrations is closed under (transfinite) composition and pushouts, this gives the monoid axiom.
\end{proof}
This implies the monoid axiom \cite[Def. 3.3]{SS00} in all the positive and non-positive, projective and flat model structures, since every cofibration there is in particular a $G$-flat cofibration. Hence, by \cite[Thm. 4.1]{SS00} we obtain various model structures on the categories of modules and algebras, where the weak equivalences (and fibrations) are those maps which forget to a weak equivalence (resp. fibration) in the respective model structure on $G$-symmetric spectra.
\begin{Cor} \label{cor:modmodules} \begin{enumerate}[(i)]
             \item For every $G$-symmetric ring spectrum $R$ the positive and non-positive, $G^{\U}$-flat and $G^{\U}$-projective stable model structures lift to the category of $R$-modules. If $R$ is commutative, they are again monoidal.
	     \item For every commutative $G$-symmetric ring spectrum $R$ the positive and non-positive, $G^{\U}$-flat and $G^{\U}$-projective stable model structures lift to the category of $R$-algebras. Furthermore, every cofibration of $R$-algebras with cofibrant source is also a cofibration of $R$-modules.
\end{enumerate}
\end{Cor}
\begin{Remark} Strictly speaking, Theorem 4.1 of \cite{SS00} does not apply verbatim to the topological case, because not every $G$-symmetric spectrum of spaces is small. However, the domains and targets of the generating (acyclic) cofibrations are small with respect to (transfinite) compositions of $G$-flat cofibrations, and so the small object argument can nevertheless be applied (cf. \cite[Rem. 2.4]{SS00}).
\end{Remark}

\subsection{Homotopical properties of the norm} \label{sec:norm}
In this section we deal with the homotopical properties of the norm functor, a multiplicative version of induction whose analog on $G$-orthogonal spectra plays a major role in the solution of the Kervaire invariant one problem in \cite{HHR16}, where its properties are studied in detail (Sections A.4 and B.5). An analysis of the norm is also needed for the construction of model structures on strictly commutative $G$-symmetric ring spectra, as we will see in the next section.

The norm can be constructed as follows: We fix a subgroup inclusion $H\leq G$. For every $H$-symmetric spectrum $X$ and every $n\in \N$, the $n$-fold smash power $X^{\wedge n}$ has a natural action of the wreath product $\Sigma_n\wr H$, i.e., the semidirect product of $\Sigma_n$ and $H^n$ associated to the $\Sigma_n$-action permuting the factors. Taking $n$ to be the index of $H$ in $G$, any choice of a system of coset representatives $g_1,\hdots, g_n\in G$ of $G/H$ gives rise to an embedding $\varphi:G\to \Sigma_n\wr H;g\mapsto (\sigma(g),h_1(g),\hdots,h_n(g))$ characterized by the formula $g\cdot g_i=g_{\sigma(g)(i)}h_i(g)$. This embedding is independent of the chosen $g_i$ up to conjugation. The norm is then defined as the $G$-symmetric spectrum obtained by restricting $X^{\wedge n}$ along this embedding.

The norm functor is symmetric monoidal, hence it maps commutative $H$-symmetric ring spectra to commutative $G$-symmetric ring spectra, yielding a left adjoint of the restriction functor.

We now summarize the homotopical properties of these constructions. They are very clean over the projective model structures. Given an $H$-set universe $\U$, the $n$-fold disjoint union $n\times \U$ becomes a $(\Sigma_n\wr H)$-set universe and we have:
\begin{Theorem} \label{theo:norm1} For every $H$-set universe $\U$ and every natural number $n$ the functor $(-)^{\wedge n}:HSp^{\Sigma}\to (\Sigma_n\wr H)Sp^{\Sigma}$ maps $H^{\U}$-stable equivalences between $H^{\U}$-projective $H$-symmetric spectra to $(\Sigma_n\wr H)^{n\times \U}$-stable equivalences of $(\Sigma_n\wr H)^{n\times \U}$-projective spectra. In particular, the norm $N_H^G:HSp^{\Sigma}\to GSp^{\Sigma}$ maps $G^{\U}$-stable equivalences between $H^{\U}$-projective $H$-symmetric spectra to $G^{G\ltimes_H \U}$-stable equivalences.
\end{Theorem}
In particular, both the $n$-fold smash power and the norm can always be left derived, as any $H^{\U}$-projective replacement functor is a left deformation for them.

The behavior with respect to the flat model structure is more delicate. It is \emph{not} true that the functor $(-)^{\wedge n}:HSp^{\Sigma}\to (\Sigma_n\wr H)Sp^{\Sigma}$ takes $H^{\U}$-stable equivalences between $H$-flat $H$-symmetric spectra to $(\Sigma_n\wr H)^{n\times \U}$-stable equivalences, in fact not even to $(\Sigma_n\times H)^{n\times \U}$-stable equivalences. Nevertheless, under some conditions on the $H$-set universe $\U$, the norm functor \emph{is} homotopical on $H$-flat $H$-symmetric spectra. The precise conditions on $\U$ needed are quite complicated and for simplicity we restrict to the case where $\U$ is the full $N$-fixed $H$-set universe associated to a subgroup $N$ of $H$ which is normal in $G$. We denote this $H$-set universe, an infinite disjoint union of all orbits $H/K$ with $N\leq K$, by $\U_H(N)$. The case where $N$ is trivial gives the complete $H$-set universe. We then have:
\begin{Theorem} \label{theo:norm2} Let $N\leq H\leq G$ be subgroups with $N$ normal in $G$. Then the norm functor $N_H^G:HSp^{\Sigma}\to GSp^{\Sigma}$ takes $H^{\U_H(N)}$-stable equivalences between $H$-flat $H$-symmetric spectra to $G^{\U_G(N)}$-stable equivalences of $G$-flat $G$-symmetric spectra. 
\end{Theorem}
In particular, for universes of this form the effect of the derived norm can be computed on $H$-flat $H$-symmetric spectra. The point in this additional work will become clear in the next section, where it is used to show that the derived norm of a commutative $H$-symmetric ring spectrum is equivalent as a $G$-symmetric spectrum to the derived norm of the underlying $H$-symmetric spectrum (cf. \cite[Sec. B.8]{HHR16}).

\begin{Remark} \label{rem:trick} We note that if $f:X\to Y$ is an $H$-flat cofibration of $H$-symmetric spectra of simplicial sets, it follows immediately from monoidality of the non-equivariant flat model structure that both $f^{\wedge n}$ and $f^{\square n}$ are $(\Sigma_n\wr H)$-flat cofibrations and also $G$-flat cofibrations, since equivariant flat cofibrations are exactly those morphisms whose underlying non-equivariant morphism is a flat cofibration. In particular, this gives the part about cofibrations of Theorem \ref{theo:norm2} in the simplicial case and also for all morphisms of $H$-symmetric spectra of spaces which are the geometric realization of a morphism of $H$-symmetric spectra of simplicial sets, in particular for (wedges of) generating cofibrations.
\end{Remark}

In the proofs of Theorems \ref{theo:norm1} and \ref{theo:norm2} we will make use of the distributive law of \cite[Sec. A.3.3.]{HHR16}, which says that both the $n$-fold smash power $(\bigvee_I X_i)^{\wedge n}$ and the norm $N_H^G (\bigvee_I X_i)$ of a wedge $\bigvee_{I} X_i$ can be rewritten as a wedge of induced smash powers respectively norms. 
We quickly recall it in the versions we need and start with the case of smash powers. For this we choose a linear ordering on the index set $I$. Given a monotone function $f:\underline{n}\to I$, we let $a_1,\hdots,a_k$ be the different values of $f$ and $n_i\in \N_{>0}$ the order of their respective preimages. Then the distributive law asserts that there is a natural $(\Sigma_n\wr H)$-equivariant decomposition
\[  (\bigvee_{i\in I} X_i)^{\wedge n}\cong \bigvee_{f:\underline{n}\to I, \text{monotone}} \Sigma_n\wr H\ltimes_{\Sigma_{n_1}\wr H\times \hdots \times \Sigma_{n_k}\wr H} (X_{a_1}^{\wedge n_1}\wedge \hdots \wedge X_{a_k}^{\wedge n_k}). \]
In the case of the norm, we need to replace monotone functions $\underline{n}\to I$ by $G$-orbits of all functions $f:G/H\to I$. Given such an $f$, we let $K$ be the stabilizer and $g_1,\hdots,g_l$ a system of representatives for double cosets $K\backslash G/H$. Then there is a $G$-isomorphism
\[ N_H^G (\bigvee_{i\in I} X_i) \cong \bigvee_{[f:G/H\to I]} G\ltimes_K(\bigwedge_{i=1}^k N_{K\cap g_iHg_i^{-1}}c_{g_i}^{*} {\res}^H_{g_i^{-1}Kg_i\cap H} X_{f(g_i)}).
\]
\begin{Remark} The above formulas hold for $H$-objects in any symmetric monoidal category for which the monoidal product distributes over the coproduct. In particular, they hold for the pushout product $\square$ of morphisms of $H$-symmetric spectra.
\end{Remark}
We now move towards the proofs of Theorems \ref{theo:norm1} and \ref{theo:norm2} and start with the following:
\begin{Prop} \label{prop:times} Let $G$ and $K$ be finite groups, $\U_G$ a $G$-set universe, $\U_K$ a $K$-set universe, $i:A\to B$ a $G^{\U_G}$-projective cofibration and $j:X\to Y$ a $K^{\U_K}$-projective cofibration. Then the pushout product $i\square j$ is a $(G\times K)^{\U_G\sqcup \U_K}$-projective cofibration. If either $i$ or $j$ is positive, so is $i\square j$. If $i$ is a $G^{\U_G}$-projective cofibration and $G^{\U_G}$-stable equivalence and $j$ is a $K^{\U_K}$-projective cofibration, then $i\square j$ is a $(G\times K)^{\U_G\sqcup \U_K}$-stable equivalence.
\end{Prop}
\begin{proof} By Section \ref{sec:changeofuniverse}, $i$ is a (positive, acyclic) $(G\times K)^{\U_G\sqcup \U_K}$-projective cofibration when equipped with trivial $K$-action, and similarly for $j$. So the result follows from monoidality of the $(G\times K)^{\U_G\sqcup \U_K}$-stable model structure, cf. Proposition \ref{prop:stablemonoidal} (and Proposition \ref{prop:levmonoidal} for the smash product of a positive and non-positive projective cofibration).
\begin{comment}
The statement about cofibrations can be checked on generating cofibrations, where it follows easily using the natural $(G\times K)$-isomorphism $\mathscr{F}_{M}(-)\wedge \mathscr{F}_{N}(-)\cong \mathscr{F}_{M\sqcup N}(-\wedge -)$. 

Hence, if $i$ is a $G^{\U}$-stable equivalence and $G^{\U_G}$-projective cofibration and $i$ is a $K^{\U_K}$-projective cofibration, we already know that $i\square j$ is a $(G\times K)^{\U_G\sqcup \U_K}$-projective cofibration and it suffices to show that the $(G\times K)^{\U_G\sqcup \U_K}$-projective quotient $B/A\wedge Y/X$ is $(G\times K)^{\U_G\sqcup \U_K}$-stably contractible. We thus fix a $(G\times K)^{\U_G\sqcup \U_K}\Omega$-spectrum $Z$ and have to show that $[B/A\wedge Y/X,Z]^{G\times K}\cong [B/A,\Hom (Y/X,Z)^K]^G$ is trivial. The $K$-symmetric spectrum $Y/X$ is $\U_K$-projective, hence it is $(G\times K)^{\U_G\sqcup \U_K}$-projective when given the trivial $G$-action and so Corollary \ref{cor:homomega} implies that $\Hom (Y/X,Z)$ is again a level-fibrant $(G\times K)^{\U_G\sqcup \U_K}\Omega$-spectrum.
Since $B/A$ is $G^{\U}$-stably contractible and $G^{\U_G}$-projective, it hence suffices to show that the $K$-fixed points of every $(G\times K)^{\U_G\sqcup \U_K}\Omega$-spectrum are a $G^{\U_G}\Omega$-spectrum, which follows from the definition.
\end{comment}
\end{proof}
\begin{Remark} This proposition does not hold over the flat model structures. 
\end{Remark}

Theorem \ref{theo:norm1} is then a consequence of the following proposition, via Ken Brown's Lemma:
\begin{Prop} \label{prop:norm1} Let $f:X\to Y$ be an $H^{\U}$-projective cofibration of $H^{\U}$-projective $H$-symmetric spectra and $n\in \N$. Then $f^{\wedge n}:X^{\wedge n}\to Y^ {\wedge n}$ is a projective $(\Sigma_n \wr H)^{n\times \U}$-cofibration. If $f$ is in addition an $H^{\U}$-stable equivalence, then $f^{\wedge n}$ is a $(\Sigma_n \wr H)^{n\times \U}$-stable equivalence.
\end{Prop}
\begin{proof} It suffices to prove the topological case. We start by showing the claim for the generating (stable acyclic) cofibrations. The generating cofibrations are of the form $H\ltimes_K \mathscr{F}_M(i:A\to B)$ for $i$ a genuine $K$-cofibration of based $K$-CW complexes. Then $(H\ltimes_K \mathscr{F}_M(i))^{\wedge n}\cong (\Sigma_n\wr H)\ltimes_{\Sigma_n\wr K}(\mathscr{F}_{M^n}(i^{\wedge n}))$, which is an $H^{n\times \U}$-projective cofibration, since $i^{\wedge n}$ is a genuine $(\Sigma_n\wr K)$-cofibration (and, for later reference, so is the $n$-fold pushout product $i^{\square n}$). Furthermore, if $i$ is also a genuine $K$-equivalence, then $i^{\wedge n}$ and $i^{\square n}$ are genuine $(\Sigma_n\wr K)$-equivalences. Hence, it remains to consider the generating stable acyclic cofibrations added from the level to the stable model structure. They are of the form $H\ltimes_K (\lambda^{(K)}_{M,N}: F_{M\sqcup N} S^N\to F_{M} S^0)$ followed by an $H$-homotopy equivalence.
Hence, it suffices to show that $(H\ltimes_K \lambda^{(K)}_{M,N})^n$ is a $(\Sigma_n\wr H)^{n\times \U}$-stable equivalence. But this follows from the identification $(H\ltimes_K \lambda^{(K)}_{M,N})^{\wedge n}\cong (\Sigma_n\wr H)\ltimes_{\Sigma_n\wr K}(\lambda^{(\Sigma_n\wr K)}_{n\times M,n\times N})$.

In order to finish the proof we have to show the following:
\begin{enumerate}[(i)]
\item For every family $\{j_i:X_i\to Y_i\}_{i\in I}$ of generating $H^{\U}$-projective (stable acyclic) cofibrations, the map $(\bigvee_{I} j_i)^{\wedge n}$ is a $(\Sigma_n\wr H)^{n\times \U}$-projective (stable acyclic) cofibration.
\item Given a pushout square \[ \xymatrix{ A \ar[r] \ar[d]_-i & X \ar[d]^-f\\
																					B \ar[r] & Y,} \]
where $i$ is a wedge of generating $H^{\U}$-projective (stable acyclic) cofibrations, the map $f^{\wedge n}:X^{\wedge n}\to Y^{\wedge n}$ is a $(\Sigma_n\wr H)^{n\times \U}$-projective (stable acyclic) cofibration.
\item Given a sequence $X_0\to X_1\to \hdots $ of $H^{\U}$-projective (stable acyclic) cofibrations, for which we already know that each individual map $(X_i)^{\wedge n}\to (X_{i+1})^{\wedge n}$ is a $(\Sigma_n\wr H)^{n\times \U}$-projective (stable acyclic) cofibration, the map $(X_0)^{\wedge n}\to (\colim X_i)^{\wedge n}$ is a $(\Sigma_n\wr H)^{n\times \U}$-projective (stable acyclic) cofibration.
\end{enumerate}
This proves the proposition, because every $H^{\U}$-projective (stable acyclic) cofibration is a retract of one obtained from the operations $(i)-(iii)$.
 
Number $(iii)$ is clear, because $(-)^{\wedge n}$ commutes with sequential colimits. Regarding $(i)$, we use the distributive law recalled above to see that $(\bigvee_{I} j_i)^{\wedge n}$ decomposes as a wedge of maps of the form $(\Sigma_n\wr H)\ltimes_{\Sigma_{n_1}\wr H\times \hdots \times \Sigma_{n_k}\wr H} (j_{a_1}^{\wedge n_1}\wedge \hdots \wedge j_{a_k}^{\wedge n_k})$. We know that each factor is a (stable acyclic) $(n_i\times \U)$-projective cofibration of $(n_i\times \U)$-projective $(\Sigma_{n_i}\wr H)$-symmetric spectra, so by Proposition \ref{prop:times} and Lemma \ref{lem:wirth2} the induction of the smash product over them is a (stable acyclic) $(\Sigma_n\wr H)^{n\times \U}$-projective cofibration. Hence, so is the wedge.

Finally, we prove $(ii)$ by induction on $n$, using that $(i)$ and $(iii)$ are already dealt with. The case $n=1$ is clear. So we fix an $n$ larger than $1$ and assume the statement being shown for all $k<n$. The $(\Sigma_n\wr H)$-symmetric spectrum $Y^{\wedge n}$ is obtained from $X^{\wedge n}$ by inductively forming pushouts along the maps $(\Sigma_n\wr H)\ltimes _{\Sigma_k\wr H\times \Sigma_{n-k}\wr H} ((i^{\square k})\wedge X^{\wedge (n-k)})$ for $k=1,\hdots, n$ (as explained, for example, in \cite[Lem. A.8.]{SS12}). Since the $i^{\square k}$ are $(\Sigma_k \wr H)^{k\times \U}$-projective cofibrations if $i$ is a wedge of generating cofibrations, as remarked above, and $X^{\wedge(n-k)}$ is $(\Sigma_{n-k}\wr H)^{(n-k)\times \U}$-projective, the part about cofibrations again follows from Proposition \ref{prop:times}.
If $i$ is a wedge of generating $H^{\U}$-projective stable acyclic cofibrations, it follows from Remark \ref{rem:trick} that the $n$-fold pushout product is a $(\Sigma_k \wr H)^{k\times \U}$-flat cofibration. Therefore, it suffices to show that the quotients $(\Sigma_n\wr H)\ltimes _{\Sigma_k\wr H\times \Sigma_{n-k}\wr H} ((B/A)^{\wedge k}\wedge X^{\wedge (n-k)})$ are all $(\Sigma_n\wr H)^{n\times \U}$-stably contractible. Since we already know that each $X^{\wedge (n-k)}$ is $(\Sigma_{n-k}\wr H)^{(n-k)\times \U}$-projective, in view of Proposition \ref{prop:times} it hence suffices to show that each $(B/A)^{\wedge k}$ is $(\Sigma_k\wr H)^{k\times \U}$-stably contractible. For all $k<n$ this follows from the induction hypothesis. For $k=n$ we consider a different pushout where $X$ is equal to $A$, $Y$ is equal to $B$ and the horizontal maps are identities. Again, we see that the map $i^{\wedge n}:A^{\wedge n}\to B^{\wedge n}$ is the composite of $n$ maps, of which the first $n-1$ are $(\Sigma_n\wr H)^{n\times \U}$-stable equivalences by the induction hypothesis. Since we have shown above that $i^{
\wedge n}$ is also a $(\Sigma_n\wr H)^{n\times \U}$-stable equivalence, this means that the last map in the composite must also be one. This last map 
is $i^{\square n}$ with quotient $(B/A)^{\wedge n}$, which finishes the proof.
\end{proof}

Now we come to the flat case. Again, by Ken Brown's Lemma the following proposition implies Theorem \ref{theo:norm2}. 
\begin{Prop} Let $N\leq H\leq G$ be subgroup inclusions with $N$ normal in $G$ and $f:X\to Y$ an $H$-flat cofibration of $H$-flat $H$-symmetric spectra. Then $N_H^G f:N_H^G X\to N_H^G Y$ is a $G$-flat cofibration. If $f$ is in addition an $H^{\U_H(N)}$-stable equivalence, then $N_H^G f$ is a $G^{\U_G(N)}$-stable equivalence.
\end{Prop}
\begin{proof} The method of proof is similar to that of Proposition \ref{prop:norm1} above. We explain the modifications needed and at which places we are making use of this specific type of $H$-set universe. We start with the generating (stable acyclic) flat cofibrations. The generating flat cofibrations are of the form $\mathscr{G}_m(i:A\to B)$ for $i$ a genuine cofibration of $(G\times \Sigma_n)$-CW complexes. Then $N_H^G(\mathscr{G}_m(i))\cong \mathscr{G}_{G/H\times m}((\Sigma_{G/H\times m})_+\wedge_{\Sigma_m^{G/H}} N_H^Gi)$, which is a flat cofibration of $G$-symmetric spectra (as - by the same proof - is the map $(\mathscr{G}_m(i))^{\square N_H^G}$). Furthermore, as a consequence of Lemma \ref{lem:tech2}, this map is also a $G^{\U_G(N)}$-level equivalence if $i:A\to B$ is in addition an $\F_{\U_H(N)}^{H,\Sigma_m}$-equivalence, which proves the statement on the generating acyclic cofibrations of the level model structure. The additional generating acyclic cofibrations of the $H^{\U}$-flat stable model 
structure are $H^{\U_H(N)}$-projective cofibrations of $H^{\U_H(N)}$-projective $H$-symmetric spectra, so Proposition \ref{prop:norm1} above implies that they are sent to $G^{G\ltimes_H \U_H(N)}$-stable equivalences of $G^{G\ltimes_H \U_H(N)}$-projective $G$-symmetric spectra under the norm. Since the $G$-set universe $G\ltimes_H \U_H(N)$ is a subuniverse of $\U_G(N)$, the change of universe for the projective model structures (cf. Section \ref{sec:changeofuniverse}) implies that every $G^{G\ltimes_H \U_H(N)}$-stable equivalence of $G^{G\ltimes_H \U_H(N)}$-projective $G$-symmetric spectra is also a $G^{\U_G(N)}$-stable equivalence.

We are then left to show the analogs of items $(i)$, $(ii)$ and $(iii)$ in the proof of Proposition \ref{prop:norm1}. Again, $(iii)$ is easy. This time we show both items $(i)$ and $(ii)$ by induction over the order of $G$, the whole statement being trivial for $G=\{e\}$. So we assume the proposition already shown for all proper subgroups of $G$.

We start by proving $(i)$: Given a wedge $\bigvee_I j_i:\bigvee_I X_i\to \bigvee_I Y_i$ of generating $H^{\U}$-flat (stable acyclic) cofibrations, the distributive law shows that $N_H^G (\bigvee_I j_i)$ is a wedge of maps of the form $G\ltimes_K(\bigwedge_{i=1}^k N^K_{K\cap g_iHg_i^{-1}}c_{g_i}^{*} {\res}^H_{g_i^{-1}Kg_i\cap H} j_{f(g_i)})$, for a function $f:G/H\to I$. This immediately shows that $N_H^G (\bigvee_I j_i)$ is again a $G$-flat cofibration. If each $j_i$ is a generating acyclic cofibration, we claim that every smash factor $N^K_{K\cap g_iHg_i^{-1}}c_{g_i}^{*} {\res}^H_{g^{-1}Kg\cap H} j_{f(g_i)}$ is a $K^{\U_K(K\cap N)}$-stable equivalence. If $K$ is equal to $G$, one can choose $g_i=1$ and the term becomes simply $N_H^G (j_{f(1)})$, which we have shown to be a $G^{\U_G(N)}$-stable equivalence. If $K$ is a proper subgroup, $c_{g_i}^{*} {\res}^H_{g_i^{-1}Kg_i\cap H} j_{f(g_i)}$ is a $K\cap gHg^{-1}$-flat cofibration and in addition a $(K\cap g_iHg_i^{-1})^{c_{g_i}^* {\res}^H_{g_i^{-1}Kg_i\cap H}\
U_H(N)}$-stable equivalence. But the $K\cap g_iHg_i^{-1}$-set universe $c_{g_i}^* {\res}^H_{g_i^{-1}Kg_i\cap H}\U_H(N)$ is isomorphic to $\U_{K\cap g_iHg_i^{-1}}(K\cap N)$, and $K\cap N$ is a subgroup of $K\cap g_iHg_i^{-1}$ which is normal in $K$. So by the induction hypothesis, applying the norm $N_{K\cap gHg^{-1}}^K$ gives a $K^{\U_K(K\cap N)}$-stable equivalence, as claimed. Hence, by the monoidality of the $K^{\U_K(K\cap N)}$-flat stable model structure, the smash product over all these is also a $K^{\U_K(K\cap N)}$-stable equivalence. Since $\U_K(K\cap N)$ is isomorphic to the restriction of $\U_G(N)$, we see that the induction to $G$ is a $G^{\U_G(N)}$-stable equivalence, so $(i)$ is shown.

For item $(ii)$ we take a pushout square as in the proof of Proposition \ref{prop:norm1} and again use that $Y^{\wedge n}$ can be obtained from $X^{\wedge n}$ by iteratively forming pushouts along the map $\Sigma_n\wr H\ltimes_{\Sigma_k\wr H\times \Sigma_{n-k}\wr H}(i^{\square k}\wedge X^{\wedge (n-k)})$ for $k=1,\hdots,n$. All these maps are $G$-flat cofibrations by Remark \ref{rem:trick}. For the case where $i$ is a wedge of generating $H^{\U_H(N)}$-stably acyclic cofibrations, we consider the quotient $(\Sigma_n\wr H)\ltimes_{\Sigma_k\wr H\times \Sigma_{n-k}\wr H} (B/A)^{\wedge k}\wedge X^{n-k}$. It remains to show that when restricted along the chosen embedding $G\to \Sigma_n\wr H$, these quotients are all $G^{\U_G(N)}$-stably contractible.
This restriction is isomorphic to the wedge over the terms in the distributive law for $N_H^G (B/A\vee X)$ associated to those functions $f:G/H\to I$ which map exactly $k$ elements to $1$. For all but the constant function with value $1$, the induction hypothesis and the same proof as for item $(i)$ shows that the associated term is $G^{\U_G(N)}$-stably contractible. Hence, it remains to show that $N_H^G (B/A)$ is $G^{\U_G(N)}$-stably contractible. For this we again consider a different pushout with $X=A$, $Y=B$ and both horizontal maps the identities. We already know that the composite $N_H^G A\to N_H^G B$ is a $G^{\U_G(N)}$-stable equivalence. Since we also showed that all maps but the last in the factorization are $G^{\U_G(N)}$-stable equivalences, this means that the last one must also be one. This last map has quotient $N_H^G (B/A)$, which is hence $G^{\U_G(N)}$-stably contractible, so we are done.
\end{proof}

\subsection{Model structures on commutative algebras}
In this section we construct model structures on the category of commutative algebras over a fixed commutative $G$-symmetric ring spectrum $R$. The following theorem summarizes the results. We recall once more that for a normal subgroup $N$ of $G$ we denote by $\U_G(N)$ the full $N$-fixed $G$-set universe.

\begin{Theorem}[Model structures on commutative algebras] \label{theo:modcomm} Let $R$ be a commutative $G$-symmetric ring spectrum and $N$ be a normal subgroup of $G$. Then the positive $G^{\U_G(N)}$-projective and the positive $G^{\U_G(N)}$-flat model structures lift to the category of commutative $R$-algebras. In either of these categories, a map is a weak equivalence or fibration if and only if the underlying map is in the respective positive model structure on $G$-symmetric spectra.
\end{Theorem}
The cofibrations in the flat model structures do not depend on the universe and have the following convenient property (cf. \cite{Shi04} for the non-equivariant version):
\begin{Prop}[Convenience] \label{prop:conv} The underlying $R$-module map of a positive $G$-flat cofibration of commutative $R$-algebras $X\to Y$ is a positive $G$-flat cofibration of $R$-modules if $X$ is (not necessarily positive) $G$-flat as an $R$-module. In particular, the $G$-symmetric spectrum underlying a positive $G$-flat commutative $G$-symmetric ring spectrum is $G$-flat.
\end{Prop}
In the last section we discussed that given a subgroup inclusion $H\leq G$, the forgetful functor from the category of commutative $G$-symmetric ring spectra to the category of commutative $H$-symmetric ring spectra has a left adjoint, the norm $N_H^G$. Since fibrations and weak equivalences are created in underlying non-equivariant spectra, it follows directly from the change of group results of Section \ref{sec:changeofgroups} that this becomes a Quillen adjunction for both the positive projective and the positive flat model structures. Moreover, the convenience property and Theorem \ref{theo:norm2} have the following consequence:
\begin{Cor} \label{cor:normcomp} For every normal subgroup $N$ of $H$ the diagram
\[ \xymatrix{ \Ho^{\U_H(N)}(\text{comm. H-rings}) \ar[r]^-{\mathbb{L}N_H^G}  \ar[d] &\Ho^{\U_G(N)}(\text{comm. G-rings}) \ar[d] \\
 \Ho^{\U_H(N)}(HSp^{\Sigma}) \ar[r]^-{\mathbb{L}N_H^G} & \Ho^{\U_G(N)}(GSp^{\Sigma}) }                                            
                                            \]
commutes up to natural isomorphism, where the vertical maps are forgetful functors. 
\end{Cor}
\begin{proof} We can compute the derived norm on the level of commutative ring spectra by replacing by a cofibrant commutative $H$-symmetric ring spectrum in the positive $H^{\U_H(N)}$-flat model structure. By Proposition \ref{prop:conv}, the $H$-symmetric spectrum underlying a cofibrant commutative $H$-symmetric ring spectrum is (positive) $H$-flat and we know that the derived norm $N_H^G$ on $H$-symmetric spectra can be computed by an $H$-flat replacement (Theorem \ref{theo:norm2}), which finishes the proof.
\end{proof}

In \cite{HHR16}, where they work with a projective model structure, this issue needs to be addressed differently (cf. \cite[Sec. B.8.]{HHR16}).

We now come to the proofs. In the non-equivariant version, a major input is the fact that given a positive flat symmetric spectrum $X$, the $n$-fold smash product $X^{\wedge n}$ is $\Sigma_n$-free, i.e., the map $(E\Sigma_n)_+\wedge _{\Sigma_n} X^{\wedge n}\to X^{\wedge n}/\Sigma_n$ is a stable equivalence. Equivariantly, the level of freeness depends on the chosen $G$-set universe. We recall that $E\F_{\U_G(N)}^{G,\Sigma_n}$ denotes the universal space for the family of graph subgroups for homomorphisms $H\to \Sigma_n$ ($H\leq G$) with kernel containing $H\cap N$ and that $\mathscr{G}_m(-)$ stands for the semi-free $G$-symmetric spectrum functor in level $m$ (cf. Section \ref{sec:free}).
\begin{Lemma} \label{lem:commfree} For every two natural numbers $m,n$ with $m>0$ and every cofibrant $(G\times \Sigma_m)$-space $A$ the map
\[     (E\F_{\U_G(N)}^{G,\Sigma_n})_+\wedge_{\Sigma_n} (\mathscr{G}_m (A))^{\wedge n}\to (\mathscr{G}_m(A))^{\wedge n}/\Sigma_n \]
is a $G^{\U_G(N)}$-level equivalence.
\end{Lemma}
\begin{Remark} As remarked in \cite[Rem. B.119]{HHR16}, there is an error in the analogous version of this result for $G$-orthogonal spectra in \cite[Lem. III.8.4]{MM02}, where they always use an $E\Sigma_n$ with trivial $G$-action, even when working over non-trivial $G$-representation universes. In the case of $G$-orthogonal spectra and the complete universe it is corrected in \cite[Prop. B.116]{HHR16}.
\end{Remark}
\begin{proof} There is a natural $(G\times \Sigma_n)$-isomorphism $(\mathscr{G}_m(A))^{\wedge n}\cong \mathscr{G}_{n\times m}((\Sigma_{n\times m})_+\wedge_{\Sigma_m^n} A^{\wedge n})$ with diagonal $G$-action on $A$ and $\Sigma_n$-action both on $A^{\wedge n}$ and by precomposition on $\Sigma_{n\times m}$ (and no $\Sigma_n$-action on the level $n\times m$, with respect to which the semi-free spectrum is formed). So the map above is naturally isomorphic to $\mathscr{G}_{n\times m}$ applied to the map
\[ (E\F_{\U_G(N)}^{G,\Sigma_n})_+\wedge_{\Sigma_n} (\Sigma_{n\times m})_+\wedge _{\Sigma_m^n} A^{\wedge n}\to (\Sigma_{n\times m})_+\wedge_{\Sigma_n\wr \Sigma_m} A^{\wedge n} \]
of cofibrant $(G\times \Sigma_{n\times m})$-spaces. The isotropy of $\U_G(N)$ is closed under supergroups and so Lemma \ref{lem:tech3} says that this map is an $\F_{\U_G(N)}^{G,\Sigma_{n\times m}}$-equivalence, which finishes the proof.
\end{proof}

\begin{Prop} \label{prop:commfree} For every positive $G$-flat $G$-symmetric spectrum $X$ the map\[ (E\F_{\U_G(N)}^{G,\Sigma_n})_+\wedge_{\Sigma_n} X^{\wedge n}\to X^{\wedge n}/\Sigma_n\]
is a $\upi_*^{\U_G(N)}$-isomorphism and in particular a $G^{\U_G(N)}$-stable equivalence.
\end{Prop}
\begin{proof} This follows from Lemma \ref{lem:commfree} via an induction over the skeleton filtration of $X$, using again that if $f:A\to B$ is the pushout of a map $i$, then $f^{\wedge n}$ is the composite of pushouts of maps of the form $\Sigma_n\ltimes_{\Sigma_{k}\times \Sigma_{n-k}}(i^{\square k}\wedge A^{\wedge n-k})$ (cf. \cite[Lem. A.8]{SS12}).
\end{proof}

For the homotopical properties of $(-)^{\wedge n}/\Sigma_n$ we hence need to examine those of the functor $(E\F_{\U}^{G,\Sigma_n})_+\wedge_{\Sigma_n} X^{\wedge n}$.
\begin{Lemma} \label{lem:comm2} Let $\U$ be a $G$-set universe, $Y$ a $(G\times \Sigma_n)$-symmetric spectrum which is $K^{\U}$-contractible for all graph subgroups $K$ in $\F_{\U}^{G,\Sigma_n}$ (where $K$ acts on $\U$ through its projection to $G$), and $A$ a cofibrant $(G\times \Sigma_n)$-space set with isotropy in $\F_{\U}^{G,\Sigma_n}$. Then $A\wedge _{\Sigma_n} Y$ is $G^{\U}$-stably contractible.
\end{Lemma}
\begin{proof} We let $A$ be an $I_G$-cell complex, the general statement follows since $G^{\U}$-stable equivalences are closed under retracts. By an induction over the cells it suffices to show that $(G\times \Sigma_n)/K_+\wedge _{\Sigma_n} Y$ is $G^{\U}$-stably contractible for all $K$ in $\F_{\U}^{G,\Sigma_n}$. Let $H$ be the projection of $K$ to $G$. Then $(G\times \Sigma_n)/K_+\wedge _{\Sigma_n} Y$ is $G$-isomorphic to $G\ltimes_H \res_H^K Y$, which is $G^{\U}$-stably contractible by the assumption and the fact that induction maps $H^{\U}$-stable equivalences to $G^{\U}$-stable equivalences.
\end{proof}
To establish the model structures we use results of \cite{Whi17}. The following is the main input for applying them:
\begin{Prop} \label{prop:comm2} Let $i:A\to B$ be a morphism of $G$-symmetric spectra. Then:
\begin{enumerate}[(i)]
  \item If $i$ is a (positive) $G$-flat cofibration, then $i^{\square n}/\Sigma_n$ is again a (positive) $G$-flat cofibration.
  \item If $i$ is a positive $G$-flat cofibration and $G^{\U_G(N)}$-stable equivalence, then $i^{\square n}/\Sigma_n$ is a $G^{\U_G(N)}$-stable equivalence.
\end{enumerate}
\end{Prop}
\begin{proof} By \cite[Lem. 5.1]{Whi17}, it suffices to show $(i)$ on the class of generating $G$-flat cofibrations, which are all of the form $\mathscr{G}_m(j):\mathscr{G}_m(A)\to \mathscr{G}_m(B)$ for a genuine $(G\times \Sigma_m)$-cofibration $j$ and $m>0$. But then $(\mathscr{G}_m(j))^{\square n}$ is equal to $\mathscr{G}_{n\times m}(\Sigma_{n\times m}\ltimes_{\Sigma_m^n} j^{\square n})$, which is even a $\Sigma_n\wr G$--flat cofibration because $\Sigma_{n\times m}\ltimes_{\Sigma_m^n} j^{\square n}$ is a genuine $\Sigma_n\wr G$-cofibration. This proves that $(-)^{\square n}/\Sigma_n$ preserves (positive) $G$-flat cofibrations.

For $(ii)$ we use that we already know that $i^{\square n}/\Sigma_n$ is a $G$-flat cofibration, and hence it suffices to show that the quotient $(B/A)^{\wedge n}/\Sigma_n$ is $G^{\U}$-stably contractible. Since $B/A$ is positive $G$-flat, we know by Proposition \ref{prop:commfree} that the quotient is $G^{\U_G(N)}$-stably equivalent to $(E\F_{\U_G(N)}^{G,\Sigma_n})_+\wedge_{\Sigma_n} (B/A)^{\wedge n}$. We now want to apply Lemma \ref{lem:comm2} and have to check that for every subgroup $H$ of $G$ and every homomorphism $\alpha:H\to \Sigma_n$ with kernel containing $H\cap N$ the spectrum $(B/A)^{\wedge n}$ is $H^{\U_H(H\cap N)}$-stably contractible, where the $H$-action is via the graph embedding $H\to G\times \Sigma_n$. But $(B/A)^{\wedge n}$ is $H$-isomorphic to a smash product of $N_{K_i}^H (B/A)$ for subgroups $K_i$ of $H$ which contain $H\cap N$. The $G$-symmetric spectrum $B/A$ is $G^{\U_G(N)}$-stably contractible, hence it is also $K_i^{\U_{K_i}(K_i\cap N)}$-stably contractible.
Since $K_i$ contains $H\cap N$, we see that $K_i\cap N$ equals $H\cap N$, which is normal in $H$. So we can use Theorem \ref{theo:norm2} to deduce that each $N_{K_i}^H (B/A)$ is $H^{\U_H(H\cap N)}$-stably contractible. All of them are also $H$-flat (by Proposition \ref{prop:norm1}), hence their smash product is $H^{\U_H(H\cap N)}$-stably contractible and we are done.
\end{proof}
In the language of \cite{Whi17}, this implies that the positive $G^{\U_G(N)}$-flat stable model structure satisfies the strong commutative monoid axiom (\cite[Def. 3.4.]{Whi17}), while the positive $G^{\U_G(N)}$-projective stable model structure satisfies the weak commutative monoid axiom (\cite[Rem. 3.3]{Whi17}). 

\begin{proof}[Proof of Theorem \ref{theo:modcomm}] This follows from Proposition \ref{prop:comm2} via \cite[Thm 3.2.]{Whi17} in the flat case and \cite[Rem. 3.3.]{Whi17} in the projective case.
\end{proof}

\begin{proof}[Proof of Proposition \ref{prop:conv}] If one assumes that the source is a positive $G$-flat $R$-module, this is \cite[Prop. 3.5.]{Whi17}. However, the only place where positive $G$-flatness of the source is used in the proof is to ensure that smashing with it preserves positive $G$-flat cofibrations of $R$-modules. For this it is sufficient that it is $G$-flat as an $R$-module. 
\end{proof}

\section{Comparison to other models}
\subsection{Quillen equivalence to $G$-orthogonal spectra}
\label{sec:gquillen}
%As we will be comparing model structures on $G$-symmetric spectra to ones on $G$-orthogonal spectra in Section \ref{sec:gquillen}, we quickly explain the analogous constructions there.
%We quickly recall $G$-orthogonal spectra, which were introduced in \cite{MM02}. Other references are \cite{Sto11} and \cite{Sch11}.
In this section we explain how our model structures relate to the respective ones on $G$-orthogonal spectra, introduced in \cite{MM02}, \cite{Sto11} and \cite{HHR16}.

Orthogonal spectra are defined similarly to symmetric spectra of spaces with orthogonal groups $O(n)$ in place of the symmetric groups $\Sigma_n$, i.e., an orthogonal spectrum is a sequence of based $O(n)$-spaces $X_n$ with structure maps $X_n\wedge S^1\to X_{n+1}$ whose iterations $X_n\wedge S^m\to X_{n+m}$ are $O(n)\times O(m)$-equivariant.

An orthogonal spectrum $X$ has an underlying symmetric spectrum of spaces $UX$ by restricting the $O(n)$-actions to $\Sigma_n$-actions via the embedding as permutation matrices. The resulting functor $U:Sp^{O}\to Sp^{\Sigma}_{\T}$ has a left adjoint $L$, which can be obtained via a left Kan extension (see \cite[I.3 and III.23]{MMSS01} for a description).
%The left adjoint $L$ is strong symmetric monoidal with respect to the smash products, whereas $U$ is lax symmetric monoidal.

\begin{Def} A $G$-orthogonal spectrum is a $G$-object in orthogonal spectra.
\end{Def}
\begin{Remark} This is not exactly the definition of \cite{MM02}, but there is an equivalence of categories (cf. \cite[V, Thm. 1.5]{MM02}, \cite[Prop. A.18]{HHR16} and Section \ref{sec:free} for the corresponding story for $G$-symmetric spectra). 
\end{Remark}
Every $G$-orthogonal spectrum $X$ can be evaluated on all finite dimensional $G$-representations~$V$ via the formula $X(V)=X_{\dim V}\wedge_{O(\dim V)}\Lin(\R^{\dim V},V)_+$, where $\Lin$ denotes the space of linear isometries, and there are $G$-equivariant generalized structure maps $\sigma_V^W:X(V)\wedge S^W\to X(V\oplus W)$. 

These evaluations for $G$-symmetric and $G$-orthogonal spectra are related in the following way:
\begin{Lemma} \label{lem:eva} For every finite $G$-set $M$ and every $G$-orthogonal spectrum $X$ there is a natural $G$-homeomorphism $UX(M)\cong X(\R[M])$.
\end{Lemma}
\begin{proof} Let $m$ be the order of $M$. Then the map sending a pair $[x\in X_m,\varphi:\underline{m}\cong M]$ to $[x,\R[\varphi]:\R^m\cong \R[M]]$ is an equivariant homeomorphism.
\end{proof}

We start with comparing the projective model structures, and there we are cheating a bit, because we really need to compare to a stronger variant of the one of \cite{MM02} and require a level equivalence to induce a genuine $H$-equivalence on all $H$-subrepresentations of $\R[\U]$ not just on restrictions of $G$-subrepresentations, and likewise use a stronger notion of $G^{\R[\U]}\Omega$-spectra. This does not affect the stable equivalences, which are the $\upi_*^{\R[\U]}$-isomorphisms. In the case where $\R[\U]$ is a complete $G$-representation universe, the positive version of this stronger model structure is used in \cite[B.4.]{HHR16}.
\begin{Theorem}[Quillen equivalence to $G$-orthogonal spectra, I] \label{theo:gquillen} The adjunction \[ L:GSp^{\Sigma}_{\T}\rightleftarrows GSp^O:U \] is a Quillen equivalence for the $G^{\U}$-projective stable model structure on $GSp^{\Sigma}_{\T}$ and the (strong) $G^{\R[\U]}$-projective stable model structure on $GSp^{O}$.
\end{Theorem}
For the flat model structures, we let $\mathcal{V}$ be a $G$-representation universe with the property that every finite dimensional $G$-subrepresentation sits inside a larger one which is isomorphic to the linearization of a $G$-set. This gives rise to a $G$-set universe $\U(\mathcal{V})$ by taking an infinite disjoint union of those $G$-orbits $G/H$ for which $\R[G/H]$ embeds into $\mathcal{V}$. In particular, we have $\R[\U(\mathcal{V})]\cong \mathcal{V}$. 
\begin{Theorem}[Quillen equivalence to $G$-orthogonal spectra, II] \label{theo:gquillen2} The adjunction \[ L:GSp^{\Sigma}_{\T}\rightleftarrows GSp^O:U \] is a Quillen equivalence for the $G^{\U(\mathcal{V})}$-flat stable model structure on $GSp^{\Sigma}_{\T}$ and the $G^{\mathcal{V}}$-flat stable model structure of \cite[Sec. 2.3.3]{Sto11} on $GSp^{O}$.
\end{Theorem}
We note that in general a $G$-set universe $\U$ is not isomorphic to $\U(\R[\U])$ and if it is not, the $G^{\U}$-flat model structure does not allow a direct comparison to one of the model structures on $G$-orthogonal spectra. Nevertheless, it is still Quillen equivalent to one via a zig-zag through the projective model structures.
For example, taking $\U$ to be a free $G$-set universe, its linearization is a complete $G$-representation universe and so $\U(\R[\U])$ is a complete $G$-set universe.
\begin{proof}[Proof of Theorem \ref{theo:gquillen}] In the (strong version of the) $G^{\R[\U]}$-projective level model structure on $GSp^O$, a map is a weak equivalence (fibration) if and only if it is a genuine $H$-equivalence ($H$-fibration) when evaluated on all finite dimensional $H$-subrepresentations $V$ of $\R[\U]$ for all subgroups $H$ of $G$. Hence, it follows directly from the natural $H$-homeomorphism $X(\mathbb{R}M)\cong (UX)(M)$ for every subgroup $H$ of $G$ and every finite $H$-set $M$ (Lemma \ref{lem:eva}) that the forgetful functor $U$ is a right Quillen functor for the level model structures. In particular, the left adjoint $L$ preserves cofibrations. The stable model structure on $G$-orthogonal spectra is obtained by a left Bousfield localization at the $\upi_*^{\R[\U]}$-isomorphisms, with fibrant objects being exactly the (strong) $G^{\R[\U]}\Omega$-spectra, so the isomorphism of Lemma \ref{lem:eva} also implies that $U$ preserves fibrant objects. By adjunction, this implies that $L$ maps 
$G^{\U}$-projective $G$-symmetric spectra which are $G^{\U}$-stably contractible to $G^{\R[\U]}$-projective $G^{\R[\U]}$-stably contractible $G$-orthogonal spectra. Since both for $G$-symmetric spectra and for $G$-orthogonal spectra a projective cofibration is stably acyclic if and only if its cofiber is stably contractible, this implies that $L$ also preserves stable acyclic cofibrations and hence is a left Quillen functor.

It remains to show that the derived functors are equivalences between the homotopy categories. Since a $\upi_*^{\U}$-isomorphism is a $G^{\U}$-stable equivalence (Theorem \ref{theo:piiso}), we see that $U$ preserves all stable equivalences. In addition, as the underlying $G$-symmetric spectra of $G$-orthogonal spectra are all $G^{\U}$-semistable (they allow $\upi_*^{\U}$-isomorphisms to $G^{\U}\Omega$-spectra), Corollary \ref{cor:piiso} implies that $U$ also reflects all stable equivalences. This means that we can use criterion \cite[Lem. A.2.(iii)]{MMSS01} and so it suffices to show that the derived unit of the adjunction is a natural isomorphism in the stable $\U$-homotopy category of $G$-symmetric spectra.
%As both model structures are stable, the derived functors $\mathbb{L}L$ and $\mathbb{R}U$ can be equipped with an exact structure (in the sense of triangulated categories) such that unit and counit become exact natural transformations.
Since both model structures are stable and $\mathbb{L}L$ and $\mathbb{R}U$ preserve arbitrary direct sums, it suffices to check that the derived unit is an isomorphism on the set of compact generators $\{G/H_+\wedge \mathbb{S}\ |\ H\leq G\}$ of the triangulated homotopy category.
% (since the subcategory of objects for which it is an isomorphism is easily shown to be closed under cones, suspensions, loops and arbitrary direct sums).
Because each $G/H_+\wedge \mathbb{S}$ is $G$-flat and $U$ preserves all $G^{\U}$-stable equivalences, this in turn is equivalent to checking that for each subgroup $H$ of $G$ the non-derived unit $\eta_{G/H_+\wedge \mathbb{S}}:G/H_+\wedge \mathbb{S}\to UL(G/H_+\wedge \mathbb{S})$ is a $G^{\U}$-stable equivalence. But this morphism is even an isomorphism, since both $L$ and $U$ take the sphere spectrum to itself and commute with the smash product with based $G$-spaces.
\begin{comment}
We now see that this map is even an isomorphism of $G$-symmetric spectra. There is a composition of adjunctions:
\[ \xymatrix{ G\T_*\ar@/^/[rr]^{\Sigma^{\infty}} \ar@/^2pc/[rrrr]^{\Sigma^{\infty}} && GSp^{\Sigma}_\T \ar@/^/[rr]^L \ar@/^/[ll]^{(-)_0} && GSp^O \ar@/^/[ll]^U \ar@/^2pc/[llll]^{(-)_0}} \]
In particular, $L$ maps a symmetric suspension spectrum to the according orthogonal suspension spectrum. So does $U$ the other way around, by definition. Furthermore, both suspension spectrum functors $\Sigma^{\infty}$ are fully faithful and hence so is $L$ when restricted to suspension spectra. It follows that the unit $\eta_{\Sigma^{\infty}A}:\Sigma^{\infty}A\to UL(\Sigma^{\infty}A)$ is an isomorphism for all based $G$-spaces $A$ and in particular for $A=G/H_+$.
\end{comment}
\end{proof}
\begin{proof}[Proof of Theorem \ref{theo:gquillen2}]
The flat $\mathcal{V}$-level model structure on $G$-orthogonal spectra of Stolz is built in the same way as the one on $G$-symmetric spectra, replacing the families $\mathcal{F}_{\U(\mathcal{V})}^{G,\Sigma_n}$ by their orthogonal analogs $\mathcal{F}_{\mathcal{V}}^{G,O(n)}$. The graph of a homomorphism $H\to \Sigma_n$ lies in $\mathcal{F}_{\U(\mathcal{V})}^{G,\Sigma_n}$ if and only if the graph of the composition $H\to \Sigma_n\to O(n)$ lies in $\mathcal{F}_{\mathcal{V}}^{G,O(n)}$, so we see that pulling back $E\mathcal{F}_{\mathcal{V}}^{G,O(n)}$ along $G\times \Sigma_n\to G\times O(n)$ gives a model for $E\mathcal{F}_{\U(\mathcal{V})}^{G,\Sigma_n}$ and it follows that $U$ is a right Quillen functor for the level model structures. From here on we can proceed as in the proof of Theorem \ref{theo:gquillen} above.
%Once we have shown that the adjunction is a Quillen pair, we can proceed as above or use 2-out-of-3 for Quillen equivalences, since we already know that our flat and projective model structures are Quillen equivalent and the same holds for $G$-orthogonal spectra (\cite[Proposition 2.3.31]{Sto11}).
%\todo{shorten ... only say what goes into choice of universe}
%The flat $\mathcal{V}$-level model structure on $G$-orthogonal spectra of Stolz is built in the same way as the one on $G$-symmetric spectra (since it was our role model), replacing the families $\mathcal{F}_{\U(\mathcal{V})}^{G,\Sigma_n}$ by their orthogonal analogs $\mathcal{F}_{\mathcal{V}}^{G,O(n)}$. The graph of a homomorphism $H\to \Sigma_n$ lies in $\mathcal{F}_{\U(\mathcal{V})}^{G,\Sigma_n}$ if and only if the graph of the composition $H\to \Sigma_n\to O(n)$ lies in $\mathcal{F}_{\mathcal{V}}^{G,O(n)}$, so we see that pulling back $E\mathcal{F}_{\mathcal{V}}^{G,O(n)}$ along $G\times \Sigma_n\to G\times O(n)$ gives a model for $E\mathcal{F}_{\U(\mathcal{V})}^{G,\Sigma_n}$ and it follows that $U$ is a right Quillen functor for the level model structures. The fibrant objects of the stable model structure on $G$-orthogonal spectra are precisely the level fibrant $G^{\mathcal{V}}\Omega$-spectra, so we again see that $U$ preserves fibrant objects of the stable model structures and can argue in the same way 
%as above that the adjunction is a Quillen pair. This finishes the proof.
\end{proof}

\subsection{Quillen equivalence to Mandell's equivariant symmetric spectra}
Finally we compare to Mandell's definition of an equivariant symmetric spectrum. As mentioned in the introduction, the underlying category depends on a choice of normal subgroup $N$ and models the equivariant stable homotopy category with respect to the $N$-fixed universe.

\begin{Def}[Mandell, introduction of \cite{Man04}] A $G\Sigma_{G/N}$-spectrum consists of a sequence of based $(G\times \Sigma_n)$-simplicial sets $X(n\times (G/N))$ and structure maps $\sigma_{n\times (G/N)}^{G/N}:X(n\times (G/N))\wedge S^{G/N}\to X((n+1)\times G/N)$ such that for all $n, m\in \N$ the iterate $X(n\times (G/N))\wedge S^{m\times G/N}\to X((n+m)\times (G/N))$ is $(G\times \Sigma_n\times \Sigma_m)$-equivariant.
\end{Def}
The notation we used in the definition makes it clear how it relates to our version of $G$-symmetric spectra: Every $G$-symmetric spectrum of simplicial sets $X$ gives rise to a $G\Sigma_{G/N}$-spectrum (which we denote by $U_{G/N} X$) by only remembering the evaluations $X(n\times (G/N))$ (together with the respective generalized structure maps) and restricting the $(G\ltimes \Sigma_{n\times (G/N)})$-action to the diagonal $(G\times \Sigma_n)$-action. In fact, both categories are enriched functor categories (cf. Section \ref{sec:free}) and the functor $U_{G/N}$ comes from restriction along an embedding of the indexing categories. Hence, there exists a left adjoint $L_{G/N}:G\Sigma_{G/N}\to GSp^{\Sigma}_{\mathcal{S}}$ given by enriched left Kan extension.

For a normal subgroup $N$ of $G$ we denote by $\U_{G/N}$ the $G$-set universe consisting of an infinite disjoint union of copies of $G/N$. It models the $N$-fixed $G$-equivariant stable homotopy category.
\begin{Theorem}[Quillen equivalence to $G\Sigma_{G/N}$-spectra] \label{theo:gquillenmandell} For every finite group $G$ and every normal subgroup $N$ the adjunction \[ L_{G/N}:G\Sigma_{G/N}\rightleftarrows GSp^{\Sigma}_{\mathcal{S}}:{U_{G/N}} \]
is a Quillen equivalence for the stable model structure of \cite[Thm. 4.1]{Man04} on $G\Sigma_{G/N}$ and the $\U_{G/N}$-projective stable model structure on $GSp^{\Sigma}_{\mathcal{S}}$. \end{Theorem}
\begin{proof} We first show that the adjunction is a Quillen pair. In the level model structure of \cite[Def. 3.1]{Man04} the weak equivalences and fibrations are defined levelwise, so it follows that the forgetful functor $U_{G/N}$ becomes a right Quillen functor for the $\U_{G/N}$-projective level model structure on $GSp^{\Sigma}$. The stable model structure (\cite[Thm. 4.1]{Man04}) is obtained via a left Bousfield localization from the level model structure, so it has the same cofibrations and we can deduce that $L_{G/N}$  preserves cofibrations. The fibrant objects of the stable model structure are precisely those level fibrant $X$ for which $X(n\times G/H)\to \Omega^{G/N} X((n+1)\times G/H)$ is a genuine $G$-equivalence, so we see that $U_{G/N}$ also preserves fibrant objects. Via the same argument as in Theorem \ref{theo:gquillen} we deduce that the adjunction is a Quillen pair. In order to see that it is a Quillen equivalence we 
consider the chain of Quillen adjunctions

\[ G\Sigma_{G/N}\stackrel[U_{G/N}]{L_{G/N}}{\rightleftarrows}GSp^{\Sigma}_{\mathcal{S},\U_{G/N},proj}\stackrel[\mathcal{S}]{|.|}{\rightleftarrows} GSp^{\Sigma}_{\T,\U_{G/N},proj} \stackrel[U]{L}{\rightleftarrows} GSp_{\R[\U],proj}^O. \]

By \cite[Thm. 10.2]{Man04}, the full composite is a Quillen equivalence. We already know that the second and third adjunction are Quillen equivalences and hence so is the first.
%also induce equivalences on the homotopy categories. It follows that so does the first, which finishes the proof.
\end{proof}
%\todo{leave out since not on Steimles homepage?}
%\begin{Remark} The comparison between $G$-symmetric spectra and $G\Sigma_{G/\{e\}}$-spectra is considered in more detail in \cite{Ste14}, emphasizing the point of view as equivariant functor categories. There, Steimle constructs another model structure on $G$-symmetric spectra which slightly differs from the projective one developed in this paper  (with respect to the complete $G$-set universe) in that it only considers evaluations at finite $G$-sets (cf. the discussion after Lemma \ref{lem:twist}). This leads to less cofibrations and in \cite{Ste14} it is shown that the functor $U_{G/\{e\}}$ takes these to cofibrations of $G\Sigma_{G/\{e\}}$-spectra in the sense of \cite{Man04}. In addition, a flat-type stable model structure on $G\Sigma_{G/\{e\}}$-spectra is developed and proven to be Quillen equivalent to the one in \cite{Man04}.
%\end{Remark}

\appendix

\section{Technical lemmas}
In this appendix we list several technical lemmas about the interplay of the families $\F_{\U}^{G,\Sigma_n}$ for varying $n$ and $G$, which we left out of the main text to shorten the exposition. The proofs are all similar in style and rely on the following:
\begin{Lemma} \label{lem:tech1} Let $G$ be a discrete group with a normal subgroup $K$, $X$ a cofibrant $G$-space on which the normal subgroup $K$ acts freely away from the basepoint and $H$ a subgroup of the quotient $G/K$. Then there is a natural homeomorphism
\begin{equation} \label{eq:tech}   (X/K)^H\cong \bigvee_{[\alpha:H\to G]} X^{\im(\alpha)}/(C(\im(\alpha))\cap K), \end{equation}
where the wedge is taken over a system of representatives of $K$-conjugacy classes of group homomorphisms $\alpha:H\to G$ lifting the inclusion $H\to G/K$, and $C(\im(\alpha))$ denotes the subgroup of elements commuting with the image of $\alpha$.
\end{Lemma}
\begin{proof}[Sketch of proof] Let $\alpha:H\to G$ be such a group homomorphism and $[h]\in H$ an element represented by $h\in G$. Then $h\cdot \alpha([h])^{-1}$ maps to $e$ under the projection, hence it lies in $K$. So if $x$ is an $\im(\alpha)$-fixed point of $X$, then we have $[h][x]=[hx]=[h\alpha(h)^{-1}x]=[x]$, so $[x]$ is an $H$-fixed point in the quotient. This shows that for each such $\alpha$ the composition $X^{\im(\alpha)}\to X\to X/K$ lands in the $H$-fixed points and we obtain a map from the wedge of the $X^{\im(\alpha)}$ over all such $\alpha$ to $(X/K)^{H}$. Furthermore, two such summands only intersect in the basepoint, because the $K$-action is free. The normal subgroup $K$ acts on this wedge, an element $k$ sends the subspace $X^{\im(\alpha)}$ homeomorphically onto $X^{\im(k\alpha k^{-1})}$ and so we obtain an injective map from the right side of the expression \ref{eq:tech} to the left. On the other hand, given an element $[x]$ fixed by $H$ and represented by $x\in X$, there is an 
associated homomorphism $H\to G$ sending $[h]$ to $k(h)^{-1}\cdot h$, where $k(h)$ is the unique element in $K$ which satisfies $hx=k(h)x$. It is an easy check that this is indeed independent of the chosen representative $h$, that it defines a group homomorphism lifting the inclusion $H\to G/K$ and that $x$ is an $\im(\alpha)$-fixed point. Hence, the map is also surjective and that it is a homeomorphism follows from $X$ being cofibrant.
\end{proof}

\begin{Lemma} \label{lem:comp} Let $\U$ be a $G$-set universe. Then for every $\mathcal{F}_{\U}^{G,\Sigma_m}$-equivalence between cofibrant $(G\times\Sigma_m)$-spaces $f:X\to Y$ and every cofibrant $(G\times \Sigma_k)$-space $A$, the map $\Sigma_{m+k}\ltimes_{\Sigma_m\times \Sigma_k} (f\wedge A)$ is an $\F_{\U}^{G,\Sigma_{m+k}}$-equivalence.
\end{Lemma}
\begin{proof} We let $L$ be a subgroup contained in $\F_{\U}^{G,\Sigma_{m+k}}$. By Lemma \ref{lem:tech1} above (applied to the normal subgroup $\Sigma_m\times \Sigma_k$ of $\Sigma_{m+k}\times \Sigma_m\times \Sigma_k$), the $L$-fixed points of $\Sigma_{m+k}\ltimes_{\Sigma_m\times \Sigma_k} (X\wedge A)$ are naturally isomorphic to a wedge of terms of the form $((\Sigma_{m+k})_+\wedge X\wedge A)^{\Gamma_\alpha}/C(\im(\alpha))$, where $\alpha$ is a homomorphism $L\to \Sigma_m\times \Sigma_k$, and it suffices to show that $((\Sigma_{m+k})_+\wedge f\wedge A)^{\Gamma_\alpha}/C(\im(\alpha))$ is a weak equivalence for every such $\alpha$. We show that $((\Sigma_{m+k})_+\wedge f\wedge A)^{\Gamma_{\alpha}}$ is a weak equivalence. The action of $C(\im(\alpha))$ is free, so this gives the statement. Since $L$ is contained in $\F_{\U}^{G,\Sigma_{m+k}}$, it is the graph of a homomorphism $\beta:H\to \Sigma_{m+k}$ for a subgroup $H$ of $G$ and for which the associated $H$-set $\underline{m+k}_{\beta}$ embeds into $\U$. Composing 
with $\alpha$ gives two further homomorphisms $\alpha_1:H\to \Sigma_m$ and $\alpha_2:H\to \Sigma_k$ with associated $H$-sets $\underline{m}_{\alpha_1}$ and $\underline{k}_{\alpha_2}$. Hence, the $\Gamma_{\alpha}$-fixed points of $\Sigma_{m+k}$ are given by the set of $H$-isomorphisms from $\underline{m}_{\alpha_1} \sqcup \underline{k}_{\alpha_2}$ to $\underline{m+k}_{\beta}$. If $\underline{m}_{\alpha_1}$ does not embed into $\U$, then there cannot exist such an isomorphism, in which case both sides are trivial and hence the map is necessarily a weak equivalence. If $\underline{m}_{\alpha_1}$ does embed into $\U$, then $f$ induces an equivalence between the $\Gamma_{\alpha_1}$-fixed points of $X$ and $Y$, since these are given by the fixed points of the graph of $\alpha_1$ and $f$ is an $\mathcal{F}_{\U}^{G,\Sigma_m}$-equivalence. Hence, the whole term is a weak equivalence and we are done.
\end{proof}

The next lemma deals with the multiplicative norm $N_H^G X$ for a based $H$-space $X$ and a subgroup inclusion $H\leq G$. We will make use of the double coset formula  \[ {\res}_K^G N_H^G X\cong \bigwedge_{[g]\in K\backslash G/H} N^K_{gHg^{-1}\cap K} c_g^* {\res}^H_{H\cap g^{-1}Kg} X. \]
Provided that $X$ also has a $\Sigma_m$-action commuting with the $H$-action, the norm $N_H^G X$ has a $G\ltimes (\Sigma_m)^{G/H}$-action, with $G$ acting on $(\Sigma_m)^{G/H}$ via its action on $G/H$.

The $H$-universe we use is the full $N$-fixed one $\U_H(N)$ for a subgroup $N$ of $H$ which is normal in $G$.
\begin{Lemma} \label{lem:tech2} Let $n,m\in \N$ and $j:A\to B$ an $\F_{\U_H(N)}^{H,\Sigma_m}$-equivalence of cofibrant $(H\times \Sigma_m)$-spaces. Then $(\Sigma_{G/H\times m})_+\wedge_{\Sigma_m^{G/H}} N_H^G j$ is an $\F_{\U_G(N)}^{G,\Sigma_{G/H\times m}}$-equivalence. Here, the $G$-action is both through $N_H^G(-)$ and by precomposition with the inverse on $\Sigma_{G/H\times m}$. \end{Lemma}
\begin{proof}Again we have to show that for every homomorphism $\beta:J\to \Sigma_{G/H\times m}$ from a subgroup $J$ of $G$ the $\Gamma_{\beta}$-fixed points of $(\Sigma_{G/H\times m})_+\wedge_{\Sigma_m^{G/H}} N_H^G j$ are a weak equivalence, provided that the associated $J$-action on $G/H\times \underline{m}$ embeds into $\U_J(J\cap N)$, i.e., has $J\cap N$ in the kernel. Using Lemma \ref{lem:tech1} above for the normal subgroup $\Sigma_m^{G/H}$ of $\Sigma_{G/H\times m}\times (G\ltimes \Sigma_m^{G/H})$, these fixed points naturally decompose into a wedge with summands of the form $((\Sigma_{G/H\times m})_+\wedge N_H^G j)^{\Gamma_{\alpha}}/(C(\im(\alpha))\cap \Sigma_m^{G/H})$, where $\alpha$ is a group homomorphism $J\to G\ltimes (\Sigma_m^{G/H})$ lifting the inclusion $J\to G$. We again consider the map on fixed points before quotiening out by $C(\im(\alpha))\cap \Sigma_m^{G/H}$ and show that it is a weak equivalence. Since $G\ltimes\Sigma_m^{G/H}$ comes with an action on $G/H\times m$, the map $\alpha$ 
also induces another action of $J$, which we denote by $(G/H\times m)_{\alpha}$. Let $g_1,\hdots,g_k$ be a system of double coset representatives for $J\backslash G/H$ and $J_i=J\cap g_iHg_i^{-1}$ the $J$-stabilizer of $g_iH$ in $G/H$, in particular $g_iH\times m$ becomes a $J_i$-set with the restricted action through $\alpha$. This gives us two things: A $J$-decomposition
\[ (G/H\times m)_{\alpha} \cong \bigsqcup_{i=1,\hdots,k} J\ltimes_{J_i} (g_iH\times m)_{\alpha_{|J_i}} \]
and an isomorphism
\[ ((\Sigma_{G/H\cdot m})_+\wedge_{\Sigma_m^{G/H}} N_H^G A)^{\im(\alpha))}\cong {\Iso}_J((G/H\times m)_{\beta},(G/H\times m)_{\alpha})_+\wedge \bigwedge_{i=1,\hdots,k} (N_{J_i}^J {\res}^{H\times \Sigma_m}_{J_i}A)^J \]
where the last restriction is formed along the composition \[ J_i\xr{(incl,\alpha)} g_iHg_i^{-1}\times \Sigma_m\xr{c_g^{*}\times \Sigma_m} H\times \Sigma_m.\]
The $J$-fixed points of a norm $N_{J_i}^J$ are isomorphic to the $J_i$-fixed points of the argument, so the last smash product factor becomes $({\res}_{J_i}^{H\times \Sigma_m}A)^{J_i}$. Now, if $L\cap N$ acts non-trivially on $(G/H\times m)_{\alpha}$, the latter cannot be isomorphic to $(G/H\times m)_{\beta}$ and hence in that case the fixed points are trivial on both sides and the statement is shown. On the other hand, if $J\cap N$ does act trivially on $(G/H\times m)_{\alpha}$, then each $J_i\cap N$ acts trivially on $(g_iH\times m)_{\alpha_{|J_i}}$. Hence the kernel of $\alpha_{|J_i}$ contains $J_i\cap N$ and so $j$ induces a weak equivalence on $({\res}_{J_i}^{H\times \Sigma_m}-)^{J_i}$. Thus, in that case $((\Sigma_{G/H\cdot m})_+\wedge N_H^G j)^{\Gamma_{\alpha}}$ is a weak equivalence as the smash product of weak equivalences between cofibrant spaces, which finishes the proof.
\end{proof}

In the final lemma we need the technical hypothesis that the isotropy of the $G$-set universe $\U$ is closed under supergroups, i.e., that whenever an orbit $G/K$ embeds into $\U$, so does $G/L$ for all $K\leq L\leq G$. We note that if $\U$ has this property, then so do all restrictions of $\U$ to a subgroup of $G$. 
\begin{Lemma} \label{lem:tech3} Let $A$ be a cofibrant $(G\times \Sigma_m)$-space for $m>0$, $n$ a natural number and $\U$ a $G$-set universe whose isotropy is closed under supergroups. Then the map
\[   (E\mathcal{F}_{\U}^{G,\Sigma_n})_+\wedge _{\Sigma_n} (\Sigma_{n\times m})_+\wedge_{\Sigma_m^n} A^{\wedge n} \to (\Sigma_{n\times m})_+\wedge _{\Sigma_n\wr \Sigma_m} A^{\wedge n} \]
is an $\F_{\U}^{G,\Sigma_{n\times m}}$-equivalence. Here, the $\Sigma_n$-action on $(\Sigma_{n\times m})_+\wedge _{\Sigma_m^n} A^{\wedge n}$ is both through $A^n$ and by precomposition on $\Sigma_{n\times m}$, permuting the blocks of size $m$. $G$ only acts through $A^{\wedge n}$.
\end{Lemma}
\begin{proof} We have to show that it induces an equivalence on all fixed points for graph subgroups $L\in \F_{U}^{G,\Sigma_{n\times m}}$. Since $m>0$, the map $\Sigma_n\wr \Sigma_m\to \Sigma_{n\times m}$ is injective and so the wreath product acts freely on both sides and we can apply Lemma \ref{lem:tech1} above for the direct product $(G\times \Sigma_{m\times n})\times (\Sigma_n\wr \Sigma_m)$ with normal subgroup $\Sigma_n\wr \Sigma_m$. So we have to show that the map $(E\mathcal{F}_{\U}^{G,\Sigma_n})_+\wedge (\Sigma_{n\times m})_+\wedge A^{\wedge n} \to (\Sigma_{n\times m})_+\wedge A^{\wedge n}$ induces an equivalence on graph subgroups in $H\times \Sigma_{n\times m}\times \Sigma_n\wr \Sigma_m$ (where $H$ is the projection of $L$ to $G$) for homomorphisms $(\alpha,\beta)$ of which $\underline{n+m}_{\alpha}$ allows an $H$-embedding into $\U$. The wreath product $\Sigma_n\wr \Sigma_m$ acts on $E\F_{\U}^{G,\Sigma_n}$ through the projection to $\Sigma_n$, and hence the fixed points of such a graph subgroup on 
$E\F_{\U}^{G,\Sigma_n}$ are contractible if $\beta$ followed by this projection to $\Sigma_n$ defines an $H$-set structure on $\underline{n}$ which embeds into $\U$, and empty otherwise. Hence, we have to show that if this $H$-action on $\underline{n}$ does not embed into $\U$, then the fixed points on the right hand side are also empty. We claim that already the fixed points on $\Sigma_{n\times m}$ are empty. As in the proofs before, a bijection in $\Sigma_{n\times m}$ is fixed under this graph subgroup if and only if it is an isomorphism from the $H$-set structure on $\underline{n\times m}$ via $\beta$ to the one via $\alpha$. Since we assumed that $\alpha$ allows an embedding into $\U$, we thus have to show that if $\underline{n}_{pr\circ \beta}$ is not $H$-embeddable, then neither is $\underline{n\times m}_{\beta}$. We decompose the $H$-action on $\Sigma_n$ into orbits $N_1,\hdots,N_k$, choose elements $n_i\in N_i$ and let $H_i\leq H$ be the stabilizer of $n_i$. Then there is an $H$-isomorphism
\[ (\underline{n\times m})_{\beta}\cong \bigsqcup_{i=1,\hdots,k} H\ltimes_{H_i} ({n_i}\times m). \]
By assumption, for one of the $H_i$'s the $H$-set $H/H_i$ does not embed into $\U$. So, since the isotropy of $\U$ is closed under supergroups, $H\ltimes_{H_i} ({n_i}\times m)$ also does not $H$-embed into $\U$ (the isotropy of an element can only be smaller) and hence neither does $(\underline{n\times m})_{\beta}$, and so we are done.
\end{proof}

\textbf{Acknowledgments}. This paper grew out of my master's thesis at the University of Bonn. I would like to thank my adviser Stefan Schwede for suggesting the topic and for his ongoing support during the thesis writing and after. I also thank Tom Bachmann, Jeremiah Heller, Haynes Miller, Irakli Patchkoria, Lu\'is Alexandre Pereira, Steffen Sagave, Wolfgang Steimle and Christian Wimmer for various discussions on the subject and helpful suggestions.

This research was supported by the Deutsche Forschungsgemeinschaft Graduiertenkolleg 1150 ``Homotopy and Cohomology'' and the Danish National Research Foundation through the Centre for Symmetry and Deformation (DNRF92).

%\bibliographystyle{alpha}
%\bibliography{literature}

\bigskip
\footnotesize

  \textsc{Department of Mathematical Sciences, University of Copenhagen, Denmark}\par\nopagebreak
  \textit{E-mail address}: \texttt{hausmann@math.ku.dk}
  
\end{document}